\documentclass[twoside,reqno,12pt]{amsart}
\usepackage[left=1in,right=1in,top=1in,bottom=1in]{geometry}

\usepackage[hidelinks]{hyperref}
\usepackage{amsfonts}
\usepackage{amssymb,mathtools}
\usepackage{color,xcolor}
\usepackage{comment}
\newcommand{\HOX}[1]{\marginpar{\footnotesize sharp1}}

\newtheorem{thm}{Theorem}[section]

\newtheorem{lem}[thm]{Lemma}
\newtheorem{prop}[thm]{Proposition}

\newtheorem{rem}[thm]{Remark}

\theoremstyle{definition}

\numberwithin{equation}{section}

\newcommand{\thmref}[1]{Theorem~\ref{sharp1}}
\newcommand{\secref}[1]{\S\ref{sharp1}}
\newcommand{\lemref}[1]{Lemma~\ref{sharp1}}

\newcommand{\exit}{\tau_{\mathrm{exit}}}

\newcommand{\C}{\mathbb{C}}

\renewcommand{\div}{\operatorname{div}}

\newcommand{\n}[1]{\langle #1 \rangle}
\newcommand{\N}{\mathbb{N}}

\newcommand{\R}{\mathbb{R}}

\renewcommand{\t}[1]{\tilde{sharp1}}

\newcommand{\tr}{\hbox{tr}\,}

\newcommand{\dbtilde}[1]{\tilde{\tilde{sharp1}}}

\def\tilde{\widetilde}
\def \bfo {\begin {eqnarray*} }
\def \efo {\end {eqnarray*} }
\def \ba {\begin {eqnarray*} }
\def \ea {\end {eqnarray*} }
\def \beq {\begin {eqnarray}}
\def \eeq {\end {eqnarray}}
\def \supp {\hbox{supp }}
\def \diam {\hbox{diam }}

\def \det {\hbox{det}}

\def \p {\partial}

\def\tilde{\widetilde}
\def \bfo {\begin {eqnarray*} }
\def \efo {\end {eqnarray*} }
\def \ba {\begin {eqnarray*} }
\def \ea {\end {eqnarray*} }
\def \beq {\begin {eqnarray}}
\def \eeq {\end {eqnarray}}
\def \supp {\hbox{supp }}
\def \diam {\hbox{diam }}

\def \det {\hbox{det}}

\def \p {\partial}

\usepackage{color,comment}

\newcommand{\h}[1]{\{ sharp1 \}}
\newcommand{\LA}{\langle}
\newcommand{\RA}{\rangle}
\newcommand{\LC}{\left(}
\newcommand{\RC}{\right)}

\usepackage{mathtools}

\begin{document}
\title[Stability of an inverse spectral problem]
{H\"older stability of an inverse spectral problem for the magnetic Schr\"odinger operator on a simple manifold}

% Author name(s)
\author[Liu]{Boya Liu}
\address{B. Liu, Department of Mathematics\\
North Dakota State University, Fargo\\ 
ND 58102, USA}
\email{boya.liu@ndsu.edu}

\author[Quan]{Hadrian Quan}
\address{H. Quan, Department of Mathematics, University of California, Santa Cruz, CA 95064, USA}
\email{hquan1@ucsc.edu}

\author[Saksala]{Teemu Saksala}
\address{T. Saksala, Department of Mathematics\\
North Carolina State University, Raleigh\\ 
NC 27695, USA}
\email{tssaksal@ncsu.edu}

\author[Yan]{Lili Yan}
\address{L. Yan, School of Mathematics, University of Minnesota, Minneapolis, MN 55455, USA}
\email{lyan@umn.edu}

\begin{abstract}
We show that on a simple Riemannian manifold, the electric potential and the solenoidal part of the magnetic potential appearing in the magnetic Schr\"odinger operator can be recovered Hölder stably from the boundary spectral data. This data contains the eigenvalues and the Neumann traces of the corresponding sequence of Dirichlet eigenfunctions of the operator. Our proof contains two parts, which we present in the reverse order. (1) We show that the boundary spectral data  can be stably obtained from the Dirichlet-to-Neumann map associated with the respective initial boundary value problem for a hyperbolic equation, whose leading order terms are \textit{a priori} known. (2) We construct geometric optics solutions to the hyperbolic equation, which reduce the stable recovery of the lower order terms to the stable inversion of the geodesic ray transform of one-forms and functions. 
\end{abstract}

\maketitle

\section{Introduction and Statement of Results}
In this paper we show that, on a simple Riemannian manifold, the electric potential and the solenoidal part of the magnetic potential of the magnetic Schrödinger operator can be recovered Hölder stably from the boundary spectral data. In other words, we solve a quantitative version of the Gel’fand inverse boundary spectral problem \cite{Gelfand}: ``Determine the coefficients of an elliptic second order partial differential operator  from the boundary spectral data''. We do not address the more difficult question of recovering the principal terms, which corresponds to the  reconstruction of a Riemannian manifold with boundary and its metric from this data.

Our proof consists of two parts, which we present in the reverse order. We show that the boundary spectral data can be stably obtained from the hyperbolic Dirichlet-to-Neumann map associated with the respective initial boundary value problem for a hyperbolic partial differential operator, whose coefficients are time-independent, and the leading order terms are \textit{a priori} known, while the lower order terms are to be determined. To stably recover the unknown lower order terms from the hyperbolic Dirichlet-to-Neumann map, we shall construct geometric optics solutions to the hyperbolic equation. Via these solutions, we reduce the stable recovery of the lower order terms to the stability of the normal operator of the geodesic ray transform of one-forms and functions. 

The present geometric framework of a simple manifold was utilized in  \cite{Stefanov_Uhlmann_xray_transform} to guarantee that the normal operator of the geodesic ray transform poses the required stability estimate. However, our reduction procedure from boundary spectral data to the hyperbolic Dirichlet-to-Neumann map is valid for all smooth compact manifolds. Furthermore, we believe that the aforementioned construction of geometric optics solutions can also be implemented on non-simple manifolds with boundary. Therefore, we expect that the first main result (Theorem \ref{thm:spectral_problem}) of this paper, which is novel even in the case of a Euclidean disk, can be extended to those geometries whose geodesic ray transform is stably invertible for functions and one-forms. We leave the extension of the geometric framework to a future work.

The existence of the magnetic potential leads to significant technical difficulties when we connect the boundary spectral data to the hyperbolic Dirichlet-to-Neumann map. The main contributions of this paper lie in this step, which is presented in Section \ref{sec:proof_spectral}. To elaborate, we shall connect the aforementioned sets of datum via an elliptic Dirichlet-to-Neumann map. Although similar techniques were also utilized in earlier works such as \cite{Bellassoued_Ferreira,Ben_Joud_spectral}, we need to define the elliptic Dirichlet-to-Neumann map in a Sobolev space of lower regularity, particularly with a  negative index. These changes also require us to define the hyperbolic Dirichlet-to-Neumann map in lower regularities.

\subsection{Problem setting and main results}

Let $(M,g)$ be a smooth compact Riemannian manifold of dimension $n\ge 2$ with a smooth boundary $\p M$. Throughout this paper we shall write $\Delta_g$ for the Laplace-Beltrami operator associated with the metric $g$. We use the notation $C^{\infty}(M, T^\ast M)$ for the collection of smooth complex-valued one-forms. We also denote by $d$ the exterior derivative 
of the smooth manifold $M$.
For each real-valued one-form $A \in C^{\infty}(M, T^\ast M)$, we introduce the respective magnetic differential
\[
d_A \colon C^\infty(M) \to C^\infty(M, T^\ast M), \quad d_A(f)=df+ifA.
\]
From a geometric perspective, $d_A$ can be interpreted as a connection on the trivial line bundle $M\times \C$ over $M$. 
The formal $L^2$-adjoint of $d_A$
is given by $d_A^\ast \alpha = d^\ast \alpha - i\n{A,\alpha}_g$, where $d^\ast$ maps one-forms to functions according to the formula  
\begin{equation}
\label{eq:dstar}
d^* \alpha
=
\div_g(\alpha^\sharp)=
-\frac{1}{\sqrt{\det g}}\p_{x^j} \big(g^{jk}\sqrt{\det g} \; \alpha_k\big).
\end{equation}
Here the notation $\sharp$ stands for the musical isomorphism taking co-vectors to vectors, $\div_g$ is the divergence of $g$, and $g^{ik}$ are components of the inverse metric in local coordinates $(x^1,x^2,\ldots,x^n)$. 
We note that $d_A$ is independent of the metric $g$ while $d^\ast_A$ is not.

After these preparations, we proceed to define the magnetic Laplacian as
\begin{equation}
\label{eq:def_magnetic_laplacian}
\begin{aligned}
-\Delta_{g,A}v:
&= (d_A^\ast d_A)v
% &= -\Delta_g-2i\n{A,d}_g+i(d^* Au)+\n{A,A}_g,
= 
-\Delta_gv -2i\langle A,dv\rangle_g + (id^*A + \n{A,A}_g)v, \quad \text{for } v \in C^\infty(M),
\end{aligned}
\end{equation}
where  $\n{\cdot, \cdot}_g$ is the Riemannian inner product of $(M,g)$.

We now proceed to state the regularity conditions that we need to impose on the lower order terms we aim to recover. For each $k \in \{0,1,2,\ldots\}$, the Sobolev space $H^k(M)$ is the completion of $C^\infty(M)$, equipped with the norm
\[
\|f\|_{H^k(M)}=\|f\|_{L^2(M)}^2 + \sum_{k=1}^{n} \|\nabla_g^k f\|_{L^2(M, T^kM)},
\]
where $\nabla_g$ is the Levi-Civita connection of the metric $g$, and
the norm $\|\cdot\|_{L^2(M)}$ is induced by the natural $L^2$-inner products 
\[
(f,g)_{L^2(M)}: = \int_M f\overline{g}dV_g, 
\quad \text{ and } \quad
\langle \alpha,\beta\rangle_{L^2(M,T^*M)} = \int_M 
\langle \alpha, \bar \beta\rangle_g
dV_g
\]
for functions and tensor fields, respectively.
For each $k \in \{0,1,2,\ldots\}$, we also define covector-valued Sobolev spaces $H^k(M,T^\ast M)$ analogously. We also note that, due to compactness, the one-form  $A$ is in $H^k(M, T^\ast M)$ if and only if $A=a_j(x)dx^j$ in any local coordinates $(x^1,x^2,\ldots,x^n)$, and each component function $a_j$ is in  $H^k(M)$. Thus the following formula leads to an equivalent definition of Sobolev norms 
\[
\|A\|_{H^k(M,T^\ast M)}=\sum_{j=1}^{n}\|a_j\|_{H^k(M)}.
\]

For each $N >0$ and $k \in \N$, we define the admissible set for  magnetic potentials
\[
\mathcal{A}(k,N)=\{A\in C^\infty(M,T^\ast M): \text{ $A$ is real-valued, } \|A\|_{H^k(M,T^\ast M)}\le N\},
\]
%\boya{We might need $k$ to be large enough, say, $k>\frac{n}{2}$, to apply Sobolev embedding. Is it sufficient to only put the norm for $A^s$?}
and for electric potentials
\[
\mathcal{Q}(N)=\{q\in C^{\infty}(M): \text{ $q$ is real-valued, }  \|q\|_{H^1(M)}\le N\}.
\]

Let us  recall that it was shown in \cite[Theorem 3.3.2]{Sharafutdinov} that the covector field $A \in H^k(M, T^\ast M)$ has the following unique Helmholtz decomposition 
\[
A=A^{\mathrm{sol}}+d \varphi, \quad d^\ast A^{\mathrm{sol}}=0, \quad \varphi|_{\p M}=0,
\]
where the covector field $A^{\mathrm{sol}}\in H^k(M, T^\ast M)$ and the function $\varphi\in H^{k+1}(M, \C)$ are called the \textit{solenoidal} and the \textit{potential} parts of $A$, respectively.

In this paper we assume that the manifold $(M,g)$ is simple, which means that $(M,g)$ is a smooth compact Riemannian manifold with a smooth strictly convex boundary, and that there is a unique smoothly varying distance minimizing geodesic between any two points of $M$. In particular, all simple Riemannian manifolds are diffeomorphic to Euclidean balls, hence they are simply connected. Furthermore, simple manifolds have no trapped geodesics nor conjugate points, and for each $p \in M$ the respective distance function $d(p,\cdot)$ is smooth in $M$. 

\subsection{Inverse spectral problem}
\label{ssec:spectral_problem}
For each $N>0$, $p\in \N$, $A \in \mathcal{A}(p,N)$, and  $q \in Q(N)$, we 
consider the magnetic Schr\"odinger operator
\begin{equation}
\label{eq:def_mag_Schro}
\mathcal{E}_{g,A,q} := -\Delta_{g,A}+q,
\end{equation}
with the domain $\mathcal{D}(\mathcal{E}_{g,A,q})=H_0^1(M) \cap H^2(M)$. 
Since magnetic Schr\"odinger operators are self-adjoint, the spectrum of $\mathcal{E}_{g,A,q}$ consists of an increasing sequence of real eigenvalues, counted according to their multiplicities
\[
-\infty<\lambda_{1,A,q} \le \lambda_{2,A,q} \le \cdots \le \lambda_{k,A,q} \to \infty, \quad k \to \infty.
\]
Here the corresponding eigenfunctions are denoted by $\varphi_{k,A,q}$, and this sequence forms an orthonormal basis of $L^2(M)$. 
Since $\varphi_{k,A,q}$ satisfies the  boundary value problem
\[
\begin{cases}
\mathcal{E}_{g,A,q}\varphi = \lambda_{k,A,q}\varphi   &\text{ in }  M,
\\
\varphi= 0 & \text{ on }   \p M, 
\end{cases}
\]
by  elliptic regularity, we have that 
\[
\|\varphi_{k,A,q}\|_{H^2(M)}
\le 
C (1+|\lambda_{k,A,q}|)\|\varphi_{k,A,q}\|_{L^2(M)} = C(1+|\lambda_{k,A,q}|).
\]
Then by interpolation,  we have for any $\sigma\in[0,2]$  that
\[
\|\varphi_{k,A,q}\|_{H^\sigma(M)}
\le C(1+|\lambda_{k,A,q}|)^\frac{\sigma}{2}.
\]
Moreover, if we denote $\psi_{k,A,q} := \LA d_A\varphi_{k,A,q}, \nu \RA_g|_{\p M}$, where $\nu$ is the unit outward normal of $\p M$, an application of the trace theorem gives us the inequality
\begin{equation}
\label{eq:est_normal_eigenfunction_bdy}
\|\psi_{k,A,q}\|_{H^{1/2}(\p M)}
\le
C(1+|\lambda_{k, A, q}|).
\end{equation}
Here the Sobolev space $H^{1/2}(\p M)$ is defined as
\[
H^{1/2}(\p M) = \{v\in L^2(\p M):  \text{There exists a function } u\in H^1(M) \text{ such that } \tr (u)=v\},
\]
equipped with the norm
\[
\|v\|_{H^{1/2}(\p M)}
=
\inf \{\|u\|_{H^1(M)}: u \in H^1(M) \text{ and } \tr (u)=v \}.
\]
On the other hand, let us also recall that the eigenvalues of $\mathcal{E}_{g,A,q}$ satisfy the following Weyl asymptotics
\begin{equation}
\label{eq_wely}
C^{-1}k^{\frac{2}{n}}\le (1+|\lambda_{k,A,q}|) \le Ck^{\frac{2}{n}},
\quad \text{ for all } k \in \N,
\end{equation}
where $C$ is a constant uniform for $A\in \mathcal{A}(p,N)$ and $q\in \mathcal{Q}(N)$, see \cite[Lemma A.1]{five_author_spectral}. 

Let $\ell^1$ be the usual Banach space of real-valued sequences such that the corresponding series is absolutely convergent.
We take $m>0$ to be a constant such that $\frac{n}{2}+1< m\le n+1$,  define a sequence of weights $(\omega_k)_{k\in\N}$ by $\omega_k:=k^{-\frac{2m}{n}}$, and use them to introduce the following weighted Banach spaces
\[
\ell_{n,m}^1\big(H^{1/2}(\p M)\big)
=
\left\{h=
(h_k)_{k\in \N}: h_k\in H^{1/2}(\p M),\, (\omega_k \|h_k\|_{H^{1/2}(\p M)})_{k\in \N}\in \ell^1
\right\},
\]
and
\[
\ell_{n,m}^1(\C)
=
\left\{y=
(y_k)_{k\in \N}: y_k \in \C,\,  (\omega_k|y_k|)_{k\in \N} \in \ell^1
\right\}.
\]
The natural norms on these spaces are given by 
\begin{equation}
\label{eq:weighted_ell_1_norms}
\|h\|_{\ell_{n,m}^1\left(H^{1/2}(\p M)\right)}
=
\sum_{k\ge1} \omega_k \|h_k\|_{H^{1/2}(\p M)}
\quad \text{ and } \quad 
\|y\|_{\ell_{n,m}^1(\C)}
=
\sum_{k\ge1} \omega_k|y_k|.
\end{equation}

For $\ell \in \{1,\,2\}$, we let 
$\sigma(\mathcal{E}_{g,A_\ell,q_\ell})= \{\lambda_{k,A_\ell,q_\ell}:\,k=1,2,\cdots\}$ be the spectrum of the magnetic Schr\"odinger operator $\mathcal{E}_{g,A_\ell,q_\ell}$ given by \eqref{eq:def_mag_Schro} with $A = A_\ell$ and $q = q_\ell$, and let $\varphi_{k,A_\ell,q_\ell}$ be the normalized eigenfunction  corresponding to $\lambda_{k,A_\ell,q_\ell}$. Let us denote $\psi_{k,A_\ell,q_\ell}: = \LA d_A\varphi_{k,A_\ell,q_\ell},\nu \RA|_{\p M}$. 
For simplicity, in what follows let us write
\[
\lambda_{\ell,k}:= \lambda_{k,A_\ell,q_\ell}, \quad \varphi_{\ell,k}: = \varphi_{k,A_\ell,q_\ell}, \quad
\text{and}\quad \psi_{\ell,k}:=\psi_{k,A_\ell,q_\ell}.
\]

The first main result of this paper is as follows.
\begin{thm}
\label{thm:spectral_problem}
Let $(M,g)$ be a simple Riemannian manifold of dimension $n\ge 2$ with smooth boundary $\p M$, 
$m\in (\frac{n}{2}+1,n+1)$, and $N\ge 0$. Let $A_1, A_2\in \mathcal{A}(\lceil \frac{n}{2}\rceil+1, N)$ and $q_1,q_2\in \mathcal{Q}(N)$ be such that $A_1=A_2$ and $q_1=q_2$ on $\p M$. 
Then there exist constants $C>0$ and $\theta, \sigma_2 \in (0,1)$, all of which depend only on $(M,g), n,m$ and $N$, such that 
\begin{equation}
\label{eq:est_spectral}
\|A_1^s-A^s_2\|_{L^2(M)}+\|q_1-q_2\|_{L^2(M)}\le C(\delta+\delta^\theta)^{\sigma_2}.
\end{equation}
Here
\begin{equation}
\label{eq:delta}
\delta := \|\psi_1-\psi_2\|_{\ell^1_{n,m}(H^{1/2}(\p M))} + \|\lambda_1-\lambda_2\|_{\ell^1_{n,m}(\mathbb{C})},
\end{equation}
where $\psi_\ell=(\psi_{\ell,k})_{k\in \N}$ and 
$\lambda_\ell=(\lambda_{\ell,k})_{k\in \N}$,  $\ell = 1,\, 2$. 
\end{thm}

To the best of our knowledge, Theorem \ref{thm:spectral_problem} is the first to establish a stable recovery of the solenoidal part of an unknown magnetic potential $A$ from the boundary spectral data of the magnetic Schr\"odinger operator. That is, we do not need to assume that $A_1 = A_2$ globally or near the boundary in Theorem \ref{thm:spectral_problem}. 

We end this subsection by noting that the conclusion of Theorem \ref{thm:spectral_problem} is natural, because the boundary spectral data $(\lambda_{A,q}, \psi_{A,q})$ do not uniquely determine the magnetic potential $A$ due to the natural gauge invariance
\begin{equation}
\label{eq:gauge}
e^{-i\varphi}\mathcal{E}_{g,A,q}e^{i\varphi} =\mathcal{E}_{g,A+d\varphi,q} 
\quad \text{ and } \quad
(\lambda_{A,q}, \psi_{A,q})
=
(\lambda_{A+d \varphi,q}, \psi_{A+d \varphi,q}),
\end{equation}
which, by a direct computation, holds for all real-valued functions $\varphi \in C^\infty(M)$ such that $\varphi|_{\p M}=0$.

\subsection{Inverse problem for a symmetric hyperbolic operator}
Given $N>0, \: k\in \N, \: A \in \mathcal{A}(k,N)$ and $q \in Q(N)$, we consider the linear hyperbolic partial differential operator with time-independent coefficients
\begin{equation}
\label{eq:def_operator}
\mathcal{H}_{g,A,q}=\p_t^2-\Delta_{g,A}+q,
\end{equation}
where $\Delta_{g,A}$ is the magnetic Laplacian given by the formula \eqref{eq:def_magnetic_laplacian}. 

For each $T>0$, throughout this paper we shall denote by $Q$ and $\Sigma$ the spacetime $Q=(0,T)\times M$ and its lateral boundary $\Sigma=(0,T)\times \p M$, respectively. Let us consider the following initial boundary value problem
\begin{equation}
\label{eq:ibvp_equivalent}
\begin{cases}
\mathcal{H}_{g,A,q}u=0   & \text{ in } Q,
\\
u(0,\cdot)=0, \quad \p_tu(0,\cdot)=0   &\text{ in } M,
\\
u=f  & \text{ on } \Sigma,
\end{cases}
\end{equation}
with $f\in L^2_0(\Sigma) := \{f \in L^2(\Sigma):\: \supp{f}\subset \subset \Sigma\}$. It then follows from \cite[Theorem 2.42%and Corollary 2.36
]{KKL_book} and \cite[Theorem 2.3]{Lasiecka_Lions_Triggiani} that \eqref{eq:ibvp_equivalent} has a unique solution $u\in C([0, T]; L^2(M))\cap C^1([0, T]; H^{-1}(M))$ such that $\LA d_Au, \nu \RA_g \in H^{-1}(\Sigma)$.

We next introduce the boundary measurement for the second inverse problem in question, which is associated to the initial boundary value problem \eqref{eq:ibvp_equivalent}. Due to the low regularity assumption of the boundary value $f$, we need to define the hyperbolic Dirichlet-to-Neumann map  $\Lambda_{A,q}:  L^2(\Sigma)\to H^{-1}(\Sigma)$ in a weak sense:
\begin{equation}
\label{eq:def_hyperbolic_DN_map}
\LA\Lambda_{A,q}(f),h\RA_{H^{-1}(\Sigma), H^1(\Sigma)}= ( f,\overline{\LA d_A\varphi, \nu \RA}_g )_{L^2(\Sigma)}, \quad \text{ for every } h\in H^1_0(\Sigma).
\end{equation}
Here the function $\varphi\in C([0,T];H^1(M))\cap C^1([0,T];L^2(M))$ is the unique solution to the problem
\[
\begin{cases}
\mathcal{H}_{g,A,q}\varphi=0   &\text{ in } Q,
\\
\varphi(T,\cdot)=0, \quad \p_t\varphi(T,\cdot)=0  &\text{ in }   M,
\\
\varphi=h &\text{ on } \Sigma,
\end{cases}
\]
with $\LA d_A\varphi, \nu \RA_g\in L^2(\Sigma)$.

In our second main result, we establish a H\"older-type stability estimate for the solenoidal part $A^s$ of the one-form $A$ and the function $q$ in terms of the hyperbolic Dirichlet-to-Neumann map $\Lambda_{A,q}$.

\begin{thm}
\label{thm:main_result_hyperbolic}
Let $T>\mathrm{diam} (M)$. If the hypotheses in Theorem \ref{thm:spectral_problem} are satisfied, then there exist constants $\sigma_1 \in (0,1)$ and $C>0$, both of which only depend on $(M,g), n,T$ and $N$, such that 
the following estimate is valid:
\begin{equation}
\label{eq:est_X_and_q}
\|A^s_1-A^s_2\|_{L^2(M)}+\|q_1-q_2\|_{L^2(M)}\le C\|\Lambda_{A_1,q_1}-\Lambda_{A_2,q_2}\|^{\sigma_1}_{\mathcal{L}(L^2(\Sigma),H^{-1}(\Sigma))}.
\end{equation}
\end{thm}

\subsection{Previous literature}

Starting with  \cite{Ambarzumian_1929}, which proved the unique recovery of a real-valued electric potential appearing in the Schr\"odinger operator $-\Delta+q$ on the open interval $(0,1)$ from partial spectral data, 
the study of inverse spectral problems dates back nearly a century. Given that the literature on this problem is extensive, we restrict our discussion only to the uniqueness and stability results concerning multidimensional inverse spectral problems. For foundational results in the one-dimensional setting, we refer readers to the seminal works   \cite{Borg,Gelfand_Levitan,Levinson}, among others. 

The first multidimensional uniqueness result for inverse spectral problems was established in \cite{Nachman_Sylvester_Uhlmann}, which showed that, for the Schr\"odinger operator, the Dirichlet eigenvalues and the traces of the normal derivatives of the corresponding eigenfunctions uniquely determine the potential $q$ in Euclidean domains. 
The geometric version of this problem is commonly solved via the celebrated Boundary Control method (BC-method), which was first developed in \cite{Belishev} for the acoustic wave equation on $\R^n$ with an isotropic wave speed. The fully geometric version of the BC-method, suitable when the wave speed is given by a Riemannian metric, was introduced   in \cite{Belishev_Kurylev}.
We refer to \cite{KKL_book} for a thorough review of related literature.
Furthermore, it was discovered in \cite{Isozaki} that the uniqueness still holds in the Euclidean setting when a finite number of spectral data are omitted, and a corresponding result on compact manifolds with boundary was obtained ion \cite{Katchalov_Kurylev}. On the other hand, it was established in \cite{Choulli_Stefanov,Isozaki} that the uniqueness for the potential $q$ can still be achieved when the spectral data are asymptotically close. Also, the authors of \cite{Kian_Morancey_Oksanen} showed that the boundary spectral data, with the normal derivatives of the eigenfunctions measured only on an open portion of the boundary, determine the potential $q$ uniquely. Finally, we refer to   \cite{Bellassoued_Kian_Mannoubi_Soccorsi,Choulli_Metidji_Soccorsi,Paivarinta_Serov,Pohjola_spectral} and the references therein for recent progress on the recovery of non-smooth potentials.

Turning our attention to inverse spectral problems for the magnetic Schr\"odinger operator, a uniqueness result was first established in \cite{Katchalov_Kurylev}, which showed that the knowledge of partial boundary spectral data, where finitely many eigenpairs are unknown, uniquely determines the metric tensor, as well as smooth lower order terms $A$ and $q$, up to a gauge similar to the one as in the equation \eqref{eq:gauge}, in  a smooth connected Riemannian manifold. 
In the Euclidean setting, it was proved in \cite{Serov_magnetic_schrodinger} that the full boundary spectral data determine the magnetic field $dA$ and $q$ uniquely for real-valued potentials $A\in W^{1,\infty}(\Omega)$ and $q\in L^\infty(\Omega)$, where $\Omega\subseteq \R^n$, $n \ge 2$, is a bounded domain with smooth boundary. A further result  \cite{Kian_Borg_Levinson} established uniqueness for $dA$ and $q$ from the asymptotic knowledge of the boundary spectral data in the Euclidean case, which was later extended to the setting of simple Riemannian manifolds in \cite{five_author_spectral}.
Additionally, under the assumption that the potentials $A$ and $q$ are known near the boundary, \cite{Bellassoued_Mannoubi} proved a partial data uniqueness result in the Euclidean setting, where normal derivatives of the eigenfunctions are measured on an arbitrary subset of the boundary.

While there exist several stability results for  inverse spectral problems, all of them only provide stability estimates for the electric potential $q$. 
To the best of our knowledge, Theorem \ref{thm:spectral_problem} is the first to provide a stable recovery of an unknown magnetic potential $A$ from the boundary spectral data of the   magnetic Schr\"odinger operator. 
In the absence of the magnetic potential $A$, a H\"older-type stability estimate was derived from boundary spectral data in the Euclidean setting in \cite{Alessandrini_Sylvester}, and this result was subsequently extended  to the setting of simple Riemannian manifolds in  \cite{Bellassoued_Ferreira}. Moreover, assuming that $A$ is known, or in terms of  Theorem \ref{thm:spectral_problem}, $A_1=A_2$ in $M$, a H\"older-type estimate was obtained for $q$ in \cite{Ben_Joud_spectral}.  For partial data results, the authors of  \cite{Choulli_Stefanov} derived a H\"older-type estimate from the boundary spectral data, where finitely many eigenpairs are omitted; in particular, only asymptotic knowledge of the spectral data was required in this result. Furthermore, in a recent work \cite{Choulli_Metidji_Soccorsi}, a H\"older-type stability was obtained for unbounded potentials, where Robin boundary conditions of the eigenfunctions are known.

From now on we discuss hyperbolic inverse problems, focusing specifically on literature concerning stability.
Although the BC-method is extremely flexible in providing uniqueness results for inverse problems for hyperbolic operators, this method does not require any geometric assumptions, (see \cite{Katchalov_Kurylev, KKL_book, LaOk} for the classical problem setting where  measurements are performed on the entire boundary, and \cite{HELIN2018132, lassas2024disjoint,saksala2025inverse} for the formulation of the problem where the measurements are taken only on open subsets of a manifold without boundary), it is known that the BC-method is very unstable due to its dependence on Tataru's unique continuation principle \cite{Tataru}. For instance,  \cite{Bosi2022Reconstruction, burago2020quantitative} provided log-log-type stability results. However, the counterexamples of \cite{mandache2001exponential} for an equivalent inverse problem show that the stability result cannot be better than logarithmic. 
If one is willing to relax geometric assumptions, or focus on the cases where the leading order terms are known, there are also other techniques
%, in addittion to the BC-method, 
that lead to improved stability results, as well as uniqueness results, when one aims to recover time-dependent coefficients. We refer readers to  \cite{Liu_Saksala_Yan,Liu_Saksala_Yan_potential} to see detailed accounts of uniqueness results for time-dependent hyperbolic equations.  

We now proceed to survey stability results for time-independent coefficients of hyperbolic operators. When the operator \eqref{eq:def_operator} does not contain the first order perturbation $A$, a H\"older-type stability estimate for the potential $q$ was established from the hyperbolic Dirichlet-to-Neumann map in the Euclidean setting in \cite{Sun_hyperbolic_stability}, and in the setting of simple  manifolds in \cite{Bellassoued_Ferreira}. On the other hand, the authors of \cite{Bellassoued_Rezig} also assumed that $A=0$, but incorporated a damping term $a(x)\p_t$, and obtained H\"older-type estimates for both $a$ and $q$ on simple manifolds. Furthermore, H\"older-type estimates for both $dA$ and $q$  were proved in \cite{Bellassoued_Aicha} in the Euclidean setting. We would like to point out that proofs of \cite{Bellassoued_Ferreira,Bellassoued_Rezig} utilized geometric optics solutions of hyperbolic equations to reduce the stability of the lower order terms to the stable inversion of the geodesic ray transform, which was originally established in \cite{Stefanov_Uhlmann_xray_transform}. We shall follow the same strategy in this paper. Meanwhile, the authors of a very recent work \cite{Filippas_Oksanen} utilized the BC-method to establish stability for the zeroth order perturbation of the wave operator in the Euclidean setting.  Moving beyond the realm of simple manifolds,   a stability estimate was obtained in \cite{arnaiz2022stability} for the potential $q$ on Riemannian manifolds with strictly convex boundary and no conjugate points, but may have a hyperbolic trapped set of geodesics. The proof in this result relies on a reduction to the stable inversion of the geodesic ray transform established in \cite{Guillarmou}. 

It is also known that one can stably recover certain simple Riemannian metrics from the Dirichlet-to-Neumann map associated to the problem \eqref{eq:ibvp_equivalent}. For instance, a conditional H\"older-type stability estimate was proven in \cite{Stefanov_Uhlmann_98} for metrics  that are sufficiently close to the Euclidean metric, and for  generic simple metrics in \cite{Stefanov_Uhlmann}. Due to the diffeomorphism invariance of the problem \eqref{eq:ibvp_equivalent}, the metric tensor can only be recovered from the Dirichlet-to-Neumann map up to a boundary fixing diffeomorphism, see for instance \cite{Stefanov_Uhlmann_98} for details. Instead of measuring the full  Dirichlet-to-Neumann map, \cite{Bellassoued_metric} shows that one can recover a simple metric log-stably when measuring a partial hyperbolic Dirichlet-to-Neumann map, provided that the metric is Euclidean near the boundary. In particular, the Neumann data is measured only in an open subset of the boundary.  Building on the results of \cite{Bellassoued_Ferreira,Stefanov_Uhlmann_98}, assuming that the metric is close to a generic simple metric, and the covector fields are also close, the author of \cite{Montalto}  simultaneously established conditional H\"older-type estimates for the metric, as well as for the first and the zeroth order term of the hyperbolic operator \eqref{eq:def_operator}, from the Dirichlet-to-Neumann map. 

Lastly, there exists an abundance of stability results for time-dependent coefficients of hyperbolic equations in the Euclidean setting, see \cite{Bellassoued_Aicha_td_17,Bellassoud_Aicha_hyperbolic_td_19,Bellassoued_Rassas_time_dependent_hyperbolic,Aicha_hyperbolic_td,Kian_stability,Senapati_time_dependent_hyperbolic} and the references therein. In these works, geometric optics solutions were utilized to reduce the stable recovery of coefficients to the stability of the light ray transform. In particular, \cite{salazar2014stability}
shows that the light ray transform of the difference of two time-dependent potentials can be estimated by a distance between the respective Dirichlet-to-Neumann operators. By utilizing the Fourier Slice Theorem, this estimate implies the stable recovery of time-dependent electric potential $q$.  Due to the lack of Fourier transform, analogous stability results are scarcely known on Riemannian manifolds. To the best of our knowledge, the only results in this direction were given in \cite{Waters_stability}, which established H\"older-type estimates for the light ray transform of time-dependent electric  potentials $q$.  
They require specialized tools for analyzing the singularities near the light-like cone, which are not fully developed yet in the geodesic case, see for instance \cite{Greenleaf_Uhlmann_1989,Greenleaf_Uhlmann_1990,Stefanov_Yang}.

\subsection{Outline for the proof of the main results}
\label{subsec:main_idea}

The proof of Theorem \ref{thm:spectral_problem} involves two main steps, which are presented in reverse order. In Section \ref{sec:proof_spectral} we show that  the boundary spectral data can be stably recovered from the hyperbolic Dirichlet-to-Neumann map associated to the respective initial boundary value problem for the hyperbolic operator \eqref{eq:def_operator}. In the second step, which is presented in Section \ref{sec:proof_hyperbolic}, we establish a H\"older-type stability for the electrical potential and the solenoidal part of the magnetic potential from the Dirichlet-to-Neumann map $\Lambda_{A,q}$ defined in \eqref{eq:def_hyperbolic_DN_map}. This stability result is achieved via a construction of geometric optics solutions to the hyperbolic equation, which we shall discuss in detail in Section \ref{sec:GO_solution}. 

The reduction from boundary spectral data to the hyperbolic Dirichlet-to-Neumann map is carried out through a series of lemmas. First, in Lemma \ref{lem:DN_map_decomposition} we establish the connection between the boundary spectral data and the family of  Dirichlet-to-Neumann maps $\Pi_{A,q}(z)$ given by the formula \eqref{eq:def_elliptic_DNmap}, which are associated with the elliptic operator $\mathcal{E}_{g,A,q}-z$, where $z \in \R$ lies in the resolvent set of $\mathcal{E}_{g,A,q}$. Subsequently, in Lemma \ref{lem_DN_E_H} we provide an explicit formula that relates $\Pi_{A,q}(z)$ to the hyperbolic Dirichlet-to-Neumann map $\Lambda_{A,q}$. Finally, in order to establish an upper bound for $\Lambda_{A,q}$ from  the norm of the  spectral data, we need to estimate the functions given in the formula \eqref{eq:two_DNmap}, which is accomplished in Propositions \ref{prop_DN_E_difference} and \ref{prop:estimate_P0_R}.

Although a similar strategy was  utilized in \cite{Alessandrini_Sylvester,Ben_Joud_spectral,Choulli_Stefanov} to establish stability estimates for the electric potential from boundary spectral data, our proof requires significant modifications due to the presence of the magnetic potential. In the  aforementioned works, it was sufficient to consider the elliptic Dirichlet-to-Neumann map $\Pi_{A,q}(z):H^{3/2}(\p M) \to H^s(\p M)$,  $0 <s <\frac{1}{2}$. However, in this paper the presence of the magnetic potential changes the exponents of $|z|$ in Proposition \ref{prop_DN_E_difference} from $-j-\frac{s}{2}-\frac{1}{4}$ in previous results to  $-j-\frac{s}{2}+\frac{1}{4}$, which introduces a major challenge to our proof. At the end of Section \ref{sec:proof_spectral}, we shall obtain  an upper bound for the hyperbolic Dirichlet-to-Neumann map by seeking the minimum value of a function involving powers of $|z|$. It requires us to redefine $\Pi_{A,q}(z)$ as an operator from $H^{1/2}(\p M)$ to $H^s(\p M)$, $-1<s<-\frac{1}{2}$. We need to verify that each result in Section \ref{sec:proof_spectral} still holds in the Sobolev space of lower regularity. Furthermore, this change also requires us to define the hyperbolic Dirichlet-to-Neumann map $\Lambda_{A,q}:L^2(\Sigma)\to H^{-1}(\Sigma)$. In earlier works, the Dirichlet-to-Neumann map was defined from $H^1$-smooth Dirichlet data to $L^2$-smooth Neumann data.

Let us next briefly describe the main ideas for the proof of Theorem \ref{thm:main_result_hyperbolic}. The first main component in the proof is the derivation integral identity and its estimate, which is given in Lemma \ref{lem:int_estimate} for the   magnetic potential, and in Lemma \ref{lem:est_int_id_q} for the  electric potential. These lemmas relate the hyperbolic Dirichlet-to-Neumann map to the difference of potentials.

The second main ingredient of the proof is the construction of geometric optics (GO) solutions to the equation $\mathcal{H}_{g,A,q}u=0$ of the form
\[
u(t,x;h)=e^{\frac{i(\psi(x)-t)}{h}} \alpha(t,x)\beta_A(t,x)+r(t,x;h),
\]
where $0<h\ll1$ is a semiclassical parameter, $\psi(x)$ is a phase function that solves an eikonal equation, and the amplitude terms $\alpha$ and $\beta_A$ solve transport equations. Finally, $r$ is a correction term. These functions are constructed in geodesic polar coordinates. Also, with the help of two amplitude terms, we can first isolate a term that will lead to the geodesic ray transform of the magnetic potential. The second amplitude term contains a freely chosen function on the inward-pointing bundle of the boundary and helps us to construct the normal operator of the geodesic ray transform.   

To conclude the proof, we substitute the GO solutions into the aforementioned integral estimates and obtain an upper bound of an integral involving the geodesic ray transform of the difference of magnetic and electric potentials, which are given in Lemma \ref{lem:DN_map_ray_transform} and Lemma \ref{lem:DN_map_ray_transform_q}, respectively. Then we utilize the stability for geodesic ray transforms of one-forms and functions established in \cite{Stefanov_Uhlmann_xray_transform}
to complete the proof of Theorem \ref{thm:main_result_hyperbolic}.

This paper is organized as follows. We first recall some properties of the geodesic ray transforms on simple manifolds in Section \ref{sec:xray_transform}. Then we move to construct geometric optics solutions to the hyperbolic equation in Section \ref{sec:GO_solution}, and use them to prove Theorem \ref{thm:main_result_hyperbolic} in Section \ref{sec:proof_hyperbolic}. Finally, we prove Theorem \ref{thm:spectral_problem} in Section \ref{sec:proof_spectral}.

\subsection*{Acknowledgments}
T.S. was supported by the National Science Foundation grant DMS-2204997.
We would like to thank Katya Krupchyk, Lauri Oksanen, and Mikko Salo for their discussions and suggestions.

\section{Geodesic Ray Transforms}
\label{sec:xray_transform}
In this short section, we recall some important properties of the geodesic ray transforms for one-forms and functions on simple manifolds.

\subsection{Geodesic ray transforms of one-forms}
\label{subsec:transform_1form}
Let $(M,g)$ be a simple manifold. Then for any unit vector $(x,\xi)\in SM$, there exists a unique maximal geodesic $\gamma_{x,\xi}(\tau):[0,\exit(x,\xi)]\to M$, which is given by the initial conditions $\gamma_{x,\xi}(0)=x$ and $\dot{\gamma}_{x,\xi}(0)=\xi$. The geodesic flow $\Phi_\tau: SM \to SM$ of $g$ is given as follows:
\[
\Phi_\tau(x,\xi)=(\gamma_{x,\xi}(\tau),\dot{\gamma}_{x,\xi}(\tau)), \quad \tau\in [0,\exit(x,\xi)].                  
\]  

For each smooth one-form $A$, the \textit{smooth symbol function} $\sigma_A\in C^\infty(SM)$ is given by
\[
\sigma_A(x,\xi)
% =\sum_{j=1}^{n}a_j(x)\xi^j=\n{A^\sharp(x), \xi}_g, 
=A(x,\xi)\quad (x,\xi)\in SM.
\]
Then we use the notation
\[
\p_+SM=\{(x,\xi) \in SM: x\in \p M, \: \n{\xi,\nu(x)}_g\le0 \},
\]
for the bundle of inward-pointing vectors at $\p M$.

We define the geodesic ray transform for one-forms  $\mathcal{I}_1:C^\infty(M,T^\ast M)\to C^\infty(\p_+SM)$ by the formula
\begin{equation}
\label{eq:def_ray_transform_1form}
\mathcal{I}_1(A)(x,\xi)
=
\int_{0}^{\exit(x,\xi)}\sigma_A(\Phi_\tau(x,\xi))d\tau, \quad (x,\xi)\in \p_+SM.
\end{equation}
By \cite[Theorem 4.2.1]{Sharafutdinov},   $\mathcal{I}_1$ can be extended to a  bounded operator $\mathcal{I}_1:H^k(M,T^\ast M)\to H^k(\p_+ SM)$ for every integer $k\ge 0$.

Let us recall that the Riemannian metric $g$ induces the volume form $dV_g$, while the first fundamental form ($g|_{\p M}$) of $\p M$ induces a volume form $d\sigma$ on $\p M$. On the other hand, on the bundle $\p SM$, we use the measure
$
dS_g=d\omega_x \wedge d\sigma,
$
where $d \omega_x$ is the volume on the sphere $S_xM$, which is a fiber of $SM$ at $x\in \p M$.

When a standard weight $\mu \in C^\infty(\p_+SM_1)$ is defined as 
$\mu(y, \theta)= -\n{\theta, \nu(y)}$, $\mathcal I_1$ is a bounded linear operator 
\[
\mathcal I_1 \colon L^2(M,T^\ast M) \to L^2_\mu(\p_+SM_1)=\{f\colon \p_+SM_1 \to \R: \: \|\sqrt{\mu}f\|_{L^2(\p_+SM_1)}<\infty\}.
\]
Thus, this operator has an adjoint
$
\mathcal{I}_1^\ast \colon L^2_\mu(\p_+SM_1) \to L^2(M,T^\ast M).
$

Let us recall that the simple metric $g$ has a smooth simple extension to a slightly larger simple manifold  $M_1$ such that $M \subset \subset M_1^{\mathrm{int}}$. Then it follows from \cite[Theorem 4]{Stefanov_Uhlmann_xray_transform} that for any one-form $A \in L^2(M,T^\ast M)$, there exist constants $C_1, C_2>0$ such that
\begin{equation}
\label{eq:est_transform_1_form}
C_1\|A^s\|_{L^2(M,T^\ast M)} \le \|\mathcal{N}_1(A)\|_{H^1(M_1,T^\ast M_1)} \le C_2\|A^s\|_{L^2(M,T^\ast M)},
\end{equation}
where the normal operator $\mathcal{N}_1=\mathcal{I}_1^\ast\mathcal{I}_1$.

If $\mathcal{O}$ is an open subset of a simple manifold $(M_1,g)$, it was established in \cite[Section 5]{Stefanov_Uhlmann_xray_transform} that  $\mathcal{N}_1$ is an elliptic pseudodifferential operator of order $-1$ on $\mathcal{O}$, and its principal symbol $\rho(x,\xi)=(\rho_{jk}(x,\xi))_{1\le j,k\le n}$, where
\[
\rho_{jk}(x,\xi)= \frac{c_n}{|\xi|}\left(g_{jk}-\frac{\xi_j\xi_k}{|\xi|^2}\right), \quad (x,\xi)\in T\mathcal{O}.
\]
Hence, for each $k\ge 0$ there exists a constant $C_k>0$ such that the estimate
\begin{equation}
\label{eq:est_normal_solenoidal}
\|\mathcal{N}_1(A)\|_{H^{k+1}(M_1,T^\ast M_1)}\le C_k\|A^s\|_{H^k(\overline{\mathcal{O}},T^\ast \overline{\mathcal{O}})}
\end{equation}
holds for all one-forms $A\in H^k(M,T^\ast M)$ that are compactly supported in $\mathcal{O}$, see \cite[Section 8]{Stefanov_Uhlmann_xray_transform}.
 
\subsection{Geodesic ray transforms of functions} 
\label{subsec:transform_functions}

We also define the geodesic ray transform for functions $\mathcal{I}_0: C^\infty(M)\to C^\infty (\p_+ SM)$ by the formula
\begin{equation}
\label{eq:def_ray_transform_functions}
\mathcal{I}_0(f)(x, \xi) =\int_{0}^{\exit(x,\xi)} f(\gamma_{x,\xi}(\tau))d\tau, \quad (x,\xi)\in \p_+SM.
\end{equation}
Then an application of \cite[Theorem 4.2.1]{Sharafutdinov} yields that on the simple manifold $M$,  $\mathcal{I}_0$ can be extended to a bounded operator $\mathcal{I}_0: H^k(M) \to H^k(\p_+ SM)$. 

Furthermore, if $\mathcal{O}$ is an open subset of a simple manifold $(M_1,g)$, as discussed  in \cite[Section 5]{Stefanov_Uhlmann_xray_transform}, the normal operator $\mathcal{N}_0=\mathcal{I}_0^\ast \mathcal{I}_0$ is an elliptic pseudodifferential operator of order $-1$ on $\mathcal{O}$ whose principal symbol is a  multiple of $|\xi|^{-1}_g$. Thus, for each $k\ge 0$ there exists a constant $C_k>0$ such that the estimate
\begin{equation}
\label{eq:est_normal_q}
\|\mathcal{N}_0(f)\|_{H^{k+1}(M_1)}\le C_k\|f\|_{H^k(\overline{\mathcal{O}})}
\end{equation}
holds for all functions $f\in H^k(\mathcal{O})$ with compact support in $\mathcal{O}$, see \cite[Section 9]{Stefanov_Uhlmann_xray_transform}. Finally, 
by \cite[Theorem 3]{Stefanov_Uhlmann_xray_transform}, there exist constants $C_1, C_2>0$ such that the inequality
\begin{equation}
\label{eq:est_transform_function}
C_1\|f\|_{L^2(M)} \le \|\mathcal{I}_0^\ast \mathcal{I}_0(f)\|_{H^1(M_1)} \le C_2\|f\|_{L^2(M)}
\end{equation}
holds for any function $f\in L^2(M)$.

\section{Geometric Optics Solutions for the Hyperbolic Operator}
\label{sec:GO_solution}

Our goal in this section is to construct geometric optics (GO) solutions to the hyperbolic operator \eqref{eq:def_operator} by following similar ideas as in \cite{Bellassoued}. These solutions will be applied in Section \ref{sec:proof_hyperbolic} to provide information on the geodesic ray transform of the difference of  magnetic potentials and electrical potentials.

Consider the nonhomogeneous initial boundary value problem with homogeneous initial and boundary conditions
\begin{equation}
\label{eq:ibvp_general_F}
\begin{cases}
\mathcal{H}_{g,A,q}u=F(t,x) &\text{ in } Q,
\\
u(0,\cdot)=0, \quad \p_tu(0,\cdot)=0 &\text{ in } M,
\\
u=0 &\text{ on } \Sigma.
\end{cases}
\end{equation}
The following lemma, which combines the results originally established in \cite[Theorems 2.30 and 2.45]{KKL_book}, indicates the well-posedness of this problem and the energy estimates the solution satisfies.  
\begin{lem}
\label{lem:wellposedness}
Let $(M,g)$ be a simple manifold of dimension $n\ge 2$ with smooth boundary $\p M$. For any function $F\in L^1(0,T;H^1(M))$ such that $F(0,\cdot)=0$ in $M$, the   problem \eqref{eq:ibvp_general_F} admits a unique solution $u \in C^2(0,T;L^2(M)) \cap C(0,T;H^2(M))\cap C^1(0,T; H_0^1(M))$ satisfying the estimates
\begin{equation}
\label{eq:est_derivative_solution}
\max_{t\in [0,T]} \left(\|u(t,\cdot)\|_{H^1(M)}+\|\p_t u(t,\cdot)\|_{L^2(M)} \right)
\le
C\|F\|_{L^2(Q)},
\end{equation}
\[
\|\p_\nu u\|_{L^2(\Sigma)}\le C\|F\|_{L^2(Q)},
\]
and
\begin{equation}
\label{eq:est_H1_F}
\max_{t\in [0,T]}
\left(\|\p_t^2 u(t,\cdot)\|_{L^2(M)}
+
\|\nabla_g \p_tu(t,\cdot)\|_{L^2(M)}
+
\|\Delta_g u(t,\cdot)\|_{L^2(M)}
\right)
\le
C\|F\|_{H^1(0,T;L^2(M))}.
\end{equation}
\end{lem}

For any function $f\in H^1(0,T;H^2(M)) \cap H^3(0,T;L^2(M))$, we define the norm $\|\cdot \|_\ast$ by the formula
\begin{equation}
\label{eq:def_star_norm}
\|f\|_\ast 
= \|f\|_{H^1(0,T;H^2(M))}+\|f\|_{H^3(0,T;L^2(M))}.
\end{equation}

We are now ready to state and prove the main result of this section. 

\begin{prop}
\label{prop:GO_solutions_initial}
Let $(M,g)$ be a simple Riemannian manifold of dimension $n\ge 2$. Let $A\in W^{1,\infty}(Q)$ and $q\in L^\infty(Q)$. Then the  initial  value problem 
\begin{equation}
\label{eq:ibvp_start}
\begin{cases}
\mathcal{H}_{g,A,q}u(t,x)=0 &\text{ in } Q,
\\
u(0,\cdot)=0, \quad \p_t u(0,\cdot)=0   &\text{ in } M,
% \\
% u=f \quad \text{on} \quad \Sigma,
\end{cases}
\end{equation} 
admits a solution $u\in  C^2(0,T;L^2(M)) \cap C(0,T;H^2(M))$ of the form
\begin{equation}
\label{eq:exp_grow_solution}
u(t,x;h)=e^{\frac{i(\psi(x)-t)}{h}} \alpha(t,x)\beta_A(t,x)+r(t,x;h), 
\quad  h>0.
\end{equation}
Here the remainder term $r \in   C^2(0,T;L^2(M)) \cap C(0,T;H^2(M))\cap C^1(0,T; H_0^1(M))$ satisfies the conditions
\[
r|_{\Sigma}=0  \quad \text{ and } \quad r|_{t=0}=\p_t r|_{t= 0}=0.
\]
Furthermore, there exists a constant $C=C(M,T)>0$ such that the following estimate holds for all $h>0$:
\begin{equation}
\label{eq:est_remainder}
\max_{t\in [0,T]}\left(h^{-1} \|r(t,\cdot)\|_{L^2(M)} + \|\p_t r(t,\cdot)\|_{L^2(M)} + \|\nabla_g r(t,\cdot)\|_{L^2(M)}\right)
\le 
C \|\alpha\|_{\ast}.
\end{equation}

The previous estimate still holds if we pose the terminal conditions $u(T,\cdot)=\p_t u(T, \cdot)=0$ instead of the initial conditions in \eqref{eq:ibvp_start}. In this case, the remainder $r(t,x)$ satisfies the conditions $r|_{t=T}=\p_t r|_{t= T}=0$.
\end{prop}

\begin{proof} 

We shall only prove the claims for $t=0$, as the corresponding results for $t=T$ can be established from similar arguments. The proof follows the main ideas presented in \cite{Bellassoued_Ferreira}, see also \cite{Bellassoued,Bellassoued_Rezig}, and consists of several steps. We begin by finding a  phase function $\psi(x)$, followed by seeking the amplitude terms $\alpha(t,x)$ and $\beta_A(t,x)$. We conclude the proof with the construction of the remainder term $r(t,x)$, as well as the verification of the estimate \eqref{eq:est_remainder}.

Let us begin by applying the hyperbolic operator $\mathcal{H}_{g,A,q}$ to the ansatz $e^{\frac{i(\psi(x)-t)}{h}}\alpha(t,x)\beta_A(t,x)$ in \eqref{eq:exp_grow_solution}. Since the phase function $\psi(x)$ is independent of the time variable $t$, we have that
\begin{equation}
\label{eq:dt_square_term}
\begin{aligned}
\p_t^2\big(e^{\frac{i(\psi(x)-t)}{h}}\alpha(t,x)\beta_A(t,x)\big) =& e^{\frac{i(\psi(x)-t)}{h}}[ \p_t^2(\alpha \beta_A)
-2ih^{-1}(\beta_A\p_t\alpha + \alpha \p_t\beta_A)
-h^{-2} \alpha \beta_A].
\end{aligned}
\end{equation}
When hitting the ansatz with the Laplacian, we get 
\begin{equation}
\label{eq:Laplacian_term}
\begin{aligned}
&\Delta_g\big(e^{\frac{i(\psi(x)-t)}{h}}\alpha(t,x)\beta_A(t,x)\big) 
\\
= &e^{\frac{i(\psi(x)-t)}{h}} [\Delta_g (\alpha \beta_A)
+ 
h^{-1} \left(2i \n{\nabla_g \psi, \alpha \nabla_g \beta_A}_g
+2i \n{\nabla_g \psi, \beta_A \nabla_g\alpha}_g
+i\alpha \beta_A \Delta_g \psi \right)
\\
&-h^{-2} \alpha \beta_A |\nabla_g \psi|_g^2].
\end{aligned}
\end{equation}
For the first order term, we use \eqref{eq:def_magnetic_laplacian} and obtain the following via direct computations: 
\begin{equation}
\label{eq:first_order_term}
2i \left \langle A,d \left(e^{\frac{i(\psi(x)-t)}{h}}\alpha(t,x)\beta_A(t,x)\right) \right \rangle_g 
= 
2ie^{\frac{i(\psi(x)-t)}{h}} \left(\left \langle A,d(\alpha \beta_A)\right \rangle_g 
+
i h^{-1} \n{A,\alpha \beta_A d \psi}_g\right).
\end{equation}
Hence, by combining \eqref{eq:dt_square_term}--\eqref{eq:first_order_term}, we see that
\begin{equation}
\label{eq:operator_GO_comp}
\begin{aligned}
\mathcal{H}_{g,A,q} \big(e^{\frac{i(\psi(x)-t)}{h}} \alpha\beta_A\big) =&e^{\frac{i(\psi(x)-t)}{h}} [h^{-2} \alpha\beta_A (|\nabla_g\psi|_g^2-1)
\\
&+ 2i\beta_A h^{-1} (-\p_t \alpha -\n{\nabla_g \psi, \nabla_g\alpha}_g
- \frac{1}{2}\alpha   \Delta_g \psi)
\\
&+2i \alpha h^{-1} (-\p_t \beta_A - \n{\nabla_g \psi, \nabla_g \beta_A}_g- i\n{A,d \psi}_g\beta_A)
\\
&+\mathcal{H}_{g,A,q} (\alpha\beta_A)].
\end{aligned}
\end{equation}

Our goal is to choose the functions $\psi,\alpha$ and $\beta_A$ such that the terms involving $h^{-2}$ and $h^{-1}$ above will vanish. To achieve this,  \eqref{eq:operator_GO_comp} indicates that the phase function $\psi(x)$ needs to satisfy the eikonal equation
\begin{equation}
\label{eq:eikonal_eq}
|\nabla_g\psi|_g^2=1,
\end{equation}
% which is the multiplier of $h^{-2} \alpha \beta_A$ in \eqref{eq:operator_GO_comp}.
while the amplitudes $\alpha$ and $\beta_A$ need to solve the following transport equations: 
% which are the multipliers of $2ih^{-1}\beta_A$ and $2ih^{-1}\alpha$, respectively, to obtain the amplitudes $\alpha(t,x)$ and $\beta_A(t,x)$:
\begin{equation}
\label{eq:transport_eq_alpha}
\p_t \alpha +\n{\nabla_g \psi, \nabla_g\alpha}_g+ \frac{1}{2}\alpha   \Delta_g \psi=0
\end{equation}
and
\begin{equation}
\label{eq:transport_eq_beta}
\p_t \beta_A + \n{\nabla_g \psi, \nabla_g \beta_A}_g+i\n{A, \nabla_g \psi}\beta_A=0.
\end{equation}

Since $(M,g)$ is simple, it has a simple extension $(M_1,g)$ such that $M \subset \subset M_1^{\mathrm{int}}$. Thanks to the simplicity of $M_1$, we may solve the eikonal and the transport equations using geodesic polar coordinates. To this end, since the exponential map at $y \in \p M_1$  is a diffeomorphism, any point $x\in M_1$ can be expressed as $x=\exp_y(r\theta)=\gamma_{y,\theta}(r)$, where 
% $(r, \theta)$ is the polar coordinate in $M_1$ centered at $y$. Here 
$r>0$ and $\theta \in S_yM_1=\{\xi \in T_y M_1: |\xi|=1\}$.
Furthermore, by Gauss' Lemma in these coordinates, which depend on the choice of $y$, the metric $g$ can be written as
\[
\tilde{g}(r,\theta)=dr^2+g_0(r,\theta),
\]
where $g_0(r,\theta)$ is a smooth positive definite metric on $S_yM_1$, see \cite[Theorem 2.93]{gallot1990riemannian}.
Then for any compactly supported function $u$ in $M$, we set
\begin{equation}
\label{eq:notation_polar}
\tilde{u}(r,\theta) = u(\exp_y(r\theta))
\end{equation}
for any $r>0$ and $\theta \in S_yM_1$, where we have extended $u$ by zero in $M_1 \setminus M$.

We now aim to find a solution to the eikonal equation \eqref{eq:eikonal_eq}. To that end,  an explicit solution to \eqref{eq:eikonal_eq} is given by the geodesic distance function to $y\in \p M_1$
%\begin{equation}
%\label{eq:phase_geo_dist}
\[
\psi(x)=d_g(x,y).
\]
%\end{equation}
Since $M_1$ is a simple manifold and $y\in M_1 \setminus \overline{M}$, we see that $\psi \in C^\infty(M)$ and
\begin{equation}
\label{eq:phase_polar}
\tilde{\psi}(r,\theta)=r=d_g(x,y).
\end{equation}

Let us next solve the first transport equation \eqref{eq:transport_eq_alpha}. Note that, for any function $f(r)$ of the geodesic distance $r$, we have
\begin{equation}
\label{eq:Laplacian_polar}
\Delta_{\tilde{g}}f(r)=f''(r)+\frac{\rho^{-1}}{2}\p_r \rho\, f'(r) ,
\end{equation}
where $\rho(r, \theta) = \det g_0(r,\theta)$ is the square of the volume element of the metric in polar coordinates, see \cite[Section 6.5]{Canzani_notes}.
Hence, by incorporating \eqref{eq:notation_polar}, \eqref{eq:phase_polar}, and \eqref{eq:Laplacian_polar} for $f(r)=r$, we get from the transport equation \eqref{eq:transport_eq_alpha} that the function $\tilde \alpha (t, r, \theta)$ needs to satisfy the equation
\[
\p_t\tilde{\alpha} + \p_r \tilde{\alpha} + \frac{1}{4} \tilde \alpha \rho^{-1} \p_r \rho=0. 
\]
%\begin{equation}
%\label{eq:transport_alpha_polar}
%\end{equation}

Let $\phi\in C_0^\infty(\R)$ and $\Psi \in C^\infty(\p_+SM_1) \cap H^2(\p_+SM_1)$. Using the same arguments as in \cite[Section 4]{Bellassoued_Ferreira}, we have that $\tilde \alpha(t,r,\theta)$ is given by the formula
\begin{equation}
\label{eq:alpha_solution}
\tilde \alpha(t,r,\theta;y)=\rho^{-\frac{1}{4}}(r,\theta)\phi(t-r)\Psi(y,\theta).
\end{equation}
Let $\varepsilon>0$ be such that $T>\diam(M_1)+3\varepsilon$. If we assume that $\supp \phi \subset (\varepsilon,2\varepsilon)$,
then for any $t<\varepsilon$, we have that $t-r<\varepsilon - r \leq \varepsilon$, which implies that $\phi(t-r)=0$.
Also, since $ r \in [0, \diam(M_1)]$, the assumption $t>T-\varepsilon$ implies that $t-r>T-\varepsilon-r\geq T-\varepsilon-\diam(M_1)>2\varepsilon$, from which we get that $\phi(t-r)=0$.
Therefore, we conclude that $\tilde{\alpha}(t,r,\theta)=0$ and $\p_t^j\tilde{\alpha}(t,r,\theta)=0$, $j=1,2$, for $t<\varepsilon$ and $t > T -\varepsilon$.

We now turn our attention to the second transport equation \eqref{eq:transport_eq_beta}. Due to Gauss' Lemma \cite[Theorem 6.9]{lee2018introduction}, we have $\nabla_g \psi=\dot \gamma_{y,\theta}(r)$, where $r$ is given by \eqref{eq:phase_polar}. Thus, it follows that
\begin{equation}
\label{eq:sigma_A_def}
\n{\tilde{A}(r,y,\theta),d \psi}_g = \n{ \tilde A^\sharp(\gamma_{y,\theta}(r)), \dot{\gamma}_{y,\theta}(r)}_g =: \tilde{\sigma}_A(r,y,\theta).
\end{equation}
We also get from the equation \eqref{eq:transport_eq_beta} that the function $\tilde \beta_A (t,r, \theta)$ must satisfy the equation
\[
\p_t \tilde \beta_A + \p_r \tilde \beta_A + i\tilde{\sigma}_A \tilde \beta_A=0.
\]
%\begin{equation}
%\label{eq:transport_beta_A_polar}
%\end{equation}
By a direct computation, we have that
\begin{equation}
\label{eq:beta_A_solution}
\tilde \beta_A(t,r,\theta; y)=\exp \left(-i\int_{0}^{t}\tilde \sigma_A(r-s, y,\theta)ds\right).
\end{equation}

Next we aim to show the existence of the remainder term $r(t,x)$, which satisfies the estimate \eqref{eq:est_remainder}. To this end, we  observe that the function $u(t,x)$  given by \eqref{eq:exp_grow_solution}  solves the   problem \eqref{eq:ibvp_start} if and only if $r(t,x)$ is a solution to the initial boundary value problem \eqref{eq:ibvp_general_F} with $F(t,x)=-\mathcal{H}_{g,A,q}\big(e^{\frac{i(\psi(x)-t)}{h}} \alpha(t,x) \beta_A(t,x) \big)$.
% \begin{equation}
% \label{eq:ibvp_remainder}
% \begin{cases}
% \mathcal{H}_{g,A,q}r(t,x)=-\mathcal{H}_{g,A,q}\big(e^{\frac{i(\psi(x)-t)}{h}} \alpha(t,x) \beta_A(t,x) \big) \quad \text{in} \quad Q,
% \\
% r(0,\cdot)=0, \quad \p_t r(0,\cdot)=0 \quad \text{in} \quad M,
% \\
% r=0 \quad \text{on} \quad \Sigma.
% \end{cases}
% \end{equation}
By Lemma \ref{lem:wellposedness}, 
% when applied to the function $F=-\mathcal{H}_{g,A,q}\big(e^{\frac{i(\psi(x)-t)}{h}}\alpha(t,x)\beta_A(t,x)\big)$, 
% the problem \eqref{eq:ibvp_general_F} has a unique solution 
we get  $r \in   C^2(0,T;L^2(M)) \cap C(0,T;H^2(M))\cap C^1(0,T; H_0^1(M))$ and  
\[
\|r(t,\cdot)\|_{H^1(M)}+\|\p_t r(t,\cdot)\|_{L^2(M)} \le C\|\mathcal{H}_{g,A,q}\big(e^{\frac{i(\psi(x)-t)}{h}}\alpha(t,x)\beta_A(t,x)\big)\|_{L^2(Q)}.
\]

In order to derive the suggested estimate \eqref{eq:est_remainder}, we begin by estimating the source term 
\begin{equation}
\label{eq:def_of_R}
R(t,x) :=-\mathcal{H}_{g,A,q}\big(e^{\frac{i(\psi(x)-t)}{h}} \alpha(t,x) \beta_A(t,x) \big)
\end{equation}
of the problem  \eqref{eq:ibvp_general_F} when $F=R$.  First, by taking into account that the functions $\psi(x)$, $\alpha(t,x)$, and $\beta_A(t,x)$ satisfy the eikonal and transport equations \eqref{eq:eikonal_eq}--\eqref{eq:transport_eq_beta}, respectively, it follows from \eqref{eq:operator_GO_comp} that 
\begin{align*}
R(t,x)
% &:=-\mathcal{H}_{g,A,q}\big(e^{\frac{i(\psi(x)-t)}{h}}\alpha(t,x)\beta_A(t,x)\big) 
% \\
&= - e^{\frac{i(\psi(x)-t)}{h}} \mathcal{H}_{g,A,q} \left(\alpha(t,x)\beta_A(t,x)\right)
% \\
% &
=:-e^{\frac{i(\psi(x)-t)}{h}} R_0(t,x).
\end{align*}
Due to the fact that $\mathcal{H}_{g,A,q}$ is a second order differential operator, we get the estimates
\[
\|R_0\|_{L^2(Q)}
\le
C\big(\|\alpha\|_{L^2(0,T;H^2(M))}+\|\alpha\|_{H^2(0,T;L^2(M))}\big)
\le
C \|\alpha\|_\ast
\]
%\begin{equation}
%\label{eq:est_R0}
%\end{equation}
and
\[
\|\p_t R_0\|_{L^2(Q)}
\le 
C\|R_0\|_{H^1(0,T;L^2(M))}
\le 
C\big(\|\alpha\|_{H^1(0,T;H^2(M))}
+
\|\alpha\|_{H^3(0,T;L^2(M))}\big)
\le
C\|\alpha\|_\ast.
\]
Thus, it follows immediately from the two estimates above that
\begin{equation}
\label{eq:est_R_and_time_derivative}
\|R_0\|_{L^2(Q)} + \|\p_t R_0\|_{L^2(Q)} 
\le 
C \|\alpha\|_{\ast}.
\end{equation}

Armed with this estimate, we next prove that
\begin{equation}
\label{eq:est_rem}
\max_{t\in [0,T]}\|r(t,\cdot)\|_{L^2(M)}
\le 
Ch \|\alpha\|_\ast.
\end{equation}
To obtain this, we  write $r=\p_t v$, where the function $v$ solves the problem \eqref{eq:ibvp_general_F} for some suitable interior source term $F=V$. Our goal is to show that the source term $V$ satisfies the estimate 
\begin{equation}
\label{eq:est_V_lambda}
\|V\|_{L^2(Q)} 
\le 
Ch\|\alpha\|_\ast,
\end{equation}
where $C$ depends on $T$. If so, the suggested estimate \eqref{eq:est_rem} follows from  
the fact that $r=\p_t v$, and in  view of the estimate \eqref{eq:est_derivative_solution} with $F=V$, we have that
\[
\max_{t \in [0,T]}\left(
\|v(t,\cdot)\|_{H^1(M)}
+
\|\p_t v(t,\cdot)\|_{L^2(M)}
\right)
\le
C\|V\|_{L^2(Q)}.
\]

Let us now show that the estimate \eqref{eq:est_V_lambda} is valid. Since the coefficients $A(x)$ and $q(x)$ are both time-independent,  direct computations yield that the function $v(t,x) :=\int_{0}^{t} r(s,x)ds$
satisfies the problem \eqref{eq:ibvp_general_F} with the interior source term
\begin{equation}
\label{eq:source_V}
V(t,x)=\int_0^t R(s,x)ds=-ih\int_0^t R_0(s,x)\p_s\left(e^{\frac{i(\psi(x)-s)}{h}}\right) ds.
\end{equation}
We integrate by parts in \eqref{eq:source_V} with respect to the $s$-variable and obtain
\begin{equation}
\label{eq:V_ibp}
V(t,x) 
=
-ih \left(R_0(t,x)e^{\frac{i(\psi(x)-t)}{h}}
-
R_0(0,x)e^{\frac{i\psi(x)}{h}}
-
\int_{0}^{t} e^{\frac{i(\psi(x)-s)}{h}}\p_sR_0(s,x)ds\right).
\end{equation}

By the trace theorem \cite[Section 5.5, Theorem 1]{Evans} and the fact that $\mathcal{H}_{g,A,q}$ is a second order partial differential operator, we deduce that
\begin{align*}
\|R_0(0,\cdot)\|_{L^2(Q)}^2
&= 
\int_{0}^T \int_M |R_0(0,x)|^2dV_gdt
\\
&= T \|R_0(0,\cdot)\|_{L^2(M)}^2
\\
& \le 
CT \|R_0(t,\cdot)\|^2_{H^{1}(0,T;L^2(M))}
\\
&\le
CT \left(\|\alpha\|_{H^1(0,T;H^2(M))}^2+\|\alpha\|_{H^3(0,T;L^2(M))}^2\right)
\\
& \le
C T\|\alpha\|_{\ast}^2.
\end{align*}
For the third term, we apply Fubini's theorem, in conjunction with the Cauchy-Schwarz inequality, to compute that
\begin{align*}
\int_0^T \int_M \int_{0}^{t} \big|\p_sR_0(s,x)\big|^2 dsdV_gdt
&=
\int_0^t \int_M \int_{0}^{T} \big|\p_sR_0(s,x)\big|^2 dtdV_gds
\\
&\le
\left(\int_0^T \int_M  \big|\p_sR_0(s,x)\big|^2dV_gds\right) \left(\int_{0}^{T}1 dt\right)
\\
&= T\|\p_tR_0\|_{L^2(Q)}^2.
\end{align*}

Since the phase function $\psi$ is real-valued,  by incorporating the two computations above, we conclude from \eqref{eq:V_ibp} that
\[
\|V\|_{L^2(Q)}^2
\le
Ch^2  \big(\|R_0\|_{L^2(Q)}^2
+T\|\alpha\|_\ast^2
+ T\left\|\p_tR_0\right\|_{L^2(Q)}^2 \big).
\]
Here we have utilized the inequality $(a+b+c)^2 \le 3(a^2+b^2+c^2)$ for   $a,b,c\in \R$. Whence, by the estimate \eqref{eq:est_R_and_time_derivative}, we get the estimate \eqref{eq:est_V_lambda}.

We still need to estimate $\|\p_t r(t,\cdot)\|_{L^2(M)}$ and $\|\nabla_g r(t,\cdot)\|_{L^2(M)}$ to complete the proof of the lemma. To this end, we recall that $R(t,x)=-e^{\frac{i(\psi(x)-t)}{h}} R_0(t,x)$. Due to the estimate \eqref{eq:est_R_and_time_derivative}, we have
\begin{equation}
\label{eq:est_R_L2}
\|R\|_{L^2(Q)} = \|R_0\|_{L^2(Q)}\le C \|\alpha\|_\ast.
\end{equation}
Thus, 
% by applying the energy estimate to equation \eqref{eq:ibvp_general_F}, see for instance \cite[Section 7.2, Theorem 2]{Evans}, 
by Lemma \ref{lem:wellposedness}, after setting $F=R$, we get
\begin{equation}
\label{eq:est_derivatives_remainder}
\max_{t\in [0,T]}
\left(
\|\p_t r(t,\cdot)\|_{L^2(M)}+ \|\nabla_g r(t,\cdot)\|_{L^2(M)} 
\right)
\le 
C \|R\|_{L^2(Q)} 
\le 
C\|\alpha\|_{\ast}.
\end{equation}
Finally, we obtain the estimate \eqref{eq:est_remainder} by combining \eqref{eq:est_rem} and \eqref{eq:est_derivatives_remainder}. This completes the proof of Proposition \ref{prop:GO_solutions_initial}.
\end{proof}

For later purposes, let us also derive estimates for $\|r(t,\cdot)\|_{H^2(M)}$ and $\|\p_t r(t, \cdot)\|_{H^1(M)}$. We compute from \eqref{eq:def_of_R} that
\[
\p_t R = ih^{-1} e^{\frac{i(\psi(x)-t)}{h}} R_0 -e^{\frac{i(\psi(x)-t)}{h}} \p_t R_0.
\]
Thus, we combine the estimates \eqref{eq:est_R_and_time_derivative} and \eqref{eq:est_R_L2} to obtain 
\begin{equation}
\label{eq:dtR}
\begin{aligned}
\|\p_t R\|_{L^2(Q)} 
% &\le
% Ch^{-1}\|R_0\|_{L^2(Q)}+C\|\p_tR_0\|_{L^2(Q)}
% \\
&\le
Ch^{-1}\|\alpha\|_\ast.
\end{aligned}
\end{equation}
Consequently, it follows  from \eqref{eq:est_R_L2} and \eqref{eq:dtR} that
\[
\|R\|_{L^2(Q)}+\|\p_t R\|_{L^2(Q)}
\le
Ch^{-1}\|\alpha\|_\ast,
\]
which implies that
\[
\|R\|_{H^1(0,T;L^2(M))}\le 
Ch^{-1}\|\alpha\|_\ast.
\]

Since the remainder $r$ solves the problem \eqref{eq:ibvp_general_F}  for $F=R$, an application of the estimate \eqref{eq:est_H1_F} yields that 
\begin{equation}
\label{eq:est_r_2nd_derivative}
\max_{t\in [0,T]}
\left(
\|\p_t^2 r(t,\cdot)\|_{L^2(M)}
+
\|\nabla_g (\p_t r(t,\cdot))\|_{L^2(M)}
+
\|\Delta_g  r(t,\cdot)\|_{L^2(M)}
\right)
\le
Ch^{-1}\|\alpha\|_\ast.
\end{equation}
Hence, by combining estimates \eqref{eq:est_rem}, \eqref{eq:est_derivatives_remainder}, and \eqref{eq:est_r_2nd_derivative}, we obtain that
\begin{equation}
\label{eq:est_remainder_H2}
\max_{t\in [0,T]}
\|r(t,\cdot)\|_{H^2(M)}
\le
Ch^{-1}\|\alpha\|_\ast.
\end{equation}
Furthermore, we deduce from the estimates \eqref{eq:est_derivatives_remainder} and \eqref{eq:est_r_2nd_derivative} that
\begin{equation}
\label{eq:est_time_der_remainder_H1}
\begin{aligned}
\max_{t\in [0,T]}
\|\p_t r(t,\cdot)\|_{H^1(M)}
&\leq
C\max_{t\in [0,T]}
\left(\|\p_t r(t,\cdot)\|_{L^2(M)}
+
\|\nabla_g (\p_t r(t,\cdot))\|_{L^2(M)}\right)
% \\
% & 
\leq 
%\nonumber
Ch^{-1}\|\alpha\|_\ast.
\end{aligned}
\end{equation}

\section{Proof of Theorem \ref{thm:main_result_hyperbolic}}
\label{sec:proof_hyperbolic}

In this section we aim to establish H\"older-type estimates for the solenoidal part $A^s$ of the magnetic potential $A$, as well as the electric potential $q$, from the hyperbolic Dirichlet-to-Neumann map $\Lambda_{A,q}$ defined by the formula \eqref{eq:def_hyperbolic_DN_map}. The proof relies on the GO solutions constructed in Section \ref{sec:GO_solution}, as well as an integral identity. We shall complete the proof in two steps, which we present in succeeding subsections. We first derive a H\"older-type stability for $A^s$ in Subsection \ref{subsec:proof_solenoidal_part_A}, where the amplitude term $\beta_A$ in the GO solutions \eqref{eq:exp_grow_solution} provides information about the geodesic ray transform of the magnetic potential. Afterwards, we apply \cite[Theorem 4]{Stefanov_Uhlmann_xray_transform} to obtain the desired estimate. In Subsection \ref{subsec:proof_electric_hyperbolic} we provide a proof for the stability of the electric potential $q$, where we   obtain information pertaining to the geodesic ray transform of $q$ from the amplitude term $\alpha$ in the GO solution.

Let the Riemannian manifold $(M,g)$, covector fields $A_1$ and $A_2$, as well as the functions $q_1$ and $q_2$ in $M$ be as in the assumptions of Theorem \ref{thm:main_result_hyperbolic}. 
% \in \mathcal{A}(\lceil n \rceil,N)$ and $q_1, q_2\in \mathcal{Q}(N)$. 
As in the previous section, we extend the manifold $M$ to a slightly larger simple manifold $M_1$ such that $M\subset \subset M_1^{\mathrm{int}}$ and $T>\diam(M_1)$. Due to the assumptions that $A_1=A_2$ and $q_1=q_2$ on $\p M$, these coefficients have $H^1$-extensions to $M_1$ so that $A_1=A_2$ and $q_1=q_2$ in $M_1\setminus M$.

In the remainder of this section, let us denote $A=A_1-A_2$ and $q=q_1-q_2$. Thus, we have that $A\in H^1(M_1, T^\ast M_1)$ and $q\in H^1(M_1)$, and both of them vanish outside $M$.

\subsection{Stability estimate for the solenoidal part of the magnetic potential}
\label{subsec:proof_solenoidal_part_A}

Our goal of this subsection is to prove the following H\"older-type stability estimate 
\begin{equation}
\label{eq:est_As}
\|A_1^s-A_2^s\|_{L^2(M)}
\le
C\|\Lambda_{A_2,q_2}-\Lambda_{A_1,q_1}\|^{\frac{1}{12}}.
\end{equation}
Here the constant $C$ depends on $(M,g), N$, $n$, and $T$.

We begin this subsection with the derivation of an integral identity. To that end, we shall make use of the following Green's identity for the magnetic Laplacian $\Delta_{g,A}$, which was established for instance in \cite[Section 2]{Bellassoued_Aicha}:
\begin{equation}
\label{eq:Green_mag_Laplacian}
\int_M v \Delta_{g,A} u dV_g 
-
\int_M u \overline{\Delta_{g,A}v}  dV_g
=
\int_{\p M} \overline{v}(\p_\nu +i \n{A,\nu}_g)u dS_g
-
\int_{\p M} \overline{u} \overline{(\p_\nu +i \n{A,\nu}_g)v} dS_g.
\end{equation}
Let us remark that this identity is valid for any functions $u,v\in H^1(M)$ such that $\Delta_g u, \Delta_g v\in L^2(M)$. 

Let $u_2 \in  C^2(0,T;L^2(M)) \cap C(0,T;H^2(M))$ be the GO solution for the problem \eqref{eq:ibvp_start} with $A=A_2$ and $q=q_2$. If we set $f:=u_2|_{\Sigma}$,
% \in L^2_0(\Sigma):= \{f \in L^2(\Sigma):\: \supp{f}\subset \Sigma^{\mathrm{int}}\}$
then $u_2$ solves the initial boundary value problem
\begin{equation}
\label{eq:ibvp_u2}
\begin{cases}
\mathcal{H}_{g,A_2,q_2}u_2=0   & \hbox{ in } Q,
\\
u_2(0,\cdot)=0, \quad \p_tu_2(0,\cdot)=0 & \hbox{ in } M,
\\
u_2=f & \hbox{ on } \Sigma.
\end{cases}
\end{equation}
On the other hand, if we choose $v$ to be the solution of the problem
\begin{equation}
\label{eq:ibvp_v}
\begin{cases}
\mathcal{H}_{g,A_1,q_1}v=0 & \hbox{ in } Q,
\\
v(0,\cdot)=0, \quad \p_tv(0,\cdot)=0 & \hbox{ in } M,
\\
v=f & \hbox{ on } \Sigma,
\end{cases}
\end{equation} 
then it is straightforward to check that the function $u:=u_2-v$ satisfies the problem
\begin{equation}
\label{eq:ibvp_difference}
\begin{cases}
\mathcal{H}_{g,A_1,q_1}u=-2i\n{A, du_2}_g+(-id^\ast A+|A_1|^2_g-|A_2|^2_g+q)u_2 & \hbox{ in } Q,
\\
u(0,\cdot)=0, \quad \p_tu(0,\cdot)=0 \quad & \hbox{ in } M,
\\
u=0 & \hbox{ on } \Sigma.
\end{cases}
\end{equation}
Since $-2i\n{A, du_2}_g+(-id^\ast A+|A_1|^2_g-|A_2|^2_g+q)u_2 \in C(0,T;H^1(M))$, it follows from Lemma \ref{lem:wellposedness} that $u\in C^1(0,T;L^2(M))\cap C(0,T;H^1(M))$.
Moreover, by Proposition  \ref{prop:GO_solutions_initial}, there exists a GO solution $u_1\in C^2(0,T;L^2(M)) \cap C(0,T;H^2(M))$ to the problem
\[
\begin{cases}
\mathcal{H}_{g,A_1,q_1}  u_1=0 & \hbox{ in } Q,
\\
u_1(T,\cdot)=0, \quad \p_t u_1(T,\cdot)=0 & \hbox{ in } M.
\end{cases}
\]

We now multiply both sides of the first equation in \eqref{eq:ibvp_difference} by  $\overline{u_1}$ and integrate over the spacetime $Q$. Since $u(0,\cdot)=\p_tu(0,\cdot)=0$ and $u_1(T,\cdot)=\p_t u_1(T,\cdot)=0$, we deduce from Green's identity \eqref{eq:Green_mag_Laplacian} that
\begin{align*}
\int_Q u_1 \mathcal{H}_{g,A_1,q_1}u dV_g dt
&=
\int_Q u_1 \mathcal{H}_{g,A_1,q_1}u dV_g dt
-
\int_Q \overline{u} \overline{\mathcal{H} _{g,A_1,q_1}u_1}  dV_g dt
\\
&=
\int_{\Sigma} \overline{u_1}(\p_\nu +i \n{A_1,\nu}_g)u dS_gdt
-
\int_{\Sigma} \overline{u} \overline{(\p_\nu +i \n{A_1,\nu}_g)u_1} dS_gdt
\\
&=
\int_{\Sigma} \overline{u_1} \p_\nu u dS_gdt.
\end{align*}
Here we have used  that the operator $\mathcal{H} _{g,A_1,q_1}$ is self-adjoint, and that $u=0$ on $\Sigma$.

Furthermore, the equation
\begin{align*}
(\Lambda_{A_2,q_2}-\Lambda_{A_1,q_1})(f)
=&
[(\p_\nu +i \n{A_1,\nu}_g)u_2]|_{\Sigma}-[(\p_\nu +i \n{A_2,\nu}_g)v]|_{\Sigma}
% \\
% &
=
\p_\nu u |_{\Sigma}.
\end{align*}
holds as the functions $u_2$ and $v$ solve the initial boundary value problems \eqref{eq:ibvp_u2} and \eqref{eq:ibvp_v}, respectively, 
and satisfy the boundary condition $u_2|_{\Sigma}=v|_{\Sigma}=f$. Here we have also used the definition \eqref{eq:def_hyperbolic_DN_map} of the Dirichlet-to-Neumann map, as well as the assumption that $A_1=A_2$ on $\p M$.
Finally, as $u$ satisfies the problem \eqref{eq:ibvp_difference}, we obtain the following integral identity as a consequence of the two previous equations: 
\begin{equation}
\label{eq:int_id}
\begin{aligned}
&\int_Q -2i\n{A, du_2}_g \overline{u_1} dV_gdt
\\
&= 
\int_{\Sigma}
\overline{u_1}(\Lambda_{A_2,q_2}-\Lambda_{A_1,q_1})(f)  dS_gdt
-
\int_Q (-id^\ast A+|A_1|^2_g-|A_2|^2_g+q) \overline{u_1} u_2 dV_gdt.
\end{aligned}
\end{equation}

We are now ready to state the first estimate to be used to prove the main estimate \eqref{eq:est_As}.

\begin{lem}
\label{lem:int_estimate}
There exists a constant $C>0$ such that for any amplitudes $\alpha_j$ and $\beta_{A_j}$, $j=1,2$, which satisfy the transport equations \eqref{eq:transport_eq_alpha} and \eqref{eq:transport_eq_beta}, respectively, the estimate
\begin{equation}
\label{eq:int_est_amplitudes}
\left|\int_Q \n{A,d\psi}_g \left(\overline{\alpha_1}\alpha_2\right) \left(\beta_{A_2}\overline{\beta_{A_1}}\right) dV_gdt\right| 
\le 
C\left(h+h^{-2} \|\Lambda_{A_2,q_2}-\Lambda_{A_1,q_1}\|_{\mathcal{L}(L^2(\Sigma),H^{-1}(\Sigma))} \right)\|\alpha_1\|_{\ast} \|\alpha_2\|_{\ast}
\end{equation}
holds for $0<h < 1$.
\end{lem}

\begin{proof}
By Proposition \ref{prop:GO_solutions_initial}, there exist GO solutions $u_1$ and $u_2$ to the equations  $\mathcal{H}_{g,A_1,q_1}  u_1=0$ and  $\mathcal{H}_{g,A_2,q_2}u_2=0$ in $Q$, respectively, which are of the forms
\begin{equation}
\label{eq:solution_v}
u_1(t,x)=e^{\frac{i(\psi(x)-t)}{h}}\alpha_1(t,x)\beta_{A_1}(t,x)+r_1(t,x)
\end{equation}
and
\begin{equation}
\label{eq:solution_u2}
u_2(t,x)=e^{\frac{i(\psi(x)-t)}{h}}\alpha_2(t,x)\beta_{A_2}(t,x)+r_2(t,x).
\end{equation}
Here the remainder $r_i$, $i=1,2$, satisfies the estimate
\begin{equation}
\label{eq:est_r1}
h^{-1} \|r_i(t,\cdot)\|_{L^2(M)} + \|\p_t r_i(t,\cdot)\|_{L^2(M)} + \|\nabla_g r_i(t,\cdot)\|_{L^2(M)}
\le 
C \|\alpha_i\|_{\ast} 
\quad 
\text{ for } 0<h<1,
\end{equation}
% and
% \begin{equation}
% \label{eq:est_r2}
% h^{-1} \|r_2(t,\cdot)\|_{L^2(M)} + \|\p_t r_2(t,\cdot)\|_{L^2(M)} + \|\nabla_g r_2(t,\cdot)\|_{L^2(M)}
% \le 
% C \|\alpha_2\|_{\ast}
% \end{equation}
where the norm $\| \cdot \|_\ast$ is given by \eqref{eq:def_star_norm}.

We now substitute the GO solutions $u_1$ and $u_2$ into the integral identity \eqref{eq:int_id},  multiply the resulting expression by $h$, and proceed to estimate each term appearing there.
% , multiply   the resulting expression by $h$, and estimate it  in the limit $h\to 0$. 
To this end, it follows from direct computations that
\begin{equation}
\label{eq:computation_du2u1}
\begin{aligned}
-2ih\n{A, du_2}_g\overline{u_1}
= &
-2\n{A,d\psi}_g(\overline{\alpha_1}\alpha_2)(\beta_{A_2}\overline{\beta_{A_1}}) 
-
2 \overline{r}_1 e^{\frac{i(\psi(x)-t)}{h}} \n{A,d\psi}_g (\alpha_2\beta_{A_2}) 
\\
&- 2ih(\overline{\alpha_1\beta_{A_1}})\n{A,d(\alpha_2\beta_{A_2})}_g
-
2ih\overline{r}_1\n{A,d(\alpha_2\beta_{A_2})}_g 
\\
&-
2ihe^{\frac{i(\psi(x)-t)}{h}} \n{A,dr_2}_g(\overline{\alpha_1\beta_{A_1}})
-
2ih\overline{r}_1\n{A, dr_2}_g.
\end{aligned}
\end{equation}
Therefore, we have 
\begin{equation}
\label{eq:int_id_GO}
\begin{aligned}
&\int_Q 	-2\n{A,d\psi}_g(\overline{\alpha_1}\alpha_2)(\beta_{A_2}\overline{\beta_{A_1}}) dV_gdt
\\
= & \int_Q -2\n{A,d\psi}_g(\overline{\alpha_1}\alpha_2)(\beta_{A_2}\overline{\beta_{A_1}}) dV_gdt
-
\int_Q 2 \overline{r}_1 e^{\frac{i(\psi(x)-t)}{h}} \n{A,d\psi}_g (\alpha_2\beta_{A_2}) dV_gdt
\\
&- \int_Q 2ih(\overline{\alpha_1\beta_{A_1}})\n{A,d(\alpha_2\beta_{A_2})}_g dV_gdt
-
\int_Q 2ih\overline{r}_1\n{A,d(\alpha_2\beta_{A_2})}_g dV_gdt
\\
&-
\int_Q 2ihe^{\frac{i(\psi(x)-t)}{h}} \n{A,dr_2}_g(\overline{\alpha_1\beta_{A_1}}) dV_gdt
-
\int_Q 2ih\overline{r}_1\n{A, dr_2}_g dV_gdt
\\
&-
\int_Q (-id^\ast A+|A_1|^2_g-|A_2|^2_g+q) \left(\overline{\alpha_1 \beta_{A_1}} +\overline{r_1}\right)
\left(\alpha_2 \beta_{A_2}+r_2\right) dV_gdt
\\
&+
\int_{\Sigma} (\Lambda_{A_2,q_2}-\Lambda_{A_1,q_1})(f) \overline{u_1} dS_gdt.
\end{aligned}
\end{equation}

Let us estimate each term on the right-hand side of \eqref{eq:int_id_GO}. To this end, by applying the Cauchy-Schwarz inequality, along with  the estimate \eqref{eq:est_r1}, we obtain
\begin{equation}
\label{eq:est_LHS_term1}
\begin{aligned}
\left|\int_Q 2 \overline{r}_1 e^{\frac{i(\psi(x)-t)}{h}} \n{A,d\psi}_g (\alpha_2\beta_{A_2})  dV_gdt\right|
&\le 
C\int_0^T \|r_1(t,\cdot)\|_{L^2(M)} \|\alpha_2(t,\cdot)\|_{L^2(M)}dt 
\\
&\le
Ch\|\alpha_1\|_\ast \|\alpha_2\|_\ast.
\end{aligned}
\end{equation}
In the computations above, we also used the assumption that $A_1, A_2\in \mathcal{A}(\lceil \frac{n}{2}\rceil+1, N)$, together with the equations \eqref{eq:sigma_A_def} and  \eqref{eq:beta_A_solution}, as well as the Sobolev embedding theorem, to bound the terms $\|\n{A,d\psi}_g\|_{L^\infty(Q)}$ and $\|\beta_{A_2}\|_{L^\infty(Q)}$.

Similarly, we find that
\begin{equation}
\label{eq:est_LHS_term2}
\begin{aligned}
\left|\int_Q 2i h\left(\overline{\alpha_1\beta_{A_1}}\right)
\n{A,d(\alpha_2\beta_{A_2})}_g dV_gdt\right|
&\le 
Ch\int_0^T \|\alpha_1(t,\cdot)\|_{L^2(M)} \|\alpha_2(t,\cdot)\|_{H^1(M)}d 
\\
&\le
Ch\|\alpha_1\|_\ast \|\alpha_2\|_\ast.
\end{aligned}
\end{equation}
Moreover, we apply the estimate \eqref{eq:est_r1} once again to see that
\begin{equation}
\label{eq:est_LHS_term3}
\begin{aligned}
\left|\int_Q 2i h\overline{r}_1\n{A,d(\alpha_2\beta_{A_2})}_g dV_gdt  \right|
&\le 
Ch\int_0^T \|r_1 (t,\cdot)\|_{L^2(M)} \|\alpha_2(t,\cdot)\|_{H^1(M)}dt 
\\
&\le
Ch^2\|\alpha_1\|_\ast \|\alpha_2\|_\ast.
\end{aligned}
\end{equation}
Furthermore, using  \eqref{eq:est_r1}, we establish the inequality
\begin{equation}
\label{eq:est_LHS_term4}
\begin{aligned}
\left|\int_Q 2i he^{\frac{i(\psi(x)-t)}{h}} \n{A,dr_2}_g(\overline{\alpha_1\beta_{A_1}}) dV_gdt \right|
&\le 
Ch\int_0^T \|\nabla_g r_2 (t,\cdot)\|_{L^2(M)} \|\alpha_1(t,\cdot)\|_{L^2(M)}dt 
\\
&\le
Ch\|\alpha_1\|_\ast \|\alpha_2\|_\ast.
\end{aligned}
\end{equation}
Additionally, by \eqref{eq:est_r1}, we get
\begin{equation}
\label{eq:est_LHS_term5}
\begin{aligned}
\left|\int_Q 2i h\overline{r}_1\n{A, dr_2}_g  dV_gdt \right|
&\le 
Ch\int_0^T \|\nabla_g r_2 (t,\cdot)\|_{L^2(M)} \|r_1(t,\cdot)\|_{L^2(M)}dt 
\\
&\le
Ch^2\|\alpha_1\|_\ast \|\alpha_2\|_\ast.
\end{aligned}
\end{equation}
Finally, we apply the Cauchy-Schwarz inequality and  the estimate \eqref{eq:est_r1} again, along with the assumption $q_1,q_2 \in \mathcal{Q}(N)$, to deduce  
\begin{equation}
\label{eq:est_aux_term}
\begin{aligned}
&\left|h \int_Q \left(-id^\ast A+|A_1|^2_g-|A_2|^2_g+q\right) \left(\overline{\alpha_1 \beta_{A_1}} +\overline{r_1}\right)
\left(\alpha_2 \beta_{A_2}+r_2\right) dV_gdt\right| 
\\
&\le 
Ch \int_0^T \left(\|\alpha_1(t,\cdot)\|_{L^2(M)} \|\alpha_2(t,\cdot)\|_{L^2(M)}
+
\|\alpha_1(t,\cdot)\|_{L^2(M)} \|r_2(t,\cdot)\|_{L^2(M)}
\right.
\\
&\left. \quad \quad \quad \quad \quad +
\|r_1(t,\cdot)\|_{L^2(M)} \|\alpha_2(t,\cdot)\|_{L^2(M)}
+
\|r_1(t,\cdot)\|_{L^2(M)} \|r_2(t,\cdot)\|_{L^2(M)}\right)dt
\\
&\le 
Ch\|\alpha_1\|_\ast \|\alpha_2\|_\ast.
\end{aligned}
\end{equation}

We now proceed to estimate the boundary term in \eqref{eq:int_id_GO}. To this end, since $r_j|_{\Sigma}=0$, $j=1,2$, we have that
\[
u_j|_{\Sigma} = e^{\frac{i(\psi(x)-t)}{h}}\alpha_j(t,x)\beta_{A_j}(t,x), \quad (t,x) \in \Sigma,
\]
where $\alpha_j$ is given by \eqref{eq:alpha_solution} in geodesic polar coordinates. We also recall that
$\alpha_j(t,x)=0$ when $t <\varepsilon$ or $t>T-\varepsilon$ for $\varepsilon>0$ such that $T>\diam(M_1)+2\varepsilon$. Therefore, $u_j|_{\Sigma}$ is compactly supported. Furthermore, we note that the space $H^1(\Sigma):= L^2(0,T;H^1(\p M)) \cap H^1(0,T;L^2(\p M))$ is  equipped with the norm
\begin{equation}
\label{eq:H_1_norm_in_Sigma}
\|u\|_{H^1(\Sigma)}^2
:= 
\|u\|_{L^2(0,T;H^1(\p M))}^2+ \|u\|_{H^1(0,T;L^2(\p M))}^2.
\end{equation}

As in the derivation of the integral identity \eqref{eq:int_id}, we again denote 
$
f:=u_2|_{\Sigma}.
$
Then we apply \cite[Theorem 2.30]{KKL_book} to the initial boundary value problem \eqref{eq:ibvp_u2} to get $(\Lambda_{A_2,q_2}-\Lambda_{A_1,q_1})(f) \in L^2(\Sigma)$. Thus, an application of \cite[Section 5.9, Theorem 1]{Evans}
%\[
%\langle \varphi, v\rangle=(\varphi,v)_{L^2(\Sigma)}, \quad \text{if } \varphi \in L^2(\Sigma) \subset H^{-1}(\Sigma), \: v \in H^{1}_0(\Sigma).
%\]
gives us 
\begin{align*}
\left|\int_{\Sigma} h (\Lambda_{A_2,q_2}-\Lambda_{A_1,q_1})(f) \overline{u_1} dS_gdt \right|
& \le
Ch
\|(\Lambda_{A_2,q_2}-\Lambda_{A_1,q_1})(f)\|_{H^{-1}(\Sigma)} 
\|u_1\|_{H^1(\Sigma)}
\\
& \le 
C h\|\Lambda_{A_2,q_2}-\Lambda_{A_1,q_1}\|_{\mathcal{L}(L^2(\Sigma),H^{-1}(\Sigma))} \|f\|_{L^2(\Sigma)} \|u_1\|_{H^1(\Sigma)}. 
\end{align*}

The computations above suggest that we need to estimate $ \|f\|_{L^2(\Sigma)}$ and $\|u_1\|_{H^1(\Sigma)}$. To this end, we utilize the trace theorem together with \eqref{eq:H_1_norm_in_Sigma} to obtain
\begin{align*}
\|u_1\|_{H^1(\Sigma)}
&\le 
\|u_1\|_{L^2(0,T;H^2(M))}+ \|u_1\|_{H^1(0,T;H^1(M))}
\\
&\le 
C \int_{0}^T \|u_1(t,\cdot)\|_{H^2(M)}
+
\left(\|u_1(t,\cdot)\|_{H^1(M)}
+
\|\p_t u_1(t,\cdot)\|_{H^1(M)}\right) dt
\\
&\le 
C \int_{0}^T \|u_1(t,\cdot)\|_{H^2(M)}
+
\|\p_t u_1(t,\cdot)\|_{H^1(M)} dt.
\end{align*}
Since $u_1$ is the GO solution given by \eqref{eq:solution_v}, differentiating it each time will produce a factor of $h^{-1}$. Thus, we utilize the estimates \eqref{eq:est_remainder_H2} and \eqref{eq:est_time_der_remainder_H1} to see that
\begin{equation}
\label{eq:est_f}
\|u_1\|_{H^1(\Sigma)}
\le
Ch^{-2}\|\alpha_1\|_{\ast}.
\end{equation}
Furthermore, since $f=u_2|_\Sigma$, where $u_2$ is the GO solution given by \eqref{eq:solution_u2}, we argue similarly as above and apply the estimate \eqref{eq:est_r1} to conclude that
\begin{equation}
\label{eq:est_u1_boundary}
\begin{aligned}
\|f\|_{L^2(\Sigma)}
&\le
C\|u_2\|_{L^2(0,T;H^1(M))}
\\
&\le 
C\int_0^T \|u_2(t,\cdot)\|_{L^2(M)}
+
\|\nabla_g  u_2(t,\cdot)\|_{L^2(M)} dt
\\
&\le
Ch^{-1}\|\alpha_2\|_{\ast}.
\end{aligned}
\end{equation}
Therefore, \eqref{eq:est_f} and \eqref{eq:est_u1_boundary} yield the estimate
\begin{equation}
\label{eq:est_bdy_term}
\begin{aligned}
\left| h\int_{\Sigma} (\Lambda_{A_2,q_2}-\Lambda_{A_1,q_1})(f) \overline{u_1} dS_gdt \right|
\le 
Ch^{-2} \|\Lambda_{A_2,q_2}-\Lambda_{A_1,q_1}\|_{\mathcal{L}(L^2(\Sigma),H^{-1}(\Sigma))} \|\alpha_1\|_{\ast} \|\alpha_2\|_{\ast}.
\end{aligned}   
\end{equation}

Finally, the estimate \eqref{eq:int_est_amplitudes} follows immediately from the estimates \eqref{eq:est_LHS_term1}--\eqref{eq:est_aux_term} and \eqref{eq:est_bdy_term}. This completes the proof of Lemma \ref{lem:int_estimate}.
\end{proof}

For each $y \in \p M_1$ we define the respective fiber
\[
S_y^+M_1
:=
\{\theta \in S_yM_1: \n{\nu,\theta}_g\le 0\}
\]
of the bundle $\p_+SM_1$. We next establish an estimate for the geodesic ray transform of the co-vector field $A$, which is given in the following lemma. Its proof consists of applying Lemma \ref{lem:int_estimate} to the amplitude terms $\alpha$ and $\beta_A$, which are given by \eqref{eq:alpha_solution} and \eqref{eq:beta_A_solution}, respectively, expressed in geodesic polar coordinates.

\begin{lem}
\label{lem:DN_map_ray_transform}
There exists a constant $C>0$ such that the estimate
\begin{equation}
\label{eq:ray_transform_DN_Map}
\begin{aligned}
&\left|\int_{\p_+SM_1} \mathcal{I}_1(A)(y,\theta)\Psi_1(y,\theta) \Psi_2(y,\theta)dS_g(\theta) \right| 
\\
&\le 
C\left(h+h^{-2} \|\Lambda_{A_2,q_2}-\Lambda_{A_1,q_1}\|_{\mathcal{L}(L^2(\Sigma),H^{-1}(\Sigma))} \right)  \|\Psi_1\|_{H^2(S_y^+M_1)}\|\Psi_2\|_{H^2(S_y^+M_1)}
\end{aligned}
\end{equation}
holds for all functions $\Psi_i \in C^\infty(\p_+ SM_1) \cap H^2(SM_1)$ when $0 <h<1$. See Section \ref{sec:xray_transform} for the definitions of $\mathcal{I}_1$.
\end{lem}

\begin{proof}
Let  $\Psi_1,\Psi_2 \in C^\infty(\p_+ SM_1) \cap H^2(SM_1)$ and choose $\phi \in C^\infty_0(\R)$ as in the construction of the amplitude $\alpha$ in the proof of Proposition \ref{prop:GO_solutions_initial}.
Following \eqref{eq:alpha_solution}, we obtain two amplitudes 
\begin{equation}
\label{eq:form_alpha_polar}
\tilde \alpha_1 (t,r,\theta)=\rho^{-\frac{1}{4}}\phi(t-r)\Psi_1(y, \theta), \quad \tilde \alpha_2 (t,r,\theta)= \rho^{-\frac{1}{4}}\phi(t-r)\Psi_2(y, \theta).
\end{equation}
Since the function $\phi(t-r)$ is smooth and compactly supported, we get the inequality
\[
\|\alpha_i(t,\cdot)\|_{H^2(M)}
\le
\|\Psi_i\|_{H^2(S_y^+M_1)},
\quad i=1,2.
\]
We also observe that the time derivative of $\tilde \alpha_1$ only acts on the function $\phi(t-r)$, which implies that $\p_t^k\tilde \alpha_i (t,r,\theta)= \rho^{-\frac{1}{4}} \phi^{(k)}(t-r)\Psi_i(y, \theta)$ for any integer $k\ge 0$. Consequently, we obtain 
\[
\|\p_t^k\tilde \alpha_i(t,\cdot)\|_{L^2(M)}
\le
\|\Psi_i\|_{L^2(S_y^+M_1)}
\le
\|\Psi_i\|_{H^2(S_y^+M_1)}.
\]
Combining the previous two inequalities, it follows from \eqref{eq:def_star_norm} that
\begin{equation}
\label{eq:alpha1_norm}
\|\alpha_i\|_\ast \le C\|\Psi_i\|_{H^2(S_y^+M_1)}, 
\quad i=1,2.
\end{equation}

Let us now change variables on the left-hand side of the estimate \eqref{eq:int_est_amplitudes} by setting  $x=\exp_y(r\theta)$, with $r>0$ and $\theta \in S_yM_1$. By the definition of $\tilde \sigma_A$ given by \eqref{eq:sigma_A_def}, we have 
\begin{align*}
&\int_Q \n{A,d\psi}_g (\overline{\alpha_1}\alpha_2)(t,x) (\beta_{A_2}\overline{\beta_{A_1}})(t,x)dV_gdt
\\
&= \int_0^T \int_{S_y^+M_1}\int_0^{\tau_{\mathrm{exit}}(y,\theta)} \tilde \sigma_A (r,y,\theta) \left(\overline{ \tilde \alpha_1} \tilde \alpha_2\right)(t,r,\theta) 
\left(\tilde \beta_{A_2}\overline{\tilde \beta_{A_1}}\right)(t,r,\theta)\rho^{\frac{1}{2}} drd\omega_y(\theta)dt.
\end{align*}
Substituting the amplitudes $\tilde{\alpha}_1$ and $\tilde{\alpha}_2$ given by \eqref{eq:form_alpha_polar} into the right-hand side of the equation above, we get
\begin{align*}
\begin{aligned}
&\int_Q \n{A,d\psi}_g (\overline{\alpha_1}\alpha_2)(t,x) (\beta_{A_2}\overline{\beta_{A_1}})(t,x)dV_gdt
\\
&= \int_0^T \int_{S_y^+M_1}\int_0^{\tau_{\mathrm{exit}}(y,\theta)} \tilde \sigma_A (r,y,\theta) \phi^2(t-r)  (\tilde \beta_{A_2}\overline{\tilde \beta_{A_1}})(t,r,\theta) \Psi_1(y,\theta)\Psi_2(y,\theta)drd\omega_y(\theta)dt
\\
&=\int_0^T \int_{S_y^+M_1}\int_\R   \tilde \sigma_A (t-\tau,y,\theta) \phi^2(\tau) \exp\left(i\int_0^t \tilde \sigma_A(s-\tau,y,\theta)ds\right) \Psi_1(y,\theta)\Psi_2(y,\theta)d\tau d\omega_y(\theta)dt
\\
&= 
-i\int_\R \phi^2(\tau) \int_{S_y^+M_1}  \left[ \int_0^T  \frac{d}{dt} \exp\left(i\int_0^t \tilde \sigma_A(s-\tau,y,\theta)ds\right)dt \right]
\Psi_1(y,\theta)\Psi_2(y,\theta)d\omega_y(\theta)d\tau 
\\
&=
-i\int_\R \phi^2(\tau) \int_{S_y^+M_1}   \left[\exp\left(i\int_0^T \tilde \sigma_A(s-\tau,y,\theta)ds\right)-1\right]\Psi_1(y,\theta)\Psi_2(y,\theta)d\omega_y(\theta) d\tau.
\end{aligned}
\end{align*}

We  recall that $\supp \phi \subset (\varepsilon,2\varepsilon)$ for some $\varepsilon>0$ such that  $T>\diam(M_1)+3\varepsilon$. Then we have, in particular, that $\supp \phi \subset (\varepsilon,T-\varepsilon)$, and we can thus assume without loss of generality that $\|\phi\|^2_{L^2(\R)}=1$. These properties and Fubini's theorem yield  
\begin{align*}
&\int_\R \phi^2(\tau) \int_{S_y^+M_1}   \left[\exp\left(i\int_0^T \tilde \sigma_A(s-\tau,y,\theta)ds\right)-1\right]\Psi_1(y,\theta)\Psi_2(y,\theta)d\omega_y(\theta) d\tau
\\
&= \int_\R \phi^2(\tau) \int_{S_y^+M_1}   \left[\exp\left(i \int_{-\tau}^{T-\tau} \tilde \sigma_A(s,y,\theta)ds\right)-1\right]\Psi_1(y,\theta)\Psi_2(y,\theta)d\omega_y(\theta) d\tau
\\
&=
\int_\R \phi^2(\tau) \int_{S_y^+M_1}   \left[\exp\left(i\int_0^{\exit (y, \theta)} \tilde \sigma_A(s,y,\theta)ds\right)-1\right]\Psi_1(y,\theta)\Psi_2(y,\theta)d\omega_y(\theta)d\tau
\\
&=
\left(\int_\R \phi^2(\tau) d\tau\right)
\left(\int_{S_y^+M_1}   \left[\exp\left(i\int_0^{\exit (y, \theta)} \tilde \sigma_A(s,y,\theta)ds\right)-1\right]\Psi_1(y,\theta)\Psi_2(y,\theta)d\omega_y(\theta)\right)
\\
&=
\int_{S_y^+M_1}   \left[\exp\left(i\int_0^{\exit (y, \theta)} \tilde \sigma_A(s,y,\theta)ds\right)-1\right]\Psi_1(y,\theta)\Psi_2(y,\theta)d\omega_y(\theta),
\end{align*}
where we have changed the variable $s-\tau \mapsto s$ in the first step.

Therefore, we obtain from the previous  computations and the definition of the operator $\mathcal{I}_1(A)$, which is given by \eqref{eq:def_ray_transform_1form}, that   
\begin{align*}
&\int_Q \n{A,d\psi}_g (\overline{\alpha_1}\alpha_2)(t,x) \left(\beta_{A_2}\overline{\beta_{A_1}}\right)(t,x)dV_gdt
\\
&=
-i\int_{S_y^+M_1} \left[\exp\left(i\mathcal{I}_1(A)\right)-1\right]\Psi_1(y,\theta)\Psi_2(y,\theta)d\omega_y(\theta).    
\end{align*}
By Lemma \ref{lem:int_estimate}, as well as the estimate \eqref{eq:alpha1_norm}, we get for $0<h<1$ that
\begin{equation}
\label{eq:est_ray_transform_exp}
\begin{aligned}
&\left|\int_{S_y^+M_1} \left[\exp\left(i\mathcal{I}_1(A)(y,\theta)\right)-1\right] \Psi_1(y,\theta) \Psi_2(y, \theta) d\omega_y(\theta)\right| 
\\
&
\le 
C \left(h+h^{-2} \|\Lambda_{A_2,q_2}-\Lambda_{A_1,q_1}\|_{\mathcal{L}(L^2(\Sigma),H^{-1}(\Sigma))} \right) \|\Psi_1(y,\cdot)\|_{H^2(S_y^+M_1)}\|\Psi_2(y,\cdot)\|_{H^2(S_y^+M_1)}.
\end{aligned}
\end{equation}

To conclude the proof, we note that integrating \eqref{eq:est_ray_transform_exp} with respect to $y \in \p M_1$ would yield \eqref{eq:ray_transform_DN_Map}, provided that
\begin{equation}
\label{eq:exponent_estimate}
|\mathcal{I}_1(A)(y,\theta)|
\le
C  |\exp\left(i\mathcal{I}_1(A)(y,\theta)\right)-1| \quad 
\text{ for all } (y,\theta) \in \p_+ SM_1.
\end{equation}
However, this inequality is false in general, as the expression on the right-hand side vanishes when $\mathcal{I}_1(A)(y,\theta)=2k\pi$, $k\in \mathbb{Z}$. Thus, in what follows we shall argue why this cannot happen if $\|\Lambda_{A_2,q_2}-\Lambda_{A_1,q_1}\|_{\mathcal{L}(L^2(\Sigma),H^{-1}(\Sigma))}$ is small.

Assume for the moment that  
\begin{equation}
\label{eq:exp_est_pointwise}
|\exp\left(i \mathcal{I}_1(A)(y,\theta)\right)-1|<1
\quad 
\text{ for all } (y,\theta) \in \p_{+}SM.
\end{equation}
It means that there exists $\delta' \in (0,\frac{\pi}{2})$ such that for all $ (y,\theta) \in \p_{+}SM$, there exists $k \in \mathbb{Z}$ such that
\[
\mathcal{I}_1(A)(y,\theta) \in (2k\pi-\delta',2k\pi +\delta').
\]
Suppose that there is $(y,\theta_0) \in \p_{+}SM$ such that 
$
\mathcal{I}_1(A)(y,\theta_0) \geq 2\pi.
$
Since $M$ is simple, the map $\mathcal{I}_1(A)$ is continuous in $\p_+SM$. We then choose any $\theta_1 \in S_{y}^+M$ tangential to $\p M$ and a path $\theta\colon [0,1] \to S_{y}^+M$ starting from $\theta_0$ and ending at $\theta_1$. Since $\mathcal{I}_1(A)(y,\theta_1)=0$, it follows from the continuity that the map
$
s \mapsto \mathcal{I}_1(A)(A)(y,\theta(s)) 
$
obtains all values between $0$ and $2\pi$. Hence, this function obtains the value $\pi$ at some $\tilde s \in (0,1)$, where we have that
\[
|\exp\left(i \mathcal{I}_1(A)(y,\theta(\tilde s))\right)-1|=2>1,
\]
which is a contradiction. Hence, it must follow that 
\begin{equation}
\label{eq:I_1(A)_point_est}
\mathcal{I}_1(A)(y,\theta) \in \left[-\frac{\pi}{2}, \frac{\pi}{2}\right]
\quad 
\text{ for all } (y,\theta) \in \p_{+}SM,
\end{equation}
whenever the inequality \eqref{eq:exp_est_pointwise} is true. Due to compactness, the estimate  \eqref{eq:exponent_estimate} follows from \eqref{eq:I_1(A)_point_est}. 
 
We now show that the estimate \eqref{eq:est_ray_transform_exp} implies the estimate \eqref{eq:exp_est_pointwise}. The proof follows from the ideas originally presented in \cite[Section 5]{Kian2020} for the dynamic Schr\"odinger equation with time-dependent coefficients. We shall modify the arguments accordingly and provide the proof for the sake of completeness.

Let $\eta\in C^\infty_0(\R^n)$ be defined by
\[
\eta(x)=
\left\{
\begin{array}{cc}
C \exp\left(\frac{1}{|x|^2-1}\right),     & \text{ if } |x|<1 
\\
0,     & \text{ if } |x|\geq 1
\end{array}
\right.
\]
where the constant $C$ is chosen such that $\|\eta\|_{L^1(\R^n)}=1$. For each $\delta >0$, we denote
\[
\eta_\delta(x)=\frac{1}{\delta^n}\eta\left( \frac{x}{\delta} \right).
\]

Let us recall that we have extended $A$ by zero to $M_1\setminus M$. Then it follows that the function $\exp(i\mathcal{I}_1(A)(y,\theta)) - 1$ is compactly supported in $S_y^+M_1$. This guarantees the existence of a finite open cover of $\p M_1$ so that for all points in an element of the cover, we can choose spherical coordinates $\theta:\mathbb{R}^{n-1}\to S_yM_1$ such that they induce coordinates on a neighborhood of $\supp(\exp(i\mathcal{I}_1(A)(y,\theta)) - 1)$. 

Fix $y\in \p M_1$ and $\theta_0 \in \p_+ S_yM_1$, and let $\alpha_0 = \alpha(\theta_0)\in \mathbb{R}^{n-1}$ be the preimage of $\theta_0$ in our spherical coordinate system. We recall that under a change of coordinates $\theta$, the Dirac delta distribution can be written as
\[ \delta(\theta(\alpha)) = \frac{\delta(\alpha)}{|J_\theta(\alpha)|}, 
\] 
where $|J_\theta(\alpha)|=\sin^{n-2}(\alpha_1)\sin^{n-3}(\alpha_2)\cdots \sin(\alpha_{n-2})$ is the Jacobian of the spherical coordinate transformation $\theta$. It allows us to express our spherical approximation of the Dirac distribution as 
\[ \Psi_\delta(y,\theta(\alpha)) = \frac{\eta_\delta(\alpha_0-\alpha)}{\sin^{n-2}(\alpha_1)\sin^{n-3}(\alpha_2)\cdots \sin(\alpha_{n-2})}, \quad \alpha=(\alpha_1,\ldots,\alpha_{n-1})\in \mathbb{R}^{n-1}.   \]

Similarly, we define a function $f:\mathbb{R}^{n-1}\to \mathbb{C}$ by $f(\alpha) = \exp(i\mathcal{I}_1(A)(y,\theta(\alpha))) - 1$ and write
\[ f\ast \Psi_\delta(\alpha_0) = \int_{\mathbb{R}^{n-1}} f(\alpha) \Psi_\delta(\alpha_0-\alpha) d\alpha , \] 
for our spherical coordinate mollification of $f$. By the estimate \eqref{eq:est_ray_transform_exp}, with $\Psi_1=\Psi_\delta$ and $\Psi_2\equiv 1$, we have that
\begin{equation*}
\left| \int_{\mathbb{R}^{n-1}} f(\alpha) \Psi_\delta(\alpha_0-\alpha) d\alpha \right| \leq C(h + h^{-2}\|\Lambda_{A_2,q_2}-\Lambda_{A_1,q_1}\|_{\mathcal{L}(L^2(\Sigma),H^{-1}(\Sigma))}) \delta^{-(n+2)}. 
\end{equation*}
Furthermore, after possibly shrinking the diameter of the charts in our spherical coordinate patches, we can assume that $\Psi_\delta$ and $\eta_\delta$ are comparable in absolute value, which implies 
\begin{equation}\label{eq:mollifier-est}
\left| \int_{\mathbb{R}^{n-1}} f(\alpha) \eta_\delta(\alpha_0-\alpha) d\alpha \right| \leq C'(h + h^{-2}\|\Lambda_{A_2,q_2}-\Lambda_{A_1,q_1}\|_{\mathcal{L}(L^2(\Sigma),H^{-1}(\Sigma))}) \delta^{-(n+2)}. 
\end{equation}
Additionally, as in \cite[Lemma 3]{Kian2020}, we have 
\[ 
\left|\int_{\mathbb{R}^{n-1}} f(\alpha)\eta_\delta(\alpha_0-\alpha)d\alpha - f(\alpha_0)\right| \leq C\delta \|f \|_{C^1(\mathbb{R}^{n-1})}   
\] 
as a result of the mean value inequality. Combining this inequality  with \eqref{eq:mollifier-est}, we conclude that
\[ 
|\exp(i\mathcal{I}_1(A)(y,\theta(\alpha_0))) - 1| =  |f(\alpha_0)| \leq C(h + h^{-2}\|\Lambda_{A_2,q_2}-\Lambda_{A_1,q_1}\|_{\mathcal{L}(L^2(\Sigma),H^{-1}(\Sigma))}) \delta^{-(n+2)}.  
\]
Now, setting $h=\delta^{n+3}$ and $\delta = \|\Lambda_{A_2,q_2}-\Lambda_{A_1,q_1}\|_{\mathcal{L}(L^2(\Sigma),H^{-1}(\Sigma))}^{\frac{1}{3n+9}}$, we conclude that
\[ 
\left|\exp(i\mathcal{I}_1(A)(y,\theta(\alpha))) - 1\right| \leq C' 
\|\Lambda_{A_2,q_2}-\Lambda_{A_1,q_1}\|_{\mathcal{L}(L^2(\Sigma),H^{-1}(\Sigma))}^{\frac{1}{3n+9}}  
\] 
for a constant $C'$ independent of $y\in \p M_1$. Thus, once $\|\Lambda_{A_2,q_2}-\Lambda_{A_1,q_1}\|_{\mathcal{L}(L^2(\Sigma),H^{-1}(\Sigma))}$ is sufficiently small, the right-hand side above will be less than 1.  This verifies the hypothesis of \eqref{eq:exp_est_pointwise}, and completes the proof of Lemma \ref{lem:DN_map_ray_transform}.
\end{proof}

\begin{proof}[Proof of the estimate \eqref{eq:est_As}]

We are now ready to complete the derivation of the estimate \eqref{eq:est_As}. We observe that the function $h \mapsto h+h^{-2}\|\Lambda_{A_2,q_2}-\Lambda_{A_1,q_1}\|$ attains its minimum value, which is a multiple of $\|\Lambda_{A_2,q_2}-\Lambda_{A_1,q_1}\|^{\frac{1}{3}}$, 
at the point  $h_0=\sqrt[3]{2\|\Lambda_{A_2,q_2}-\Lambda_{A_1,q_1}\|}$. Thus, let us first consider the case that $h_0<\delta$, where $\delta \ll 1$, which is equivalent to the condition 
$0\le \|\Lambda_{A_2,q_2}-\Lambda_{A_1,q_1}\|<\frac{\delta^3}{2}$.

By Lemma \ref{lem:DN_map_ray_transform}, we have the estimate
\begin{equation}
\label{eq:ray_transform_DN_Map_global}
\left|\int_{\p_+SM_1} \mathcal{I}_1(A)(y,\theta)\Psi_1(y,\theta) \Psi_2(y, \theta)  dS_g(\theta)\right| 
\le 
C\|\Lambda_{A_2,q_2}-\Lambda_{A_1,q_1}\|^{\frac{1}{3}} \|\Psi_1\|_{H^2(\p_+SM_1)}\|\Psi_2\|_{H^2(\p_+SM_1)}.
\end{equation}
Let us now set $\Psi_1(y,\theta)=\mathcal{I}_1(\mathcal{N}_1(A))(y,\theta)$, where $\mathcal{N}_1 =\mathcal{I}_1^\ast \mathcal{I}_1$, and $\Psi_2(y,\theta)=\mu(y,\theta)=-\n{\theta, \nu(y)}$, which does not depend on $A$ or $q$.  As noted in Section \ref{sec:xray_transform}, the operator $\mathcal{I}_1: H^2(M,T^\ast M)\to H^2(\p_+ SM)$ is bounded. Furthermore, since $\mathcal{I}_1$ preserves smoothness, the smoothness of the function $\Psi_1$ follows from the smoothness of $A$, see \cite[Theorem 2.1]{sharafutdinov2005regularity}. Thus, we apply the estimate \eqref{eq:est_normal_solenoidal} to obtain 
\[
\|\Psi_1\|_{H^2(\p_+SM_1)}
\le 
C\| \mathcal{N}_1(A)\|_{H^2(M_1, T^\ast M_1)} 
\le 
C \|A^s\|_{H^1(M, T^\ast M)}.
\]

Due to our choice of $\Psi_1$ and $\Psi_2$, the left-hand side of the estimate \eqref{eq:ray_transform_DN_Map_global} is indeed $\|\mathcal{N}_1(A)\|_{L^2(M_1)}^2$. Thus, we get from the previous inequality that
\[
\|\mathcal{N}_1(A)\|_{L^2(M_1)}^2
\le
C\|\Lambda_{A_2,q_2}-\Lambda_{A_1,q_1}\|^{\frac{1}{3}} \|A^s\|_{H^1(M, T^\ast M)}.
\]
Furthermore, by incorporating \cite[Theorem 3.3.2]{Sharafutdinov} and the assumption $A_1,A_2 \in \mathcal{A}(\lceil \frac{n}{2}\rceil+2,N)$, we have that
\[
\|A^s\|_{H^1(M, T^\ast M)}
\le
C\|A\|_{H^1(M, T^\ast M)}
\le
CN.
\]
Applying the interpolation theorem \cite[Theorem 7.22]{Grubb}  and the estimate \eqref{eq:est_normal_solenoidal}, we see that  
\begin{align*}
\|\mathcal{N}_1(A)\|_{H^1(M_1)}
&\le
C \|\mathcal{N}_1(A)\|_{L^2(M_1)} ^{1/2} \|\mathcal{N}_1(A)\|_{H^2(M_1, T^\ast M_1)}^{1/2}
\\
&\le 
C \|\mathcal{N}_1(A)\|_{L^2(M_1)}^{1/2}
\|A^s\|_{H^{1}(M, T^\ast M)}^{1/2}
\\
&\le
C\|\Lambda_{A_2,q_2}-\Lambda_{A_1,q_1}\|^{\frac{1}{12}}.
\end{align*}
By utilizing the estimate \eqref{eq:est_transform_1_form}, we finally arrive at the estimate \eqref{eq:est_As}.

If $\|\Lambda_{A_2,q_2}-\Lambda_{A_1,q_1}\| \ge \frac{\delta^3}{2}$, we set $\sigma_1=\frac{1}{12}$ and  continuous inclusion $L^\infty(M) \hookrightarrow L^2(M)$ to obtain
\[
\|A^s\|_{L^2(M)}
\le
C\|A^s\|_{L^\infty(M)}
\le 
CN
\le
\frac{C}{\delta^{3\sigma_1}} \delta^{3\sigma_1}
\le
C\|\Lambda_{A_2,q_2}-\Lambda_{A_1,q_1}\|^{\sigma_1}.
\]
This completes the proof of estimate \eqref{eq:est_As}.
\end{proof}

\subsection{Stability estimate for the electric potential}
\label{subsec:proof_electric_hyperbolic}
In this subsection we utilize the stability estimate for $A^s$ to derive a H\"older-type estimate for the electric potential, namely,
\begin{equation}
\label{eq:est_electric}
\|q_1-q_2\|_{L^2(M)}
\le
C\|\Lambda_{A_2,q_2}-\Lambda_{A_1,q_1}\|^{\sigma_2},
\end{equation}
where the constant $C$ depends on $(M,g), N$, $n$, and $T$, and $\sigma_2 \in (0,1)$.
% \teemu{I added the constant $C$.}
We first consider the case that $0<\|\Lambda_{A_2,q_2}-\Lambda_{A_1,q_1}\|<1$.

Let us begin the proof with the derivation of an estimate for $\|A^s\|_{L^\infty(M)}$. To this end,  the Sobolev embedding theorem, in conjunction with the interpolation inequality \cite[Theorem 7.22]{Grubb},  implies that 
\begin{align*}
\|A^s\|_{L^\infty(M)} 
&\le 
C\|A^s\|_{H^{\frac{n}{2}+\eta} (M)}
\\
&\le
C\|A^s\|_{L^2(M)}^{\kappa} \|A^s\|_{H^{k'}(M)}^{1-\kappa}.
\end{align*}
Here the constants $\eta \in (0,1)$, $\frac{n}{2}+\eta<k'<\lceil \frac{n}{2}\rceil+1$, and $\kappa = 1- \frac{\frac{n}{2}+\eta}{k'}\in (0,1)$.
Since $A\in \mathcal{A}(\lceil \frac{n}{2}\rceil+1, N)$, by \cite[Theorem 3.3.2]{Sharafutdinov}, we have that
$\|A^s\|_{H^{k'}(M)} \le \|A\|_{H^{k'}(M)} \le N$. Therefore, we get that
\[
\|A^s\|_{L^\infty(M)} 
\le 
C\|A^s\|_{L^2(M)}^{\kappa}.
\]
Hence, we apply the estimate \eqref{eq:est_As} to conclude that
\begin{equation}
\label{eq:est_As_Linfty}
\|A^s\|_{L^\infty(M)}
\le
C \|\Lambda_{A_2,q_2}-\Lambda_{A_1,q_1}\|^{\zeta},
\end{equation}
where $\zeta =\frac{\kappa }{12}$.
 
Thanks to \cite[Theorem 3.3.2]{Sharafutdinov}, let us write  $A=A^s+d\varphi$, where the solenoidal part $A^s$ satisfies the property $d^\ast A^s=0$, and the function $\varphi\in H^{k+1}(M)$ vanishes on $\p M$. We also define one-forms $A_1'$ and $A_2'$ by setting $A_1'=A_1-\frac{1}{2}d\varphi$ and $A_2'=A_2+\frac{1}{2}d\varphi$. It is clear that $A_1'-A_2'=A^s$. Moreover, due to the gauge invariance of the hyperbolic Dirichlet-to-Neumann map, we have $\Lambda_{A_j,q_j}=\Lambda_{A_j',q_j}$, $j=1,2$, see for instance \cite{Bellassoued_Aicha}.

By Proposition  \ref{prop:GO_solutions_initial}, there exists a GO solution $u_1\in  C^2(0,T;L^2(M)) \cap C(0,T;H^2(M))$ to the equation $\mathcal{H}_{g,A_1',q_1} u_1=0$ 
in $Q$ given by
\begin{equation}
\label{eq:solution_u1_q}
u_1(t,x)=e^{\frac{i(\psi(x)-t)}{h}}\alpha_1(t,x)\beta_{A_1'}(t,x)+r_1(t,x),
\end{equation}
as well as a GO solution $u_2\in  C^2(0,T;L^2(M)) \cap C(0,T;H^2(M))$ to the equation $\mathcal{H}_{g,A_2',q_2}u_2=0$ in $Q$, which is of the form
\begin{equation}
\label{eq:solution_u2_q}
u_2(t,x)=e^{\frac{i(\psi(x)-t)}{h}}\alpha_2(t,x)\beta_{A_2'}(t,x)+r_2(t,x).
\end{equation}
Here the remainders $r_1$ and $r_2$ satisfy the estimate \eqref{eq:est_r1}. Let us again denote $f=u_2|_{\Sigma}$.

We next substitute the GO solutions above into the integral identity \eqref{eq:int_id} and analyze the resulted expression when $0<h<1$. Different from  Subsection \ref{subsec:proof_solenoidal_part_A}, we shall isolate the term involving the electrical potential $q$ on the left-hand side. 

Our first main result in this subsection is as follows.
\begin{lem}
\label{lem:est_int_id_q}
There exists a constant $C>0$ such that for any amplitudes $\alpha_j$, $j=1,2$, which satisfy the transport equation \eqref{eq:transport_eq_alpha}, the estimate
\begin{equation}
\label{eq:int_est_amplitudes_q}
\left|\int_Q q (\overline{\alpha_1}\alpha_2) dV_gdt\right| 
\le 
C\left(h+h^{-3} \|\Lambda_{A_2,q_2}-\Lambda_{A_1,q_1}\|^\zeta \right) \|\alpha_1\|_{\ast} \|\alpha_2\|_{\ast}
\end{equation}
holds for all $0<h<1$. Here $\zeta \in(0,1)$ is the same as in the estimate \eqref{eq:est_As_Linfty}.
\end{lem}

\begin{proof}
By rearranging the terms and replacing $A$ with $A^s$  in \eqref{eq:int_id} , we have a new integral identity
\begin{equation}
\label{eq:int_id_q}
\begin{aligned}
\int_Q q \overline{u_1}u_2 dV_gdt
&= 
\int_Q 2i\n{A^s, du_2}_g \overline{u_1} dV_gdt
-
\int_Q (-id^\ast A^s+|A_1'|^2_g-|A_2'|^2_g) \overline{u_1} u_2 dV_gdt
\\
&+
\int_{\Sigma} (\Lambda_{A_2',q_2}-\Lambda_{A_1',q_1})(f) \overline{u_1} dS_gdt
\end{aligned}
\end{equation}
for any GO solutions $u_1$ and $u_2$ given by 
\eqref{eq:solution_u1_q} and \eqref{eq:solution_u2_q}.  
We next substitute these solutions into the left-hand side of \eqref{eq:int_id_q}. Using similar computations as in \eqref{eq:computation_du2u1}, as well as the fact that $d^\ast A^s=0$, 
we obtain  
\begin{equation}
\label{eq:int_id_computation_q}
\begin{aligned}
&\int_Q q (\overline{\alpha_1}\alpha_2) (\beta_{A_2'}\overline{\beta_{A_1'}}) dV_gdt
\\
&= 
\int_Q 2h^{-1}\n{A^s,d\psi}_g(\overline{\alpha_1}\alpha_2)(\beta_{A_2'}\overline{\beta_{A_1'}}) dV_gdt 
+
\int_Q  2 h^{-1}\overline{r}_1 e^{\frac{i(\psi(x)-t)}{h}} \n{A^s,d\psi}_g (\alpha_2\beta_{A_2'}) dV_gdt
\\
&-
\int 2i(\overline{\alpha_1\beta_{A_1'}})\n{A^s,d(\alpha_2\beta_{A_2'})}_g dV_gdt
-
\int_Q 2i\overline{r}_1\n{A^s,d(\alpha_2\beta_{A_2'})}_g dV_gdt 
\\
&+
\int_Q 2ie^{\frac{i(\psi(x)-t)}{h}} \n{A^s,dr_2}_g(\overline{\alpha_1\beta_{A_1'}}) dV_gdt
+
\int_Q 2i\overline{r}_1\n{A^s, dr_2}_g dV_gdt
\\
&+
\int_Q (|A_1'|^2_g-|A_2'|^2_g) \left(\overline{\alpha_1 \beta_{A_1'}} +\overline{r_1}\right)
\left(\alpha_2 \beta_{A_2'}+r_2\right) dV_gdt
\\
&-
\int_Q q \left(\overline{\alpha_1 \beta_{A_1'}} r_2+\alpha_2 \beta_{A_2'} \overline{r}_1 + \overline{r}_1 r_2 \right) dV_gdt
+
\int_{\Sigma} (\Lambda_{A_2',q_2}-\Lambda_{A_1',q_1})(f) \overline{u_1} dS_gdt.
\end{aligned}
\end{equation}

Let us next analyze each term on the right-hand side of \eqref{eq:int_id_computation_q}. To this end, we apply the Cauchy-Schwarz inequality, along with the fact that $|\beta_{A'_i}|=1$ (see \eqref{eq:beta_A_solution}), to obtain
\begin{equation}
\label{eq:est_term1_q}
\begin{aligned}
\left|\int_Q 2h^{-1}\n{A^s,d\psi}_g(\overline{\alpha_1}\alpha_2)(\beta_{A_2'}\overline{\beta_{A_1'}}) dV_gdt  \right|
&\le
Ch^{-1}\|A^s\|_{L^\infty(M)}\int_0^T \|\alpha_1(t,\cdot)\|_{L^2(M)} \|\alpha_2(t,\cdot)\|_{L^2(M)} dt
\\
&\le
Ch^{-1} \|A^s\|_{L^\infty(M)} \|\alpha_1\|_\ast \|\alpha_2\|_\ast.
\end{aligned}
\end{equation}
Proceeding as in the estimates \eqref{eq:est_LHS_term1}--\eqref{eq:est_LHS_term5}, we get the following upper bounds for the terms below:
\begin{equation}
\label{eq:est_term2_q}
\left|\int_Q  2 h^{-1}\overline{r}_1 e^{\frac{i(\psi(x)-t)}{h}} \n{A^s,d\psi}_g (\alpha_2\beta_{A_2'}) dV_gdt\right|
\le
C  \|A^s\|_{L^\infty(M)} \|\alpha_1\|_\ast \|\alpha_2\|_\ast,
\end{equation}
\begin{equation}
\label{eq:est_term3_q}
\left|\int_Q 2i \left(\overline{\alpha_1\beta_{A_1'}}\right)
\n{A^s,d(\alpha_2\beta_{A_2'})}_g dV_gdt\right|
\le
C \|A^s\|_{L^\infty(M)} \|\alpha_1\|_\ast \|\alpha_2\|_\ast,
\end{equation}
\begin{equation}
\label{eq:est_term4_q}
\left|\int_Q 2i \overline{r}_1\n{A^s,d(\alpha_2\beta_{A_2'})}_g dV_gdt  \right|
\le
Ch \|A^s\|_{L^\infty(M)} \|\alpha_1\|_\ast \|\alpha_2\|_\ast,
\end{equation}
\begin{equation}
\label{eq:est_term5_q}
\left| \int_Q 2ie^{\frac{i(\psi(x)-t)}{h}} \n{A^s,dr_2}_g(\overline{\alpha_1\beta_{A_1'}}) dV_gdt \right|
\le
C  \|A^s\|_{L^\infty(M)} \|\alpha_1\|_\ast \|\alpha_2\|_\ast,
\end{equation}
and
\begin{equation}
\label{eq:est_term6_q}
\left| \int_Q 2i\overline{r}_1\n{A^s, dr_2}_g dV_gdt \right|
\le
Ch  \|A^s\|_{L^\infty(M)} \|\alpha_1\|_\ast \|\alpha_2\|_\ast.
\end{equation}

Turning attention to the integral involving $|A_1'|^2_g-|A_2'|^2_g$,  we observe that
\[
|A_1'|^2_g-|A_2'|^2_g = \n{A_1'+A_2', A_1'-A_2'}_g
=
\n{A_1+A_2, A^s}_g.
\]
Thus, applying the same reasoning as  in the estimate \eqref{eq:est_aux_term}, we have that
\begin{equation}
\label{eq:est_term7_q}
\begin{aligned}
\left| \int_Q (|A_1'|^2_g-|A_2'|^2_g) \left(\overline{\alpha_1 \beta_{A_1'}} +\overline{r_1}\right)
\left(\alpha_2 \beta_{A_2'}+r_2\right) dV_gdt \right|
&\le 
C \|A^s\|_{L^\infty(M)} \|\alpha_1\|_{\ast} \|\alpha_2\|_{\ast}.
%\\
%&\le
%C \|\Lambda_{A_2,q_2}-\Lambda_{A_1,q_1}\|^{\sigma} 
%\|\alpha_1\|_{\ast} \|\alpha_2\|_{\ast}.
\end{aligned}
\end{equation}

To analyze the integral involving $q$, by utilizing the estimate \eqref{eq:est_r1}, along with the Cauchy-Schwarz inequality, we obtain
\begin{equation}
\label{eq:est_term8_q}
\begin{aligned}
& \left| \int_Q q \left(\overline{\alpha_1 \beta_{A_1'}} r_2+\alpha_2 \beta_{A_2'} \overline{r}_1 + \overline{r}_1 r_2 \right)dV_gdt\right|
\\
&\le
C\int_0^T \left( 
\|\alpha_1(t,\cdot)\|_{L^2(M)} \|r_2(t,\cdot)\|_{L^2(M)}
+
\|r_1(t,\cdot)\|_{L^2(M)} \|\alpha_2(t,\cdot)\|_{L^2(M)}\right.
\\
&\left. \quad \quad \quad \quad 
+
\|r_1(t,\cdot)\|_{L^2(M)} \|r_2(t,\cdot)\|_{L^2(M)}\right)dt
\\
&\le
Ch\|\alpha_1\|_\ast \|\alpha_2\|_\ast.
\end{aligned}
\end{equation}

Finally, to handle the boundary term, we adopt the same arguments between the estimates \eqref{eq:est_aux_term} and \eqref{eq:est_bdy_term}, as well as  the fact that $\Lambda_{A_j',q_j}=\Lambda_{A_j,q_j}$, $j=1,2$, to see that
\begin{equation}
\label{eq:est_term9_q}
\left|\int_{\Sigma} (\Lambda_{A_2',q_2}-\Lambda_{A_1',q_1})(f) \overline{u_1} dS_gdt\right|
\le
h^{-3} \|\Lambda_{A_2,q_2}-\Lambda_{A_1,q_1}\| \|\alpha_1\|_{\ast} \|\alpha_2\|_{\ast}.
\end{equation}
Therefore, by combining the estimates \eqref{eq:est_As_Linfty}  and \eqref{eq:est_term1_q}--\eqref{eq:est_term9_q}, we conclude that for $0<h<1$ we have
\begin{equation}
\label{eq:est_combine_q}
\begin{aligned}
\left|\int_Q q  (\overline{\alpha_1}\alpha_2)  \left(\beta_{A_2'}\overline{\beta_{A_1'}}\right) dV_gdt\right| 
&\le 
C\left(h+h^{-1}\|A^s\|_{L^\infty(M)}+h^{-3} \|\Lambda_{A_2,q_2}-\Lambda_{A_1,q_1}\| \right) \|\alpha_1\|_{\ast} \|\alpha_2\|_{\ast}
\\
&\le
C\left(h+h^{-3} \|\Lambda_{A_2,q_2}-\Lambda_{A_1,q_1}\|^\zeta \right) \|\alpha_1\|_{\ast} \|\alpha_2\|_{\ast}.
\end{aligned}
\end{equation}
where $\zeta\in (0,1)$ is the same as in \eqref{eq:est_As_Linfty}.

We next show that \eqref{eq:est_combine_q} still holds if  $ \beta_{A_2'}\overline{\beta_{A_1'}} $ is removed from the integral. To this end, let us write
\[
\int_Q q\overline{\alpha_1}\alpha_2 dV_gdt 
= 
\int_Q q\overline{\alpha_1}\alpha_2 \left(1-\beta_{A_2'}\overline{\beta_{A_1'}}\right) dV_gdt
+ 
\int_Q q\overline{\alpha_1}\alpha_2 \beta_{A_2'}\overline{\beta_{A_1'}} dV_g dt
\]
and observe from \eqref{eq:beta_A_solution} that
$\beta_{A_2'}\overline{\beta_{A_1'}} = \beta_{A^s}$.

Let $a\in \R$. By Taylor's expansion of the function $s \mapsto e^{ias}$ at $0$, we have that
\[
e^{ias}=1+ias-\frac{1}{2}a^2s^2\int_{0}^1e^{ia\rho}(1-s)^2d\rho.
\]
Hence, after choosing $s=1$ and
$
a=\int_{0 }^{t} \tilde \sigma_{A^s}(r-s, y,\theta)ds,
$
where $t\in (0,T), \; r\in(0,\exit(y,\theta))$, and $(y,\theta) \in \p_+SM_1,$
we get from \eqref{eq:beta_A_solution} and the estimate
\[
\left|\int_{0 }^{t} \tilde \sigma_{A^s}(r-s, y,\theta)ds\right|
% \leq 
% \int_0^T|\sigma_{A^s}(r-s, y,\theta)|ds
\leq 
T\|A^s\|_{L^\infty(M,T^\ast M)}
\leq T\|A\|_{H^{\lceil \frac{n}{2}\rceil+1}(M,T^\ast M)}\le TN,
\]
that 
\[
\|1-\beta_{A^s}\|_{L^\infty((0,T) \times \p_+SM_1)} \le C\|A^s\|_{L^\infty(M,T^\ast M)},
\]
where $C$ only depends on $T,$ $(M,g)$, and $N$.
Therefore, by the Cauchy-Schwarz inequality and the estimate \eqref{eq:est_As_Linfty}, we have that
\begin{equation}
\label{eq:est_difference_amplitude}
\begin{aligned}
\left|\int_Q q \overline{\alpha_1}\alpha_2  \left(1-\beta_{A_2'}\overline{\beta_{A_1'}}\right) dV_gdt\right|
&\le
C\|1-\beta_{A^s}\|_{L^\infty((0,T) \times \p_+SM_1)}
\int_0^T
\|\alpha_1(t,\cdot)\|_{L^2(M)} \|\alpha_2(t,\cdot)\|_{L^2(M)} dt
\\
&\le
C\|\Lambda_{A_2,q_2}-\Lambda_{A_1,q_1}\|^\zeta  \|\alpha_1\|_{\ast} \|\alpha_2\|_{\ast}.
\end{aligned}
\end{equation}

Finally, the estimate \eqref{eq:int_est_amplitudes_q} follows immediately by combining \eqref{eq:est_combine_q} and \eqref{eq:est_difference_amplitude}. This completes the proof of Lemma \ref{lem:est_int_id_q}.
\end{proof}

We next substitute the amplitudes $\alpha_j$, $j=1,2$, given by \eqref{eq:alpha_solution} in geodesic polar coordinates, into the left-hand side of the estimate \eqref{eq:int_est_amplitudes_q}. This yields the following estimate for the geodesic ray transform of the electric potential $q$.

\begin{lem}
\label{lem:DN_map_ray_transform_q}
Let $y\in \p M_1$. Then for any function $\Psi_i \in H^2(\p_+SM_1) \cap C^\infty(\p_+SM_1)$, $i=1,2$,
there exists a constant $C>0$ such that the estimate 
\begin{equation}
\label{eq:ray_transform_DN_Map_q}
\begin{aligned}
&\left|\int_{\p_+SM_1}  \mathcal{I}_0(q)(y,\theta)\Psi_1(y,\theta) \Psi_2 (y,\theta) d\omega_y(\theta)\right| 
\\
&\le 
C\left(h+h^{-3} \|\Lambda_{A_2,q_2}-\Lambda_{A_1,q_1}\|^\zeta\right)
\|\Psi_1\|_{H^2(\p_+ SM_1)}
\|\Psi_2\|_{H^2(\p_+ SM_1)}    
\end{aligned}
\end{equation}
holds for $0<h<1$. Here
% the set  $S_y^+M_1$ is defined in the same way as in Lemma \ref{lem:DN_map_ray_transform}, and 
the constant $\zeta \in (0,1)$ is the same as in the estimate \eqref{eq:est_As_Linfty}. 
\end{lem}

\begin{proof}
After choosing $y \in \p M_1$ and setting $x=\exp_y(r\theta)$, where $r>0$ and $\theta \in S_yM_1$, the left-hand side of the estimate \eqref{eq:int_est_amplitudes_q} reads
\[
\int_Q q(x) (\overline{\alpha_1}\alpha_2)(t,x)  dV_gdt
=
\int_0^T \int_{S_y^+M_1} \int_0^{\tau_{\mathrm{exit}}(y,\theta)} \tilde q(r,y,\theta) (\overline{ \tilde \alpha_1} \tilde \alpha_2)(t,r,\theta) \rho^{\frac{1}{2}} drd\omega_y(\theta)dt.
\]
Similar to the proof of Lemma \ref{lem:DN_map_ray_transform}, we have two amplitudes $\tilde \alpha_1$ and $\tilde \alpha_2$ given by \eqref{eq:form_alpha_polar}, which satisfy the estimate \eqref{eq:alpha1_norm}.  By substituting them into the right-hand side of the equation above, we have that
\begin{align*}
\int_Q q(x) (\overline{\alpha_1}\alpha_2)(t,x) dV_gdt
&= \int_0^T \int_{S_y^+M_1}\int_0^{\tau_{\mathrm{exit}}(y,\theta)} \tilde q (r,y,\theta) \phi^2(t-r)  \Psi_1(y,\theta)\Psi_2(y,\theta)drd\omega_y(\theta)dt.
\end{align*}

Let us recall that
$\supp \phi \subset (\varepsilon, 2\varepsilon)$ for some constant $\varepsilon>0$ such that  $T>\diam(M_1)+3\varepsilon$.
Then it follows from Fubini's theorem that
\begin{align*}
&\int_0^T \int_{S_y^+M_1}\int_0^{\tau_{\mathrm{exit}}(y,\theta)} \tilde q (r,y,\theta) \phi^2(t-r)  \Psi_1(y,\theta)\Psi_2(y,\theta)drd\omega_y(\theta)dt 
\\
&=
\int_{S_y^+M_1}\int_0^{\tau_{\mathrm{exit}}(y,\theta)} \tilde q (r,y,\theta) \Psi_1(y,\theta)\Psi_2(y,\theta) \left( \int_0^T \phi^2(t-r) dt \right) drd\omega_y(\theta).
\end{align*}
We next perform a change of variables $t-r\mapsto t$ and utilize the supporting properties of $\phi$ to get that
\[
\int_0^T \phi^2(t-r) dt = \int_{-r}^{T-r} \phi^2(t) dt = \int_\R \phi^2(t)dt.
\]
Let us again assume without loss of generality that $\|\phi\|^2_{L^2(\R)}=1$. Then we conclude that
\begin{align*}
&\int_0^T \int_{S_y^+M_1}\int_0^{\tau_{\mathrm{exit}}(y,\theta)} \tilde q (r,y,\theta) \phi^2(t-r)  \Psi_1(y,\theta)\Psi_2(y,\theta)drd\omega_y(\theta)dt
\\	
&= \int_{S_y^+M_1}\int_0^{\tau_{\mathrm{exit}}(y,\theta)} \tilde q (r,y,\theta) \Psi_1(y,\theta)\Psi_2(y,\theta)drd\omega_y(\theta)
\\
&=
\int_{S_y^+M_1} \mathcal{I}_0(q)(y,\theta) \Psi_1(y,\theta) \Psi_2(y,\theta)d\omega_y(y,\theta).
\end{align*}
From here, we apply Lemma \ref{lem:est_int_id_q}, as well as the estimate \eqref{eq:alpha1_norm}, to get that 
\begin{align*}
&\int_{S_y^+M_1} \mathcal{I}_0(q)(y,\theta)\Psi_1(y,\theta) \Psi_2(y,\theta) d\omega_y(\theta)
\\
&\le
C\left(h+h^{-3} \|\Lambda_{A_2,q_2}-\Lambda_{A_1,q_1}\|^\zeta\right)
\|\Psi_1(y,\cdot)\|_{H^2(S_y^+M_1)}
\|\Psi_2(y,\cdot)\|_{H^2(S_y^+M_1)}.    
\end{align*}

Finally, we integrate both sides of the inequality above with respect to $y\in \p M_1$ to obtain the estimate \eqref{eq:ray_transform_DN_Map_q}. This completes the proof of Lemma \ref{lem:DN_map_ray_transform_q}.
\end{proof}

\begin{proof}[Proof of the estimate \eqref{eq:est_electric}]
We are now ready to complete the derivation of the estimate \eqref{eq:est_electric}. Let us first minimize the right-hand side of the estimate \eqref{eq:ray_transform_DN_Map_q} with respect to the variable $h$.
To this end, we note that the minimum value of  the function $h \mapsto h+h^{-3} \|\Lambda_{A_2,q_2}-\Lambda_{A_1,q_1}\|^\zeta$ is a multiple of $\|\Lambda_{A_2,q_2}-\Lambda_{A_1,q_1}\|^{\frac{\zeta}{4}}$, which is achieved when $h_0=\sqrt[4]{3\|\Lambda_{A_2,q_2}-\Lambda_{A_1,q_1}\|^\zeta}$. Hence, we first consider the case that $h_0<\delta$, where $\delta \ll 1$, which is equivalent to the condition 
$0\le \|\Lambda_{A_2,q_2}-\Lambda_{A_1,q_1}\| <
\left[\frac{(\delta')^4}{3}\right]^{\frac{1}{\zeta}}$. 

By Lemma \ref{lem:DN_map_ray_transform_q}, we get that
\begin{equation}
\label{eq:est_transform_q}
\left|\int_{\p_+SM_1} \mathcal{I}_0(q)(y,\theta)\Psi_1(y,\theta) \Psi_2(y,\theta) d\omega_y(\theta)\right| 
\le 
C\|\Lambda_{A_2,q_2}-\Lambda_{A_1,q_1}\|^{\frac{\zeta}{4}} \|\Psi_1\|_{H^2(\p_+ SM_1)}
\|\Psi_2\|_{H^2(\p_+ SM_1)}.
\end{equation}
We now choose $\Psi_1(y,\theta)=\mathcal{I}_0(\mathcal{N}_0(q))(y,\theta)$ and $\Psi_2(y,\theta)=\mu(y,\theta)=-\n{\theta, \nu(y)}$, where $\mathcal{N}_0 =\mathcal{I}_0^\ast \mathcal{I}_0$, and $\mathcal{I}_0$ is the geodesic ray transform of functions defined by the formula \eqref{eq:def_ray_transform_functions}. With this choice of $\Psi_1$ and $\Psi_2$, the left-hand side of \eqref{eq:est_transform_q} becomes $\|\mathcal{N}_0(q)\|_{L^2(M_1)}^2$.  Since $\mathcal{I}_0: H^2(M_1) \to H^2(\p_+SM_1)$ is a bounded operator, we apply the estimate \eqref{eq:est_normal_q} to see that
\[
\|\Psi\|_{H^2(\p_+SM_1)}
\le 
C\|\mathcal{N}_0(q)\|_{H^2(M_1)} 
\le 
C \|q\|_{H^1(M)}.
\]
Hence, the estimate \eqref{eq:est_transform_q} and the fact that $q\in \mathcal{Q}(N)$ yield 
\begin{align*}
\|\mathcal{N}_0(q)\|_{L^2(M_1)}^2
&\le
C\|\Lambda_{A_2,q_2}-\Lambda_{A_1,q_1}\|^{\frac{\zeta}{4}} \|q\|_{H^1(M)}
\\
&\le 
C\|\Lambda_{A_2,q_2}-\Lambda_{A_1,q_1}\|^{\frac{\zeta}{4}}.
\end{align*}

In order to apply the estimate \eqref{eq:est_transform_function}, we need to first derive an estimate for $\|\mathcal{N}_0(q)\|_{H^1(M_1)}$. To achieve this, by the interpolation inequality and  the estimate \eqref{eq:est_normal_q}, we deduce that
\begin{align*}
\|\mathcal{N}_0(q)\|_{H^1(M_1)}^2
&\le
C \|\mathcal{N}_0(q)\|_{L^2(M_1)}
\|\mathcal{N}_0(q)\|_{H^2(M_1)}
\\
&\le 
C \|\mathcal{N}_0(q)\|_{L^2(M_1)}
\|q\|_{H^1(M)}
\\
&\le
C\|\Lambda_{A_2,q_2}-\Lambda_{A_1,q_1}\|^{\frac{\zeta}{8}}.
\end{align*}
Therefore, by the estimate \eqref{eq:est_transform_function}, we get  
\begin{equation}
\label{eq:Holder_electric}
\|q\|_{L^2(M)}
\le 
C\|\Lambda_{A_2,q_2}-\Lambda_{A_1,q_1}\|^{\frac{\zeta}{16}}.
\end{equation}

When $\|\Lambda_{A_2,q_2}-\Lambda_{A_1,q_1}\| \ge \left[\frac{(\delta')^4}{3}\right]^{\frac{1}{\zeta}}$, due to the continuous inclusion $L^\infty(M) \hookrightarrow L^2(M)$, we have that
\[
\|q\|_{L^2(M)}
\le
C\|q\|_{L^\infty(M)}
\le 
CN
\le
\frac{CN}{(\delta')^{\sigma_2}} (\delta')^{\sigma_2}
\le
\frac{CN}{(\delta')^{\sigma_2}} \|\Lambda_{A_2,q_2}-\Lambda_{A_1,q_1}\|^{\sigma_2},
\]
where $\sigma_2=\frac{\zeta}{16}$. This completes the proof of the estimate \eqref{eq:est_electric}.
\end{proof}

Finally, we obtain the estimate \eqref{eq:est_X_and_q} by combining the estimates \eqref{eq:est_As} and \eqref{eq:est_electric}. The proof of Theorem \ref{thm:main_result_hyperbolic} is now complete.

\section{Proof of Theorem \ref{thm:spectral_problem}}
\label{sec:proof_spectral}

This section is devoted to the proof of Theorem \ref{thm:spectral_problem}. Towards this goal, in Subsection \ref{subsec:relation_bsp_DN_map} we first connect the boundary spectral data to the hyperbolic Dirichlet-to-Neumann map, which will be accomplished via a family of elliptic Dirichlet-to-Neumann maps defined below. Then we show in Subsection \ref{subsec:norm_estimates} that the norm of the hyperbolic Dirichlet-to-Neumann map is bounded above by the norm of the spectral data.

Let us  begin  by defining a family of elliptic Dirichlet-to-Neumann maps. 
Recall that $\sigma(\mathcal{E}_{g,A,q})= \{\lambda_{k,A,q}:\,k=1,2,\cdots\}$ is the spectrum of the magnetic Schr\"odinger operator \eqref{eq:def_mag_Schro},
% \[
% \mathcal{E}_{g,A,q}=-\Delta_{g,A}+q,
% \]
while $\varphi_{k,A,q}$ is the (Dirichlet) eigenfunction corresponding to the eigenvalue $\lambda_{k,A,q}$. We also reserve the notation  $\psi_{k,A,q} = \LA d_A\varphi_{k,A,q},\nu \RA|_{\p M}$ for the Neumann boundary values of the eigenfunctions. 
For simplicity, when we only work with one set of coefficients $(A,q)$, we shall denote 
\[
\lambda_k:= \lambda_{k,A,q}, \quad \varphi_{k}: = \varphi_{k,A,q}, \quad \psi_k:=\psi_{k,A,q}.
\]

Let $\rho(\mathcal{E}_{g,A,q}) = \C \setminus \sigma(\mathcal{E}_{g,A,q})$ be the resolvent set of $\mathcal{E}_{g,A,q}$. 
Then for any complex number $z \in \rho(\mathcal{E}_{g,A,q})$ and any function $f \in H^{1/2}(\p M)$, the boundary value problem
\begin{equation}
\label{eq:bvp_resolvent}
\begin{cases}
(\mathcal{E}_{g,A,q}-z) u = 0 &\text{ in } M,
\\
u=f &\text{ on } \p M, 
\end{cases}
\end{equation}
admits a unique solution $u \in H^1(M)$, see \cite[Theorem 23.4]{Eskin2011}. We shall define the corresponding elliptic Dirichlet-to-Neumann map $\Pi_{A,q}(z): H^{\frac{1}{2}}(\p M) \to H^{-\frac{1}{2}}(\p M)$ by the formula
\begin{equation}
\label{eq:def_elliptic_DNmap}
\Pi_{A,q}(z)(f) = \LA d_Au,\nu \RA_g.
\end{equation}
Here $\LA d_Au,\nu \RA_g$ is defined in the weak sense:
\[
\LA \LA d_Au,\nu \RA_g,h \RA_{H^{-\frac{1}{2}}(\p M),H^{1/2}(\p M)}:  = (d_A u, \overline{d_A v} )_{L^2(M)} + ( (q-z)u,\overline{v} )_{L^2(M)},  \quad h\in H^{1/2}(\p M),
\]
where $v\in H^1(M)$ is an extension of $h$. Let us remark that this definition is independent of the choice of $v$ since $u$ solves \eqref{eq:bvp_resolvent}. For each  $-1<s<-\frac{1}{2}$, thanks to \cite[Chapter 4, Proposition 3.4]{Taylor_book}, we compactly
embed $H^{-\frac{1}{2}}(\p M)$ into $H^{s}(\p M)$, and view the elliptic Dirichlet-to-Neumann map as a bounded linear operator $\Pi_{A,q}: H^{1/2}(\p M)\to H^s(\p M)$. The operator norm is denoted by $\|\cdot\|_{\frac{1}{2},s}$.

\subsection{Connection between boundary spectral data and hyperbolic Dirichlet-to-Neumann map}
\label{subsec:relation_bsp_DN_map}

In this subsection we establish the relationship between the boundary spectral data and the hyperbolic Dirichlet-to-Neumann map through two key lemmas. We first show in Lemma \ref{lem:DN_map_decomposition} that the spectral data can be expressed in terms of the elliptic Dirichlet-to-Neumann map $\Pi_{A,q}(z)$ defined above. Afterwards, we derive an explicit formula that relates the elliptic and hyperbolic Dirichlet-to-Neumann maps in Lemma \ref{lem_DN_E_H}. 

\begin{rem}
\label{rem:positivity_of_q}
In the proof of Theorem \ref{thm:spectral_problem}, we need to incorporate an additional assumption that the operators $\mathcal{E}_{g,A_\ell,q_\ell}$ are positive definite for both $\ell\in \{1,2\}$. However, this does not conflict with the claim of the theorem, as the assumption $q_1,q_2 \in \mathcal{Q}(N)$ implies that $q_\ell+N \geq 0$, and the operators $\mathcal{E}_{g,A_\ell,q_\ell+N}$ are positive definite with eigenvalues $(\lambda_{\ell,k}+N)_{k\in \N}$ and eigenfunctions $(\varphi_{\ell,k})_{k\in \N}$. 
Thus, the difference of the spectral data of the operators $\mathcal{E}_{g,A_1,q_1}$ and $\mathcal{E}_{g,A_2,q_2}$ is equivalent to that of the shifted operators $\mathcal{E}_{g,A_1,q_1+N}$ and $\mathcal{E}_{g,A_2,q_2+N}$. Hence, throughout this section we may assume without loss of generality that $q_\ell$ is non-negative and that the operator $\mathcal{E}_{g,A_\ell,q_\ell}$ is positive definite.
\end{rem}

The following lemma extends \cite[Lemma 2.28]{Choulli_book}, which was given for the Schr\"odinger operator $-\Delta +q$, as we need the corresponding result for the magnetic Schr\"odinger operator $\mathcal{E}_{g,A,q}$. 

\begin{lem}
\label{lem:DN_map_decomposition}
Let $A\in W^{1,\infty}(M, T^\ast M)$, $q\in L^\infty(M, \R)$ be non-negative, and $z\in \rho(\mathcal{E}_{g,A,q})$. Then for any integer $m > \frac{n}{2}+1$,
% and all functions $f\in H^{1/2}(\p M)$, 
we have that
\[
\LC\frac{d^m}{d z^m} \Pi_{A,q}(z)\RC f
= -
m! \sum_{k=1}^\infty \frac{(f,  \psi_k)_{L^2(\p M)}}{(\lambda_k-z)^{m+1}} \psi_k,
\]
where the map $z \mapsto \Pi_{A,q}(z)(f)$ is $H^{-\frac{1}{2}}(\p M)$-valued for each function $f\in H^{\frac{1}{2}}(\p M)$.
\end{lem}

\begin{proof}
We use the notation $R_{A,q}:  \rho(\mathcal{E}_{g,A,q}) \to \mathcal{L}(L^2(M), H^2(M))$ for the resolvent operator $R_{A,q}(z)=(\mathcal{E}_{g,A,q} -z)^{-1}$
of the boundary value problem
\[
\begin{cases}
(\mathcal{E}_{g,A,q}-z) v = h &\text{ in } M,
\\
v=0 &\text{ on } \p M, 
\end{cases}
\]
where $h\in L^2(M)$. By following the same arguments as in \cite[Proposition 2.30]{Choulli_book}, we have that
\begin{equation}
\label{eq:resolvent_formula}
v = R_{A,q}(z)h = \sum_{k=1}^\infty\frac{(h,\varphi_k)_{L^2(M)}}{\lambda_k-z}\varphi_k, 
\quad \text{ for all } h \in L^2(M).
\end{equation}

Let $f\in H^{1/2}(\p M)$, and let  $F\in H^{1}(M)$ be a weak solution of the boundary value problem
\begin{equation}
\label{eq:eq_F}
\begin{cases}
-\Delta_{g,A}F = 0 & \text{ in }  M,
\\
F = f & \text{ on }  \p M.
\end{cases}
\end{equation}
The solution $F$ exists since the magnetic Laplacian $\Delta_{g,A}$ is positive definite. Then the unique solution $u(z)$ of the boundary value problem \eqref{eq:bvp_resolvent} can be represented as 
\begin{equation}\label{eq_5_Lemma51_u}
u(z) = F-R_{A,q}(z)\left((q-z)F\right).
\end{equation}

We note that the map $z \mapsto u(z)$, as in \eqref{eq_5_Lemma51_u}, is holomorphic on the resolvent set for the Banach space $L^2$  as defined in \cite[Definition 3.30]{rudin1991functional}. 
Thus, the map $u^{(m)}(z) := \frac{d^m}{dz^m}u(z)$ is well defined for all $m\ge 0$. By differentiating the equation \eqref{eq:bvp_resolvent} with respect to $z$, we see that, for any integer  $m\ge 1$, the function $u^{(m)}(z)$ satisfies the boundary value problem
\[
\begin{cases}
(\mathcal{E}_{g,A,q}-z)u^{(m)} = mu^{(m-1)} & \text{ in }  M,
\\
u^{(m)} = 0 &  \text{ on }  \p M.
\end{cases}
\]
Therefore, we can apply the resolvent operator iteratively to write  
\begin{equation}
\label{eq_5_lemma51_u^m}
u^{(m)}
= 
R_{A,q}(z)(mu^{(m-1)})
= 
m!(R_{A,q}(z))^m u(z)
= 
m!(R_{A,q}(z))^m [F-R_{A,q}(z)((q-z)F)].
\end{equation}
Here we have utilized  \eqref{eq_5_Lemma51_u} in the last step.

Next we choose $k \in \N$ and note that since $\varphi_k$ is an eigenfunction of the operator $\mathcal{E}_{g,A,q}$, it follows immediately that
\begin{equation}
\label{eq:formula_for_(q-z)phi_k}
(q-\overline{z})\varphi_k = (\lambda_k - \overline{z})\varphi_k + \Delta_{g,A}\varphi_k.
\end{equation}
On the other hand, since the function $F\in H^1(M)$ satisfies the  problem \eqref{eq:eq_F}, we obtain the following weak version of Green's formula for every $k \in \N$:
\begin{equation}
\label{eq:weak_Green_formula}
0= \LA-\Delta_{g,A}F,\overline{\varphi}_k\RA_{H^{-1}(M),H^1_0(M)} =  (d_A F, d_A \varphi_k)_{L^2(M)} = (F,-\Delta_{g,A}\varphi_k)_{L^2(M)}+(f,\psi_k)_{L^2(\p M)}.
\end{equation}
Hence, the identity \eqref{eq:formula_for_(q-z)phi_k} together with \eqref{eq:weak_Green_formula} yield 
\[
((q-z)F,\varphi_k)_{L^2(M)} = (\lambda_k-z)(F,\varphi_k)_{L^2(M)} + (f,\psi_k)_{L^2(\p M)}.
\]

Let us now use \eqref{eq:resolvent_formula}, as well as the $L^2$-orthogonality of the eigenfunctions $\varphi_k$ and the previous equation, to compute that 
\begin{align*}
(R_{A,q}(z))^{m+1}[(q-z)F]
&= \sum_{k=1}^\infty\frac{((q-z)F,\varphi_k)_{L^2(M)}}{(\lambda_k-z)^{m+1}}\varphi_k
\\
&= 
\sum_{k=1}^\infty\frac{(\lambda_k-z)(F,\varphi_k)_{L^2(M)} + (f,\psi_k)_{L^2(\p M)}}{(\lambda_k-z)^{m+1}}\varphi_k
\\
&= R_{A,q}(z)^m F + \sum_{k=1}^\infty\frac{ (f,\psi_k)_{L^2(\p M)}}{(\lambda_k-z)^{m+1}}\varphi_k.
\end{align*}
Therefore, we get from \eqref{eq_5_lemma51_u^m} that
\[
u^{(m)}(z)= -m!\sum_{k=1}^\infty\frac{( f,\psi_k )_{L^2(\p M)}}{(\lambda_k-z)^{m+1}} \varphi_k.
\]
Consequently, we have proven that
\begin{equation}
\label{eq:series_derivative_DN_map}
\LC\frac{d^m}{dz^m}\Pi_{A,q}(z)\RC f 
=
\frac{d^m}{dz^m}\LA d_Au(z),\nu\RA 
= 
\LA d_Au^{(m)}(z),\nu \RA_g|_{\p M} 
= -m!\sum_{k=1}^\infty\frac{( f,\psi_k )_{L^2(\p M)}}{(\lambda_k-z)^{m+1}} \psi_k.
\end{equation}

Moreover, by the Cauchy-Schwarz inequality, as well as the formulas  \eqref{eq:est_normal_eigenfunction_bdy} and \eqref{eq_wely}, we deduce the following estimate: 
\begin{align*}
\left\|\frac{( f,\psi_k )_{L^2(\p M)}}{(\lambda_k-z)^{m+1}} \psi_k\right\|_{H^{-\frac{1}{2}}(\p M)} 
&\le
\frac{\|f\|_{L^2(\p M)}\|\psi_k\|_{H^{1/2}(\p M)} }{|\lambda_k-z|^{m+1}} \|\psi_k\|_{H^{1/2}(\p M)} 
\\
&\le 
Ck^{-\frac{2(m-1)}{n}}
\quad 
\text{ for every } k \in \N.   
\end{align*}
Hence, when $m >\frac{n}{2}+1$, the series appearing in \eqref{eq:series_derivative_DN_map} converges in $H^{-\frac{1}{2}}(\p M)$. 
This completes the proof of Lemma \ref{lem:DN_map_decomposition}.
\end{proof}

In what follows we aim to express the hyperbolic Dirichlet-to-Neumann map of the  operator $\mathcal{H}_{g,A,q}$, as defined in \eqref{eq:def_operator}, in terms of the elliptic Dirichlet-to-Neumann maps of the problem \eqref{eq:bvp_resolvent} for the operator $\mathcal{E}_{g,A,q+z}$. 
In doing so, we use the notation $\Lambda_{A,q}^\sharp$ for the restriction of the hyperbolic Dirichlet-to-Neumann map $\Lambda_{A,q}$  to the space 
\[
\mathcal{H}_1: = \left\{f\in H^{2n+4}(0,T;H^\frac{1}{2}(\p M)): \p_t^j f(0, \cdot)=0, \: 0\le j\le 2n+3\right\}.
\]
In the next lemma, we state and prove the mapping properties of the restricted hyperbolic Dirichlet-to-Neumann map $\Lambda_{A,q}^\sharp$.
\begin{lem}
\label{lem:boundedness_of_restricted_Lambda}
For each $s\in (-1,-\frac{1}{2})$, the map $\Lambda^\sharp_{A,q}$ defines a bounded linear operator from $\mathcal{H}_1$ to $\mathcal{H}_2: = L^2(0,T;H^s(\p M))$, and in what follows we shall denote its operator norm by $\|\Lambda^\sharp_{A,q}\|_{\mathcal{L}(\mathcal{H}_1,\mathcal{H}_2)}$.
\end{lem}
\begin{proof}
For any function $f\in \mathcal{H}_1$, we let $u:= v+w$ to be the unique solution to the problem \eqref{eq:ibvp_equivalent}, where
$v$ solves the problem 
\[
\begin{cases}
(\p_t^2 - \Delta_{g,A})v=0   & \text{ in } Q,
\\
v(0,\cdot)=0, \quad \p_tv(0,\cdot)=0   &\text{ in } M,
\\
v=f  & \text{ on } \Sigma,
\end{cases}
\]
while $w$ is the solution to the problem
\[
\begin{cases}
\mathcal{H}_{g,A,q}w= -qv   & \text{ in } Q,
\\
w(0,\cdot)=0, \quad \p_tw(0,\cdot)=0   &\text{ in } M,
\\
w=0  & \text{ on } \Sigma,
\end{cases}
\]
As the coefficients $A$ and $q$ are time-independent, we note that $\p_t^2v$ solves the first problem above when $f$ is replaced by $\p_t^2f$. Hence, we apply hyperbolic energy estimates \cite[Lemma 2.42]{KKL_book} to $\p_t^2v$ to obtain
\[
\|\p_t^2v\|_{L^2(0,T;L^2(M))} \le C \|f\|_{\mathcal{H}_1}.
\]
Since $\Delta_{g,A} v = \p_t^2 v$, we get from  elliptic regularity that
\[
\|v\|_{L^2(0,T;H^1(M))} \le 
C \LC \|\p_t^2v\|_{L^2(0,T;H^{-1}(M))}+\|f\|_{\mathcal{H}_1}\RC
\le C \|f\|_{\mathcal{H}_1}.
\]
Therefore, we have that
\[
\|\LA d_Av,\nu \RA_g\|_{\mathcal{H}_2} \le C \|\LA d_Av,\nu \RA_g\|_{L^2(
0,T;H^{-\frac{1}{2}}(\p M))}\le C\|f\|_{\mathcal{H}_1}.
\]

We next shift our attention to the function $w$ and apply \cite[Theorem 2.45]{KKL_book} to obtain 
\[
\|w\|_{L^2(0,T;H^1(M))}\le C \|qv\|_{L^2(0,T;L^2(M))}\le C\|f\|_{\mathcal{H}_1}.
\]
Hence, we get
\[
\|\LA d_Aw,\nu \RA_g\|_{\mathcal{H}_2} \le C \|\LA d_Aw,\nu \RA_g\|_{L^2(
0,T;H^{-\frac{1}{2}}(\p M))}\le C\|f\|_{\mathcal{H}_1}.
\]

Thus, we have proven the  estimate
\[
\|\Lambda_{A,q}^\sharp(f)\|_{\mathcal{H}_2} = \|\LA d_Au,\nu \RA_g\|_{\mathcal{H}_2} \le  \|\LA d_Av,\nu \RA_g\|_{\mathcal{H}_2} + \|\LA d_Aw,\nu \RA_g\|_{\mathcal{H}_2} \le C\|f\|_{\mathcal{H}_1}. 
\]
This completes the proof of Lemma \ref{lem:boundedness_of_restricted_Lambda}.
\end{proof}

In the proof of the next result, we follow the steps presented in  \cite[Lemma 3.1]{Alessandrini_Sylvester}. For the convenience of the readers, we will provide the complete proof.

\begin{lem}
\label{lem_DN_E_H}
For any function $f\in \mathcal{H}_1$, we have that
\begin{equation}
\label{eq:two_DNmap}
\Lambda_{A,q}^\sharp (f)= \sum_{j=0}^{n+1}\left( \frac{1}{j!} \frac{d^j}{d z^j}\Pi_{A,q}(z) \right)\bigg|_{z=0}\left(-\p_t^2\right)^jf + \mathfrak{R}_{A,q}(f),
\end{equation}
where
\begin{equation}\label{eq:def_mathcal_R}
\mathfrak{R}_{A,q}(f)(t,\cdot) = - \sum_{k=1}^\infty  \lambda_k^{-n-2}  \int_0^t s_k(t-s) \left(\left(-\p_s^2\right)^{n+2}f(s), \psi_k\right)_{L^2(\p M)} ds \; \psi_k,
\end{equation}
and  
\begin{equation}
\label{eq:function_s_k}
s_k(t)=\frac{\sin (\sqrt{\lambda_k}t)}{\sqrt{\lambda_k}}, \quad k \in \N.
\end{equation}
\end{lem}

\begin{proof}
For given $f\in \mathcal{H}_1$, $t$, and $z$, we first consider the problem
\begin{equation}
\label{eq:ibvp_u0}
\begin{cases}
(\mathcal{E}_{g,A,q}-z)u^0(t;z) = 0 & \text{ in }  M,
\\
u^0(0;z)=\p_t u^0(0;z)=0 & \text{ in }  M,
\\
u^0(t;z) = f(t) & \text{ on }  \p M.
\end{cases}
\end{equation}
By Remark \ref{rem:positivity_of_q}, we may assume without loss of generality  that zero is not an eigenvalue of $\mathcal{E}_{g, A,q}$.

We first show that  this problem has a solution in $H^{2n+4}(0,T;H^1(M))$ for each $z$ near zero. To achieve this, we have that $f(t)\in H^{1/2}(\p M)$ for each $t$. Thus, 
there exists a function $u^0(t;z)\in H^1(M)$ that solves the elliptic boundary value problem given in the first and the third line of the equation \eqref{eq:ibvp_u0}. Hence, we can define an $H^1(M)$-valued map $t \mapsto u^0(t;z)$ accordingly. Furthermore, since $f\in \mathcal{H}_1$, we get from 
\cite[Theorem 24.1]{Eskin2011}
that $u^0(\cdot;z)\in H^{2n+4}(0,T;H^1(M))$ .
Also, by differentiating the equation \eqref{eq:ibvp_u0} in $t$, we obtain $\p_t^j u^0(0;z)=0$ for all $0\le j\le 2n+3$. Therefore, $u^0(t;z)$ solves \eqref{eq:ibvp_u0} whenever $z$ is close to zero.

Then we proceed recursively to find the functions  $u^j$, $1\le j\le n+1$, that solve the initial boundary value problem 
\begin{equation}
\label{eq:ibvp_uj}
\begin{cases}
(\mathcal{E}_{g,A,q}-z)u^j(t;z) = \left(-\p_t^2\right) u^{j-1}(t;z) & \text{ in }  M,
\\
u^j(0;z)=\p_t u^j(0;z)=0 & \text{ in }  M,
\\
u^j(t;z) = 0  &\text{ on }  \p M.
\end{cases}
\end{equation}
By elliptic regularity and   \cite[Theorem 24.1]{Eskin2011}, we have $u^j\in H^{2n+4-2j}(0,T;H^{2j+1}(M))$. 

Our next goal is to prove that
\begin{equation}
\label{eq:der_of_DN_map_and_u^j}
\LA d_Au^{j}(t;0),\nu\RA_g 
=
\LC\frac{1}{j!}\frac{d^j}{dz^j}\Pi_{A,q}(z)\bigg|_{z=0}\RC \LC-\p_t^2\RC^j f(t),
\quad \text{ for all } j \in \{0,1,\ldots, n+1\}.
\end{equation}
The case $j=0$ follows immediately from \eqref{eq:ibvp_u0} since
\begin{equation}\label{eq:case_j0}
\LA d_Au^{0}(t;z),\nu\RA_g = \Pi_{A,q}(z)f(t).
\end{equation}

When $j\ge 1$, we differentiate the first equation in $\eqref{eq:ibvp_u0}$ with respect to the $t$-variable $2j$ times and  with respect to the $z$-variable $j$ times. This leads us to the problem
\begin{equation}\label{eq:diff_uj}
\begin{cases}
(\mathcal{E}_{g,A,q}-z)\LC-\frac{1}{j}\frac{d^j}{dz^j}\p_t^{2j}u^0(t;z)\RC =-\p_t^2\LC \frac{d^{j-1}}{dz}\p_t^{2(j-1)}u^0(t;z)\RC & \text{ in }  M,
\\
\frac{1}{j}\frac{d^j}{dz^j}\p_t^{2j}u^0(0;z)=\p_t \frac{1}{j}\frac{d^j}{dz^j}\p_t^{2j}u^0 (0;z)=0 & \text{ in }  M,
\\
\frac{1}{j}\frac{d^j}{dz^j}\p_t^{2j}u^0(t;z)  = 0 & \text{ on }  \p M.
\end{cases}
\end{equation}
Comparing \eqref{eq:diff_uj} with \eqref{eq:ibvp_uj} when $j=1$, we get
\[
u^1(t;z) =  - \frac{d}{dz}\p_t^2 u^0(t;z).
\]
Similarly, when $j=2$, we have that
\[
u^2(t;z) =  \frac{1}{2!} \frac{d^2}{dz^2}\p_t^4 u^0(t;z).
\]
Indeed, it follows from induction that
\[
u^j(t;z) = (-1)^j \frac{1}{j!} \frac{d^j}{dz^j}\p_t^{2j} u^0(t;z).
\]
Therefore, the  equation above together with \eqref{eq:case_j0} yield
\[
\LA d_Au^{j}(t;z),\nu\RA_g  = (-1)^j \frac{1}{j!} \frac{d^j}{dz^j}\p_t^{2j} \LA u^0(t,z) ,\nu\RA_g = \LC\frac{1}{j!}\frac{d^j}{dz^j}\Pi_{A,q}(z)\RC \LC-\p_t^2\RC^j f(t),
\]
which proves \eqref{eq:der_of_DN_map_and_u^j} by evaluation at $z = 0$.

Since $\p_t^2 u^{n+1}(\cdot;0) \in L^2(0,T;H^{2n+3}(M))$, it follows from 
\cite[Theorem 2.3]{Lasiecka_Lions_Triggiani}
that there exists a function $h\in C(0,T;H^1(M))\cap C^1(0,T;L^2(M))$ that solves the hyperbolic initial boundary value problem  
\begin{equation}
\label{eq:hyperbolic_eq_for_h}
\begin{cases}
\mathcal{H}_{g,A,q}h =\left(-\p_t^2\right) u^{n+1}(\cdot ;0)   &\hbox{ in } Q,
\\
h(0,\cdot)=\p_th(0,\cdot)=0 &\hbox{ in }  M,
\\
h = 0 &\hbox{ on } \Sigma.
\end{cases}
\end{equation}

Due to the choices of functions $u^0(\cdot;0),\ldots, u^{n+1}(\cdot;0)$, and $h$, by a direct computation, we can express the solution $u$ of the boundary value problem \eqref{eq:ibvp_equivalent}, with boundary value $f$, as 
\[
u= \sum_{j=0}^{n+1} u^j(\cdot;0)+h.
\]
Hence, we have for all $t$ that
\[
\Lambda_{A,q}^\sharp (f)(t)=\langle d_A u(t),\nu \rangle_g
=\sum_{j=0}^{n+1}\langle d_A u^j(t),\nu \rangle_g+\langle d_A h,\nu \rangle_g.
\]

Thanks to \eqref{eq:der_of_DN_map_and_u^j}, we only need to find an expression for the last term to complete the proof. To this end, we apply the definition of the resolvent operator $R_{A,q}(z)$ at $z=0$ to the equation \eqref{eq:ibvp_uj}. Hence, we get by iteration and the formula \eqref{eq:resolvent_formula} that 
\begin{align*}
u^{n+1}(t)
=&\sum_{k=1}^\infty \frac{1}{\lambda_k^{n+1}}\LC \LC-\p_t^2\RC^{n+1} u^0(t),\varphi_k \RC_{L^2(M)} \varphi_k.
\end{align*}
Then we use the equation \eqref{eq:ibvp_u0}, as well as the weak formulation of the Green's formula \eqref{eq:weak_Green_formula}, to obtain that
\begin{align*}
0 &= \LA\mathcal{E}_{g,A,q}u^0,\overline{\varphi}_k\RA_{H^{-1}(M),H^1_0(M)} 
= (u^0,\mathcal{E}_{g,A,q}\varphi_k)_{L^2(M)}+(f,\psi_k)_{L^2(\p M)}.
\end{align*}
Therefore, by combining the two former equations, we see that
\begin{equation}
\label{eq:formula_uj}
\begin{aligned}
u^{n+1} (t)
&= -\sum_{k=1}^\infty \frac{1}{\lambda_k^{n+2}}\LC \LC-\p_t^2\RC^{n+1} f(t), \psi_k \RC_{L^2(\p M)} \varphi_k.
\end{aligned}
\end{equation}

To complete the proof, let us show that
$\LA d_A h(t), \nu \RA_g=\mathfrak{R}_{A,q}(f)(t,\cdot)$, where the right-hand side is given in \eqref{eq:def_mathcal_R}.
To achieve this, we first use the eigenfunction decomposition 
\[
h(t,x) =\sum_{k = 1}^\infty h_k(t) \varphi_k(x). 
\]
Since $h$ solves the equation \eqref{eq:hyperbolic_eq_for_h}, 
we have for each $k \in \N$ that
\[
\p_t^2h_k + \lambda_k h_k =-\LC \p_t^2 u^{n+1} ,\varphi_k\RC_{L^2(M)} = : g_k(t),\quad h_k|_{t=0} = \p_th_k|_{t=0}  = 0.
\]
If we let $s_k$ be defined as in \eqref{eq:function_s_k}, then it is straightforward to check that 
\[
h_k(t) = \int_0^t s_k(t-s) g_k(s)ds.
\]
Hence, we have the following representation:
\[
h = \sum_{k=1}^\infty \LC\int_0^t s_k(t-s)\LC\LC-\p_s^2\RC u^{n+1}(s) ,\varphi_k\RC_{L^2(M)}\,ds\RC \varphi_k.
\]

Finally, we use the expression for $u^{n+1}$ given in the formula \eqref{eq:formula_uj}. Then the $L^2$-orthogonality of the eigenfunctions $\varphi_k$ implies that
\[
h = - \sum_{k=1}^\infty \LC\lambda_k^{-n-2}  \int_0^t s_k(t-s) \left(\left(-\p_s^2\right)^{n+2}f(s), \psi_k\right)_{L^2(\p M)} ds\RC \varphi_k.
\]
As a consequence, we see that $\LA d_A h(t), \nu \RA_g=\mathfrak{R}_{A,q}(f)(t,\cdot)$. This completes the proof of Lemma \ref{lem_DN_E_H}.
\end{proof}

\subsection{Proof of Theorem \ref{thm:spectral_problem}}
\label{subsec:norm_estimates}

Equipped with the lemmas established in Subsection \ref{subsec:relation_bsp_DN_map}, we are now ready to complete the derivation of the estimate \eqref{eq:est_spectral}. Let us first  introduce the notations 
\begin{equation}
\label{eq:op_P_j}
\mathcal{P}^{(j)}(z)  = \frac{d^j}{dz^j}\LC\Pi_{A_1,q_1}(z)-\Pi_{A_2,q_2}(z) \RC, \quad  j\ge 0,
\end{equation}
and 
\[
\mathfrak{R}(f) = \mathfrak{R}_{A_1,q_1}(f) -\mathfrak{R}_{A_2,q_2}(f), \quad f\in \mathcal{H}_1.
\]
In these notations, it follows from Lemma \ref{lem_DN_E_H} that
\begin{equation}
\label{eq:key_identity_difference}
(\Lambda_{A_1,q_1}^\sharp-\Lambda_{A_1,q_2}^\sharp) (f)= \sum_{j=0}^{n+1}\frac{1}{j!}\mathcal{P}^{j}(0
)\left(-\p_t^2\right)^jf + \mathfrak{R}(f),
\quad f\in \mathcal{H}_1.
\end{equation}

We next provide estimates for the sizes of each term in \eqref{eq:key_identity_difference} with respect to $z$ and $\delta$. To achieve this, we shall need the following lemma, which we state and prove below.

\begin{lem}
\label{lemma_A1}
Let $A\in W^{1,\infty}(M,T^*M)$, $q\in L^\infty(M)$ be non-negative, and $z < 0$. 

\begin{itemize}
\item[(a)] If for some function $f \in L^2(M)$, the function $v\in L^2(M)$ is the solution to the nonhomogeneous boundary value problem
\[
\begin{cases}
(\mathcal{E}_{g,A,q}-z) v = f &\text{ in } M,
\\
v = 0 & \text{ on }  \p M,
\end{cases}
\]
then for any $\sigma\in[0,2]$, we have the  estimate
\[
\|v\|_{H^\sigma(M)}\le C |z|^{\frac{\sigma}{2}-1}\|f\|_{L^2(M)}
\]
for some constant $C>0$.

\item[(b)] If for some function $h\in H^{1/2}(\p M)$, the function $w\in L^2(M)$ is the solution to the homogeneous boundary problem
\[
\begin{cases}
(\mathcal{E}_{g,A,q}-z) w = 0, &\hbox{ in } M,
\\
w = h, &\hbox{ on } \p M,
\end{cases}
\]
then there exists a constant $C>0$, which is independent of $h$ and $w$, such that the following estimate holds:
\begin{equation}
\label{eq:est_w}
\|w\|_{L^2(M)}\le C \|h\|_{H^{1/2}(\p M)}.
\end{equation}
\end{itemize}
\end{lem}

\begin{proof}
We first show part (a). To this end, we multiply the equation $\mathcal{E}_{g,A,q}v:=(-\Delta_{g,A}+q-z) v = f$ by $\overline{v}$ and integrate by parts to obtain
\[
\int_M |d_Av|^2+ (q-z)|v|^2 dV_g 
=
\int_M f \overline{v} dV_g.
\]
Since $|d_Av|^2\ge 0$, $q\ge 0$ and $z < 0$, the Cauchy-Schwarz inequality yields 
\begin{align*}
\int_M (-z) |v|^2 dV_g
\le 
\int_M f \overline{v} dV_g
\le
\|f\|_{L^2(M)}\|v\|_{L^2(M)}.
\end{align*}
Hence, it follows that
\[
\|v\|_{L^2(M)}
\le 
|z|^{-1} \|f\|_{L^2(M)}.
\]
Furthermore, by \cite[Theorem 24.1]{Eskin2011}, we get that
\[
\|v\|_{H^2(M)}
\le
|z|\|v\|_{L^2(M)} + \|f\|_{L^2(M)}
\le C\|f\|_{L^2(M)}. 
\]
Finally, the interpolation inequality \cite[Theorem 7.22]{Grubb}, in conjunction with the previous  two inequalities, implies that the following estimate holds for any $\sigma \in [0,2]$:
\begin{align*}
\|v\|_{H^\sigma(M)} 
&\le 
C \|v\|_{L^2(M)}^{1-\frac{\sigma}{2}}  \|v\|_{H^2(M)}^{\frac{\sigma}{2}}
\le 
C |z|^{\frac{\sigma}{2}-1} \|f\|_{L^2(M)}.    
\end{align*}
The derivation of part (a) is now complete.

We next establish part (b). To this end, let us  write the function $w=w_1+w_2$, where $w_1$ satisfies the problem
\begin{equation}
\label{eq:bvp_w1}
\begin{cases}
-\Delta_{g,A} w_1 = 0 & \text{ in }   M,
\\
w_1 = h & \text{ on }  \p M,
\end{cases}
\end{equation}
and $w_2$ is the solution to the problem
\begin{equation}
\label{eq:bvp_w2}
\begin{cases}
(-\Delta_{g,A}+q-z) w_2 =(z -q)w_1 & \text{ in }   M,
\\
w_2= 0, & \text{ on } \p M.
\end{cases}
\end{equation}

First, an application of \cite[Theorem 23.4]{Eskin2011} 
to \eqref{eq:bvp_w1} gives us the following estimate for $w_1$:
\begin{equation}
\label{eq:est_w1}
\|w_1\|_{H^1(M)}
\le C\|h\|_{H^{1/2}(\p M)}.
\end{equation}
To estimate $w_2$, we apply part (a) with $\sigma =0$ to \eqref{eq:bvp_w2} and obtain the inequality
\[
\|w_2\|_{L^2(M)}
\le
C|z|^{-1} \|z-q\|_{L^\infty(M)}\|w_1\|_{L^2(M)}. 
\]
Due to the assumption  $z< -2\|q\|_{L^\infty(M)}$, we have $\|z-q\|_{L^\infty(M)}<\frac{3}{2}|z|$, which implies that
\begin{equation}
\label{eq:est_w2}
\|w_2\|_{L^2(M)}
\le
C\|w_1\|_{L^2(M)}.
\end{equation}
Since $w=w_1+w_2$, the estimate \eqref{eq:est_w}
follows immediately from  \eqref{eq:est_w1} and \eqref{eq:est_w2}. This completes the proof of Lemma \ref{lemma_A1}.
\end{proof}

We next state and prove two important propositions, the first of which provides an estimate for the difference of the elliptic Dirichlet-to-Neumann maps evaluated at $z<0$.
\begin{prop}
\label{prop_DN_E_difference}
Let $-1 < s< -\frac{1}{2}$,  $A_\ell\in W^{1,\infty}(M,T^*M)$, and $0\le q_\ell\in L^\infty(M)$ for $\ell \in \{1,2\}$. Then we have for $z< 0$ that
\begin{equation}
\label{eq:est_Pjz}
    \|\mathcal{P}^{(j)}(z)\|_{\frac{1}{2},s}
\le C 
|z|^{-j+\frac{s}{2}+\frac{1}{4}},
\quad j\ge 0,
\end{equation}
where $C>0$ is a constant depending on $j,N$ and $M$.
\end{prop}

\begin{proof}
For $\ell \in \{1,2\}$, $z<0$, and function $f\in H^{1/2}(\p M)$, we let $u_\ell=u_
\ell(z)\in H^1(M)$ be the unique solution of the Dirichlet problem
\[
\begin{cases}
(\mathcal{E}_{g,A_\ell,q_\ell}-z) u_l = 0 & \hbox{ in } M,
\\
u_\ell = f & \hbox{ on } \p M.
\end{cases}
\]

Let us set $u^{(0)}_\ell:=u_\ell$ and $u_\ell^{(j)}: = \frac{d^j}{dz^j}u_\ell$, $j\ge 1$. It then follows from the same arguments as in the proof of Lemma \ref{lem:DN_map_decomposition} that the function  $u_\ell^{(j)}$, $j\ge 1$, satisfies the boundary value problem
\[
\begin{cases}
(\mathcal{E}_{g,A_\ell,q_\ell} -z) u^{(j)}_\ell = ju^{(j -1)}_\ell & \hbox{ in } M,
\\
u^{(j)}_\ell = 0 & \hbox{ on } \p M.
\end{cases}
\]
Also, we have for $\ell \in \{1,2\}$ that 
\begin{equation}
\label{eq:lemma53_1}
\frac{d^j}{dz^j}\Pi_{A_\ell,q_\ell}(z)f = \LA d_{A_\ell}u_\ell^{(j)},\nu \RA_g|_{\p M},  \quad  j\ge 0.
\end{equation}
Furthermore, an application of Lemma \ref{lemma_A1} gives us the estimates
\begin{equation}
\label{eq:est_ul0}
\|u^{(0)}_\ell\|_{L^2(M)}\le C \|f\|_{H^{1/2}(M)},
\end{equation}
and 
\begin{equation}
\label{eq:est_ulj}
\|u^{(j)}_\ell\|_{H^{\sigma}(M)}\le C |z|^{\frac{\sigma}{2}-1} \|u^{(j-1)}_\ell\|_{L^2(M)}, \quad \text{for any  $j\ge 1$ and $\sigma \in [0,2]$}.
\end{equation}

For $j\geq 0$, we shall denote  $v^{(j)} := u_1^{(j)}-u_2^{(j)}$. Also, let us denote $\widetilde{A}: = A_1 - A_2$, $\widetilde{q}: = q_1 - q_2$. We recall that the operator $d^\ast$ is given in the formula \eqref{eq:dstar}. 
Since $v^{(j)}|_{\p M}=0$ and $A_1|_{\p M} = A_2|_{\p M}$, it follows from \eqref{eq:op_P_j}  and \eqref{eq:lemma53_1} that
\begin{equation}
\label{eq:op_P_j_representation}
\mathcal{P}^{(j)}(z)f 
=\LA d_{A_1}u_1^{(j)},\nu \RA_g|_{\p M}-\LA d_{A_2}u_2^{(j)},\nu \RA_g|_{\p M} =  \p_\nu v^j|_{\p M}, \quad  j\ge 0.
\end{equation}
Moreover, the functions $v^{(j)}, \: j \geq 0$, satisfy the Dirichlet problems
\begin{equation}
\label{eq:equation_for_v_0}
\begin{cases}
(\mathcal{E}_{g,A_1,q_1}-z) v^{(0)} =  2i\LA \widetilde A, du^{(0)}_2\RA_g - \LC i d^*\widetilde A + \LA A_1,A_1\RA_g - \LA A_2,A_2\RA_g +\widetilde q\RC u^{(0)}_2 & \hbox{ in }  M
\\
v^{(0)} = 0 & \hbox{ on } \p M.
\end{cases}
\end{equation}
and
\begin{equation}
\label{eq:equation_for_v_j}
\begin{cases}
(\mathcal{E}_{g,A_1,q_1} -z) v^{(j)} = j v^{(j-1)} + 2i\LA \widetilde A, du^{(j)}_2\RA_g - \LC i d^*\widetilde A + \LA A_1,A_1\RA_g - \LA A_2,A_2\RA_g +\widetilde q\RC u^{(j)}_2 & \hbox{ in }  M
\\
v^{(j)} = 0 & \hbox{ on } \p M.
\end{cases}
\end{equation}

We apply part (a) of Lemma \ref{lemma_A1} to the equations \eqref{eq:equation_for_v_0} and \eqref{eq:equation_for_v_j}, along with elliptic regularity, to get the following estimate  for any $\sigma \in [0,2]$:
\begin{equation}
\label{eq_lemma53_proof_2}
\|v^{(j)}\|_{H^\sigma(M)}
\le C|z|^{\frac{\sigma}{2}-1} \left( \|v^{{(j-1)}}\|_{L^2(M)} +\|u_2^{(j)}\|_{H^1(M)} \right), 
\quad j\ge 0.
\end{equation}
Here we have adopted the notation $v^{(-1)}: = 0$.
Furthermore, by iterating the estimate \eqref{eq:est_ulj} with $\sigma =1$, we obtain 
\begin{equation}
\label{eq:u2j_H1}
\|u_2^{(j)}\|_{H^1(M)}
\le C|z|^{\frac{1}{2}-1}\|u_2^{(j-1)}\|_{L^2(M)}
\le C|z|^{\frac{1}{2}-j}\|u_2^{(0)}\|_{L^2(M)},
\quad  j\ge0,
\end{equation}
where we have applied part (b) of Lemma \ref{lemma_A1} in the last step.
From here, by utilizing the estimates \eqref{eq:est_ul0} and \eqref{eq:u2j_H1}, we deduce from \eqref{eq_lemma53_proof_2}  that 
\begin{align*}
\|v^{(j)}\|_{H^\sigma(M)}
&\le C\left(  |z|^{\frac{\sigma}{2}-j}\|v^{(0)}\|_{L^2(M)} + |z|^{\frac{\sigma}{2}-\frac{1}{2}-j}\|f\|_{H^{1/2}(\p M)} \right)
\\
&\le C|z|^{\frac{\sigma}{2}-\frac{1}{2}-j}\|f\|_{H^{1/2}(\p M)}, \quad  j\ge0. 
\end{align*}

Finally, as $s \in (-1,\frac{1}{2})$, by taking $\sigma = s + \frac{3}{2} \in (\frac{1}{2},1)$, and applying the trace theorem, we conclude that 
\[
\|\p_\nu v^{(j)}\|_{H^s(\p M)} 
\le C\|v^{(j)}\|_{H^{\sigma}(M)}
\le C |z|^{\frac{s}{2}+\frac{1}{4}-j}\|f\|_{H^{1/2}(M)}.
\]
Then the estimate \eqref{eq:est_Pjz} follows immediately from the equation \eqref{eq:op_P_j_representation} and the inequality above. This completes the proof of  Proposition \ref{prop_DN_E_difference}.
\end{proof}

Our second proposition provides an estimate for the difference of the elliptic Dirichlet-to-Neumann maps evaluated at $z=0$, as well as one for the operator $\mathfrak{R}$, which appeared on the right-hand side of \eqref{eq:key_identity_difference}. This estimate connects the aforementioned quantities to the difference of the spectral data $\delta$, which was defined in the formula \eqref{eq:delta}.

\begin{prop}
\label{prop:estimate_P0_R}
Let $-1 < s< -\frac{1}{2}$,  $A_\ell\in W^{1,\infty}(M,T^*M)$, and $0\le q_\ell\in L^\infty(M)$  for $\ell \in \{1,2\}$. Then the following estimate holds for $z<0$:
\begin{equation}
\label{eq:est_Pj(0)}
\|\mathcal{P}^{(j)}(0)\|_{\frac{1}{2},s} \le C\LC |z|^{\frac{s}{2}+\frac{1}{4}-j}+ \delta|z|^{n+1-j} \RC, \quad 0\le j \le n.
\end{equation}
Here the constant $C>0$ depends on $j,N$, and $M$.

We also have the following estimates for any function $f\in \mathcal{H}_1$:
\begin{equation}
\label{eq:estimate_P(n+1)}
\|\mathcal{P}^{(n+1)}(0)\|_{\frac{1}{2},s} \le C\delta,
\end{equation}
and 
\begin{equation}
\label{eq:est_mathfrakR}
\|\mathfrak{R}(f)\|_{L^2(0,T;H^{s}(\p M))}\le C\delta
\end{equation}
Here the constant $C>0$ depends on $N$ and $M$.
\end{prop}

\begin{proof}
We first prove  \eqref{eq:estimate_P(n+1)} by establishing an upper bound for $\|\mathcal{P}^{(n+1)}(z)\|_{\frac{1}{2},s}$ whenever $z\le  0$. 
Since  we have $n+1\geq \frac{n}{2}+2$ for $n\geq 2$,  we can apply Lemma \ref{lem:DN_map_decomposition} to the equation \eqref{eq:op_P_j}. Thus, for any function $f\in H^{1/2}(\p M)$, we have that
\[
\mathcal{P}^{(n+1)}(z)f = -
(n+1)! \sum_{k=1}^\infty \frac{(f,  \psi_{1,k})_{L^2(\p M)}}{(\lambda_{1,k}-z)^{n+2}} \psi_{1,k}
+(n+1)! \sum_{k=1}^\infty \frac{(f,  \psi_{2,k})_{L^2(\p M)}}{(\lambda_{2,k}-z)^{n+2}} \psi_{2,k}.
\]
We split this expression into three terms as follows: 
\begin{align*}
\mathcal{P}^{(n+1)}(z)f
=&
-(n+1)! \sum_{k=1}^\infty(f,\psi_{1,k})_{L^2(\p M)}\psi_{1,k}\LC \frac{1}{(\lambda_{1,k}-z)^{n+2}} - \frac{1}{(\lambda_{2,k}-z)^{n+2}}\RC
\\
& -(n+1)! \sum_{k=1}^\infty \frac{(f,\psi_{1,k}-\psi_{2,k})_{L^2(\p M)}}{(\lambda_{2,k}-z)^{n+2}}\psi_{1,k} 
\\
& -(n+1)! \sum_{k=1}^\infty \frac{(f,\psi_{2,k})_{L^2(\p M)}}{(\lambda_{2,k}-z)^{n+2}}(\psi_{1,k}-\psi_{2,k})
\\
=:& 
I+II+III,
\end{align*}
and estimate each term above individually. 

To estimate term $I$, we use $z\le 0$, in conjunction with Weyl's formula \eqref{eq_wely}, to get that
\begin{align*}
\left| \frac{1}{(\lambda_{1,k} - z)^{n+2}} - \frac{1}{(\lambda_{2,k} - z)^{n+2}} \right|
&\le
C \max \left(\frac{1}{|\lambda_{1,k}^{n+3}|}, \frac{1}{|\lambda_{2,k}^{n+3}|}\right)  \left| \lambda_{1,k} - \lambda_{2,k} \right|
\\
& \le
\frac{C}{{k^{\frac{2(n+3)}{n}}}} \left| \lambda_{1,k} - \lambda_{2,k} \right|.
\end{align*}
On the other hand, it follows from the inequalities  \eqref{eq:est_normal_eigenfunction_bdy} and \eqref{eq_wely} that
\[
\|\psi_{1,k}\|_{H^{-1/2}(\p M)}
\le
\|\psi_{1,k}\|_{L^2(\p M)}
\le
C\lambda_{1,k}\le Ck^{\frac{2}{n}}.
\]
Since $s \in (-1,-\frac{1}{2})$, we get from the two previous estimates that 
\begin{equation}
\label{eq:est_term_1}
\begin{aligned}
\|I\|_{H^s(\p M)}
&\le
C\|f\|_{H^{1/2}(\p M)} \sum_{k = 1}^\infty \|\psi_{1,k}\|_{H^{-1/2}(\p M)} \|\psi_{1,k}\|_{L^2(\p M)} k^{-\frac{2(n+3)}{n}} |\lambda_{1,k}-\lambda_{2,k}|
\\
&\le
C\|f\|_{H^{1/2}(\p M)} \sum_{k = 1}^\infty k^{-\frac{2(n+1)}{n}}|\lambda_{1,k}-\lambda_{2,k}|
\\
&\le C\|f\|_{H^{1/2}(\p M)}\|\lambda_1-\lambda_2\|_{\ell^1_{n,m}(\mathbb{C})}.
\end{aligned}
\end{equation}
Here we have used the fact that $m\le n+1$ and the definition \eqref{eq:weighted_ell_1_norms} for the weighted $\ell^1$-norm in the last step .

For the term $II$, we compute that
\begin{equation}
\label{eq;est_term_2}
\begin{aligned}
\|II\|_{H^s(\p M)}
&\le
C\|f\|_{H^{1/2}(\p M)} \sum_{k = 1}^\infty \frac{1}{(\lambda_{2,k}-z)^{n+2}} \|\psi_{1,k}\|_{H^{-1/2}(\p M)}  \|\psi_{1,k}-\psi_{2,k}\|_{H^{1/2}(\p M)}
\\
&\le
C\|f\|_{H^{1/2}(\p M)} \sum_{k = 1}^\infty k^{-\frac{2(n+1)}{n}} \|\psi_{1,k}-\psi_{2,k}\|_{H^{1/2}(\p M)}
\\
&\le
C\|f\|_{H^{1/2}(\p M)}\|\psi_1-\psi_2\|_{\ell^1_{n,m}(H^{1/2}(\p M))},
\end{aligned}
\end{equation}
where we have applied the fact that $m\le n+1$ again in the last step.

By using the same arguments as above, we obtain for the term $III$ that
\begin{equation}
\label{eq:est_term_3}
\|III\|_{H^s(\p M)} \le C\|f\|_{H^{1/2}(\p M)}\|\psi_1-\psi_2\|_{\ell^1_\omega(H^{1/2}(\p M))}.
\end{equation}
Combining the estimates \eqref{eq:est_term_1}--\eqref{eq:est_term_3}, and using the definition \eqref{eq:delta} for $\delta$, we arrive at the inequality
\begin{equation}
\label{eq:estimate_P(n+1)z}
\|\mathcal{P}^{(n+1)}(z)\|_{\frac{1}{2},s} \le C\delta,
\end{equation}
for any $z\le 0$. Thus, the estimate \eqref{eq:estimate_P(n+1)} follows immediately.

By arguing similarly as above, and utilizing the definition of the operator $\mathfrak{R}$ given in Lemma \ref{lem_DN_E_H}, we conclude that the estimate \eqref{eq:est_mathfrakR} holds for any function $f\in \mathcal{H}_1$.

Finally, to establish the estimate \eqref{eq:est_Pj(0)}, we shall utilize Taylor's formula and expand $\mathcal{P}^{(j)}(0)$ at   $z<0$ as follows:
\[
\mathcal{P}^{(j)}(0)  = \sum_{l=j}^n \frac{\mathcal{P}^{(l)}(z)}{(l-j)!}(-z)^{l-j}+ \int_z^0 \frac{\mathcal{P}^{(n+1)}(\tau)}{(n-j)!}(-\tau)^{n-j}\,d\tau.
\]
Thus, an application of Proposition \ref{prop_DN_E_difference} and the estimate \eqref{eq:estimate_P(n+1)z} yields that
\begin{align*}
\|\mathcal{P}^{(j)}(0)\|_{\frac{1}{2},s} 
&\le C \sum_{l = j}^n \|\mathcal{P}^{(l)}(z)\|_{\frac{1}{2},s}|z|^{l-j} + C \sup_{\tau\in(z,0)} \|\mathcal{P}^{(n+1)}(\tau)\|_{\frac{1}{2},s}|z|^{n+1-j} \\
&\le C \LC |z|^{\frac{s}{2}+\frac{1}{4}-j}+ \delta|z|^{n+1-j} \RC.
\end{align*}
This completes the proof of Proposition \ref{prop:estimate_P0_R}.
\end{proof}

Let us assume for now that  $z\le -1$. Since $-1<s<-\frac{1}{2}$, we get from Proposition \ref{prop:estimate_P0_R}  and the estimate \eqref{eq:key_identity_difference} that 
\begin{equation}
\label{eq:estimate_for_Delta}
\begin{aligned}
\|\Lambda_{A_1,q_1}^\sharp-\Lambda_{A_1,q_2}^\sharp\|_{\mathcal{L}(\mathcal{H}_1;\mathcal{H}_2)}
&\le C \sum_{j=0}^{n+1}\|\mathcal{P}^{j}(0
)\|_{\frac{1}{2},s} + \|\mathfrak{R}\|_{\mathcal{L}(\mathcal{H}_1;\mathcal{H}_2)} 
\\
&\le C\sum_{j=0}^{n} \LC |z|^{\frac{s}{2}+\frac{1}{4}-j}+ \delta|z|^{n+1-j}\RC + C\delta
\\
&\le C\LC |z|^{\frac{s}{2}+\frac{1}{4}}+ \delta|z|^{n+1}\RC.
\end{aligned}
\end{equation}

For the moment we denote $x:=|z|$ and look for  the minimal value of the function $\phi\colon (0,\infty) \to (0,\infty)$ defined by $\phi(x)=x^{\frac{s}{2}+\frac{1}{4}}+\delta x^{n+1}$. By a direct computation, we see that the only critical point of $\phi$ is attained at $a = \LC-\frac{\frac{s}{2}+\frac{1}{4}}{(n+1)\delta}\RC^{\frac{1}{n-s/2+3/4}}$. Moreover, the function $\phi$ is decreasing on $(0,a)$ and increasing on $(a,\infty)$. Since we have assumed that $z\leq -1$, we need to consider separately the cases when $a<1$ and $a>1$. The latter case occurs when $\delta\leq \delta(n,s)$ for $\delta(n,s):=\frac{-(\frac{s}{2}+\frac{1}{4})}{n+1}$; therefore, $\phi$ obtains its minimal value $K(s,n)\delta^\theta$, where $\theta = \frac{-(\frac{s}{2}+\frac{1}{4})}{(n+1)-(\frac{s}{2}+\frac{1}{4})}$ and $K(s,n)$ is a constant depending on $s$ and $n$, at the critical point $x=a$.  
When $a<1$, i.e., $\delta>\delta(n,s)$, the function $\phi$ obtains its minimum value
\[
(1+ \delta)\le (\frac{1}{\delta} \delta+\delta) \le C\delta 
\]
at $x=1$. 

The previous discussion, combined with the estimate \eqref{eq:estimate_for_Delta}, implies that
\begin{equation}
\label{eq:restriction_spectral}
\|\Lambda_{A_1,q_1}^\sharp-\Lambda_{A_1,q_2}^\sharp\|_{\mathcal{L}(\mathcal{H}_1, \mathcal{H}_2)}
\le
\begin{cases}
C\delta^\theta, &  \text{ if } \delta \le \delta(n,s),
\\
C'\delta, &  \text{ if } \delta > \delta(n,s).
\end{cases}
\end{equation}
Hence, the proof of Theorem \ref{thm:spectral_problem} will be complete after we verify the following proposition.

\begin{prop}
\label{prop:stability_DN_map_restriction}
Under the same hypotheses as in Theorem \ref{thm:main_result_hyperbolic}, there exist constants $\sigma_2 \in (0,1)$ and $C>0$, which only depend on $(M,g), n,T$ and $N$, such that the following estimate is valid:
\begin{equation}
\label{eq:est_DN_restriction}
\|A^s_1-A^s_2\|_{L^2(M)}+\|q_1-q_2\|_{L^2(M)}
\le 
C\|\Lambda_{A_1,q_1}^\sharp-\Lambda_{A_2,q_2}^\sharp\|^{\sigma_2}_{\mathcal{L}(\mathcal{H}_1,\mathcal{H}_2)}.
\end{equation}
Here the spaces $\mathcal{H}_1$ and $\mathcal{H}_2$ are as defined in Subsection \ref{subsec:relation_bsp_DN_map}.
\end{prop}

\begin{proof}
The proof mostly follows from the arguments in Section \ref{sec:proof_hyperbolic}, hence we shall only highlight the key differences when considering the restriction of the hyperbolic Dirichlet-to-Neumann map $\Lambda_{A,q}^\sharp$. In what follows we again denote $A=A_1-A_2$ and $q=q_1-q_2$.

We first establish a H\"older-type stability for $A^s$ from $\Lambda_{A,q}^\sharp$. To this end, we observe that the integral identity in this case reads 
\begin{equation}
\label{eq:int_id_restriction}
\begin{aligned}
&\int_Q -2i\n{A, du_2}_g \overline{u_1} dV_gdt
\\
&= 
\int_{\Sigma}
\overline{u_1}(\Lambda_{A_2,q_2}^\sharp-\Lambda_{A_1,q_1}^\sharp)(f)  dS_gdt
-
\int_Q (-id^\ast A+|A_1|^2_g-|A_2|^2_g+q) \overline{u_1} u_2 dV_gdt,
\end{aligned}
\end{equation}
where $u_1$ and $u_2$ are the GO solutions given by \eqref{eq:solution_v} and \eqref{eq:solution_u2}, respectively. Furthermore, due to the construction of $u_2$, the boundary value $f=u_2|_{\Sigma} \in \mathcal{H}_1$. 

We now multiply   \eqref{eq:int_id_restriction} by $h$ and estimate each term appearing in the resulting expression. Let us remark that the only difference between the obtained result and  \eqref{eq:int_id_GO} is the integral over the lateral boundary $\Sigma$, which becomes
\[
\int_{\Sigma} (\Lambda_{A_2,q_2}^\sharp
-\Lambda_{A_1,q_1}^\sharp)(f) \overline{u_1} dS_gdt.
\]

By the Cauchy-Schwarz inequality and the fact that $\Lambda_{A,q}^\sharp:\mathcal{H}_1 \to \mathcal{H}_2$ is a bounded linear operator, as was shown in Lemma \ref{lem:boundedness_of_restricted_Lambda}, we get that
\[
\left| h\int_{\Sigma} (\Lambda_{A_2,q_2}-\Lambda_{A_1,q_1})(f) \overline{u_1} dS_gdt \right|
\le 
Ch \|\Lambda_{A_2,q_2}^\sharp-\Lambda_{A_1,q_1}^\sharp\|_{\mathcal{L}(\mathcal{H}_1, \mathcal{H}_2)} \|f\|_{\mathcal{H}_1} \|u_1\|_{L^2(\Sigma)}.
\]
We now proceed to estimate $\|f\|_{\mathcal{H}_1}$ and $\|u_1\|_{L^2(\Sigma)}$. To achieve this, for any function $v\in H^{2n+4}(0,T;H^1(M))$, we define the norm $\|v\|_{\ast \ast}$ via the formula 
\[
\|v\|_{\ast \ast}: = \|v\|_{H^{2n+4}(0,T;H^1(M))}.
\]
 
We note that each $t$-derivative of the GO solution $u_2$ introduces a factor of $h^{-1}$. Thus, the equation $f=u_2|_\Sigma$ and an application of the trace theorem yield
\begin{align*}
\|f\|_{\mathcal{H}_1} 
&\le
C\|u_2\|_{\ast \ast}
\\
&\le
C \int_0^T \sum_{k=0}^{2n+4} \|\p_t^{(k)} u_2(t,\cdot)\|_{L^2(M)} + 
\|\p_t^{(k)} \nabla u_2(t,\cdot)\|_{L^2(M)} dt
\\
&\le
Ch^{-2n-4}\|\alpha_2\|_{\ast \ast}.
\end{align*}
On the other hand, we argue similarly as in the estimate \eqref{eq:est_u1_boundary} to obtain
\[
\|u_1\|_{L^2(\Sigma)}
\le
Ch^{-1} \|\alpha_1\|_{\ast \ast}.
\]
It then follows from the previous two inequalities that
\begin{equation}
\label{eq:est_bdy_term_restriction}
\left| h\int_{\Sigma} (\Lambda_{A_2,q_2}-\Lambda_{A_1,q_1})(f) \overline{u_1} dS_gdt \right|
\le 
Ch^{-2n-4} \|\Lambda_{A_2,q_2}^\sharp-\Lambda_{A_1,q_1}^\sharp\|_{\mathcal{L}(\mathcal{H}_1, \mathcal{H}_2)} \|\alpha_1\|_{\ast \ast}
\|\alpha_2\|_{\ast \ast}.
\end{equation}
Furthermore, by \eqref{eq:est_bdy_term_restriction}, as well as the same computations as in the estimates \eqref{eq:est_LHS_term1}--\eqref{eq:est_aux_term}, we get  the following estimate  for $0<h<1$: 
\begin{align*}
    \left|\int_Q \n{A,d\psi}_g \left(\overline{\alpha_1}\alpha_2\right) \left(\beta_{A_2}\overline{\beta_{A_1}}\right) dV_gdt\right| 
\le 
C\left(h+h^{-2n-4} \|\Lambda_{A_2,q_2}^\sharp-\Lambda_{A_1,q_1}^\sharp\|_{\mathcal{L}(\mathcal{H}_1, \mathcal{H}_2)} \right)\|\alpha_1\|_{\ast \ast} \|\alpha_2\|_{\ast \ast}.
\end{align*}

Let us now choose the amplitudes $\alpha_i$, $i=1,2$, in the same way as in \eqref{eq:form_alpha_polar}. 
Then it follows from the same arguments as in the proof of Lemma \ref{lem:DN_map_ray_transform} that
\begin{align*}
&\left|\int_{\p_+SM_1} \mathcal{I}_1(A)(y,\theta)\Psi_1(y,\theta) \Psi_2(y,\theta)dS_g(\theta) \right| 
\\
&\le 
C\left(h+h^{-2n-4} \|\Lambda_{A_2,q_2}^\sharp-\Lambda_{A_1,q_1}^\sharp\|_{\mathcal{L}(\mathcal{H}_1, \mathcal{H}_2)}  \right)  \|\Psi_1\|_{H^2(\p_+SM_1)}\|\Psi_2\|_{H^2(\p_+SM_1)}.
\end{align*}   

As the minimum value of the function $h \mapsto h+h^{-2n-4} \|\Lambda_{A_2,q_2}^\sharp-\Lambda_{A_1,q_1}^\sharp\|_{\mathcal{L}(\mathcal{H}_1, \mathcal{H}_2)}$ is a multiple of $\|\Lambda_{A_2,q_2}^\sharp-\Lambda_{A_1,q_1}^\sharp\|_{\mathcal{L}(\mathcal{H}_1, \mathcal{H}_2)}^{\frac{1}{2n+5}}$, we deduce from the previous inequality that
\begin{align*}
&\left|\int_{\p_+SM_1} \mathcal{I}_1(A)(y,\theta)\Psi_1(y,\theta) \Psi_2(y,\theta)dS_g(\theta) \right| 
\\
&\le 
C\|\Lambda_{A_2,q_2}^\sharp-\Lambda_{A_1,q_1}^\sharp\|_{\mathcal{L}(\mathcal{H}_1, \mathcal{H}_2)}^{\frac{1}{2n+5}} \|\Psi_1\|_{H^2(\p_+SM_1)}\|\Psi_2\|_{H^2(\p_+SM_1)}.
\end{align*}
Proceeding as in the end of Subsection \ref{subsec:proof_solenoidal_part_A}, we conclude that there exists a constant $C>0$, which depends on $(M,g), N, n$, and $T$, such that
\begin{equation}
\label{eq:est_As_restriction}
\|A^s\|_{L^2(M)}
\le
C\|\Lambda_{A_2,q_2}^\sharp-\Lambda_{A_1,q_1}^\sharp\|_{\mathcal{L}(\mathcal{H}_1, \mathcal{H}_2)}^{\frac{1}{8n+20}}.
\end{equation}

We now move to derive a H\"older-type estimate for the electric potential $q$. By the Sobolev embedding theorem and the interpolation inequality, we obtain from \eqref{eq:est_As_restriction} that
\begin{equation}
\label{eq:est_As_Linfty_restriction}
\|A^s\|_{L^\infty(M)}
\le
C \|\Lambda_{A_2,q_2}-\Lambda_{A_1,q_1}\|^{\zeta'},
\end{equation}
where $\zeta' =\frac{\kappa}{8n+20}$. Here the constant  $\kappa$ is the same as in Subsection \ref{subsec:proof_electric_hyperbolic}.

Arguing similarly as in the proof of Lemmas \ref{lem:est_int_id_q} and \ref{lem:DN_map_ray_transform_q}, we get that
\begin{align*}
&\left|\int_{\p_+SM_1}  \mathcal{I}_0(q)(y,\theta)\Psi_1(y,\theta) \Psi_2(y,\theta) dS_g(\theta)\right| 
\\
&\le 
C\left(h+h^{-2n-5} \|\Lambda_{A_2,q_2}^\sharp-\Lambda_{A_1,q_1}^\sharp\|^{\zeta'} \right)
\|\Psi_1\|_{H^2(\p_+ SM_1)}
\|\Psi_2\|_{H^2(\p_+ SM_1)}       
\end{align*}
for $0<h<1$. 
From here, it is straightforward to check that the minimum value of the function $h \mapsto h+h^{-2n-5} \|\Lambda_{A_2,q_2}^\sharp-\Lambda_{A_1,q_1}^\sharp\|^{\zeta'}$ is a multiple of $\|\Lambda_{A_2,q_2}^\sharp-\Lambda_{A_1,q_1}^\sharp\|^{\frac{\zeta'}{2n+6}}$. Thus, we get from the   inequality above that
\begin{align*}
& \left| \int_{\p_+SM_1} \mathcal{I}_0(q)(y,\theta)\Psi_1(y,\theta) \Psi_2(y,\theta) dS_g(\theta)\right|
\\
&\le
C\|\Lambda_{A_2,q_2}^\sharp-\Lambda_{A_1,q_1}^\sharp\|^{\frac{\zeta'}{2n+6}}
\|\Psi_1\|_{H^2(\p_+ SM_1)}
\|\Psi_2\|_{H^2(\p_+ SM_1)}.   
\end{align*}
By the same reasoning as the end of Subsection \ref{subsec:proof_electric_hyperbolic}, we have that
\begin{equation}
\label{eq:Holder_electric_restriction}
\|q\|_{L^2(M)}
\le
C\|\Lambda_{A_2,q_2}^\sharp-\Lambda_{A_1,q_1}^\sharp\|^{\frac{\zeta'}{8n+24}}.
\end{equation}

Therefore, we obtain the claimed estimate \eqref{eq:est_DN_restriction} from the inequalities \eqref{eq:est_As_restriction} and \eqref{eq:Holder_electric_restriction}. This completes the proof of Proposition \ref{prop:stability_DN_map_restriction}.
\end{proof}

\bibliographystyle{abbrv}
\bibliography{bib_stability}
\end{document}